\documentclass{amsart}% http://www.ctan.org/pkg/amsart
\usepackage{lipsum}% http://ctan.org/pkg/lipsum
\makeatletter
\g@addto@macro{\endabstract}{\@setabstract}
\newcommand{\authorfootnotes}{\renewcommand\thefootnote{\@fnsymbol\c@footnote}}%
\makeatother

\usepackage{enumitem}
\usepackage[dvipsnames]{xcolor}
\usepackage{setspace}
\usepackage[english]{babel}
\usepackage{amsmath,amsfonts,amssymb,amsthm,epsfig,epstopdf,titling,url,array}
\usepackage[margin=1.0in]{geometry}
\usepackage{tikz-cd}
\usepackage{calc}
\usepackage{color,hyperref}
\usepackage[noabbrev]{cleveref}
\hypersetup{colorlinks}
\hypersetup{colorlinks={true},linkcolor={red},citecolor=green}
\usepackage{graphicx}
\usepackage{mathrsfs}

\graphicspath{{images/}}
\theoremstyle{plain}
\newtheorem{theorem}{Theorem}

\newtheorem{thm}{Theorem}[subsection]
\newtheorem{lem}[thm]{Lemma}
\newtheorem{prop}[thm]{Proposition}
\newtheorem{cor}[thm]{Corollary}

\newcommand\scalemath[2]{\scalebox{#1}{\mbox{\ensuremath{\displaystyle #2}}}}
\theoremstyle{plain}
\newtheorem{defn}[thm]{Definition}

\theoremstyle{remark}
\newtheorem{rem}[thm]{Remark}

\newcommand{\mycomment}[1]{}
\usepackage{tikz}

\makeatletter
\renewcommand{\tocsection}[3]{%
  \indentlabel{\@ifnotempty{#2}{\bfseries\ignorespaces#1 #2\quad}}\bfseries#3}
\renewcommand{\tocsubsection}[3]{%
  \indentlabel{\@ifnotempty{#2}{\ignorespaces#1 #2\quad}}#3}

\renewcommand{\tocsubsubsection}[3]{%
  \indentlabel{\@ifnotempty{#2}{\ignorespaces#1 #2\quad}}#3}

\newcommand\@dotsep{4.5}
\def\@tocline#1#2#3#4#5#6#7{\relax
  \ifnum #1>\c@tocdepth % then omit
  \else
    \par \addpenalty\@secpenalty\addvspace{#2}%
    \begingroup \hyphenpenalty\@M
    \@ifempty{#4}{%
      \@tempdima\csname r@tocindent\number#1\endcsname\relax
    }{%
      \@tempdima#4\relax
    }%
    \parindent\z@ \leftskip#3\relax \advance\leftskip\@tempdima\relax
    \rightskip\@pnumwidth plus1em \parfillskip-\@pnumwidth
    #5\leavevmode\hskip-\@tempdima{#6}\nobreak
    \leaders\hbox{$\m@th\mkern \@dotsep mu\hbox{.}\mkern \@dotsep mu$}\hfill
    \nobreak
    \hbox to\@pnumwidth{\@tocpagenum{\ifnum#1=1\bfseries\fi#7}}\par% <-- \bfseries for \section page
    \nobreak
    \endgroup
  \fi}
\AtBeginDocument{%
\expandafter\renewcommand\csname r@tocindent0\endcsname{0pt}
}
\def\l@subsection{\@tocline{2}{0pt}{2.5pc}{5pc}{}}
\def\l@subsubsection{\@tocline{2}{0pt}{4pc}{5pc}{}}
\makeatother

\usepackage[new]{old-arrows}
\usepackage{bm} 
\usepackage{multicol}
\usetikzlibrary{decorations.pathmorphing}
\DeclareMathOperator{\GL}{GL}

\DeclareMathOperator{\ch}{ch}

\DeclareMathOperator{\Hom}{Hom}
\usepackage{mathtools}

\DeclareMathOperator{\vol}{vol}

\DeclareMathOperator{\PGL}{PGL}

\newcommand\blfootnote[1]{%
  \begingroup
  \renewcommand\thefootnote{}\footnote{#1}%
  \addtocounter{footnote}{-1}%
  \endgroup
}

\date{} % clear date

\begin{document}
\hypersetup{citecolor=blue}
\hypersetup{linkcolor=red}

\begin{center}
  \LARGE 
 ON RANKIN-SELBERG INTEGRAL STRUCTURES AND EULER SYSTEMS FOR $\GL_2\times \GL_2.$
 
  \normalsize
  \bigskip
  Alexandros Groutides \par 
  
 Mathematics Institute, University of Warwick\par
\end{center}

\begin{abstract}
   We study how Rankin-Selberg periods and distinction problems interact with integral structures in spherical Whittaker type representations. Using this representation-theoretic framework, we settle a conjecture of Loeffler by showing that the local Euler factors appearing in the construction of the motivic Rankin-Selberg Euler system for a product of modular forms are integrally optimal; i.e. any construction of this type with any choice of integral input data in the recipe of Loeffler-Skinner-Zerbes, would give local factors appearing in tame norm relations at $p$, which are integrally divisible by the Euler factor $\mathcal{P}_p^{'}(\mathrm{Frob}_p^{-1})$ modulo $p-1$. We also interpret this as an integrality result on the unramified part of the period associated to the Rankin-Selberg convolution of two modular forms.
\end{abstract}
\blfootnote{We gratefully acknowledge support from the following research grant: ERC Grant No. 101001051—Shimura varieties and the Birch–Swinnerton-Dyer conjecture.}
\vspace{-2em}
\section{Introduction} 
\subsection{Overview of the results}
This paper is devoted to the study of Rankin-Selberg periods associated to the convolution of two modular forms. These play an important role in the study of Rankin-Selberg $L$-functions, and the geometry and arithmetic of modular curves; in particular they are fundamental input for the Euler system constructed in \cite{Lei_2014}. We establish integral properties of these periods, when the input data lies in suitable lattices, within the context of Euler systems and distinction problems.

For our first main result, we show that \textit{any} family of integral input data with determinant level in the sense of \cite{Loeffler_2021}, gives rise to cohomology classes which are norm-compatible in the strongest possible integral sense, i.e. as classes in motivic cohomology whose local factors appearing in tame norm relations at $p$ are integral operator multiples of the Euler factor $\mathcal{P}_p^{'}(\mathrm{Frob}_p^{-1})$ modulo $\scalemath{0.9}{(p-1)}$. It is important to note that our results hold for non-trivial coefficient sheaves as well. However, we avoid doing this here to ease exposition. We refer the reader to \cite{Loeffler_2021}, \cite{Loeffler_Skinner_Zerbes_2021} and \cite{grossi2020norm} for the formalisms required to extend this to non-trivial motives. The integral structures with determinant level, arise from global integral considerations under \'etale regulator, and they are precisely the ones which give rise to classes in the integral \'etale cohomology. We refer the reader to \Cref{sec group schemes eq maps} for details and definitions. 

We write $H=\GL_2$ and $G=\GL_2\times \GL_2$ with $H\hookrightarrow G$ diagonally. We adopt the notation $H_p:=H(\mathbf{Q}_p)$ and $H_p^\circ:=H(\mathbf{Z}_p)$, and the same for $G$. 

\begin{theorem}[\Cref{thm euler system norm relations}, \Cref{cor base change}]\label{thm intro A}Let $S$ be a finite set of primes containing $2$, $K_S$ an open compact level subgroup unramified away from $S$, and $c\in\mathbf{Z}_{>1}$ coprime to $6S$. Let $\mathscr{Z}_{cS}$ be the set of square-free positive integers coprime to $cS$. For \textit{any} family $\underline{\delta}=(\delta_p)_{p\notin S}$ of \textit{integral data with determinant level}, the integral Rankin-Eisenstein classes
$$_c\Xi_{\mathrm{mot},n}(\underline{\delta})\in H_\mathrm{mot}^3\left(Y_G(K_S)\times_\mathbf{Q}\mathbf{Q}(\mu_n),\mathbf{Q}(2)\right),\ n\in\mathscr{Z}_S$$
where $_c\Xi_{\mathrm{mot},n}(\underline{\delta})$ depends on $\delta_p$ for $p|n$, satisfy the following norm relations:
\begin{enumerate}
\item For all $n,m\in\mathscr{Z}_{cS}$ with $\tfrac{m}{n}=p$ prime
\begin{align*} 
\mathrm{norm}^{\mathbf{Q}(\mu_m)}_{\mathbf{Q}(\mu_n)}\left(_c\Xi_{\mathrm{mot},m}(\underline{\delta})\right)&=\mathcal{P}_{\mathrm{Tr}(\delta_p)}^\mathrm{cycl}\cdot\ _c\Xi_{\mathrm{mot},n}(\underline{\delta})\end{align*}
with
\begin{align*}
\mathcal{P}_{\mathrm{Tr}(\delta_p)}^\mathrm{cycl}&\in\left(p-1,\mathcal{P}_p^{'}(\mathrm{Frob}_p^{-1})\right)\subseteq\mathcal{H}_{G_p}^\circ(\mathbf{Z}[1/p])[\mathrm{Gal}(\mathbf{Q}(\mu_n)/\mathbf{Q})].
\end{align*}
 The local factor $\mathcal{P}_{\mathrm{Tr}(\delta_p)}^\mathrm{cycl}$ is canonical and $\mathrm{Frob}_p\in \mathrm{Gal}(\mathbf{Q}(\mu_n)/\mathbf{Q})$ is arithmetic Frobenius at $p$. Additionally, $\mathcal{P}_p^{'}(x)\in\mathcal{H}_{G_p}^\circ(\mathbf{Z}[1/p])[x]$ is the polynomial over the integral-away-from-$p$ spherical Hecke algebra which interpolates local $L$-factors of the unramified principal-series of $G_p.$
 \item Specialising to $\underline{\delta}=\underline{\delta}_1$, we have for all $n,m\in\mathscr{Z}_{cS}$ with $\tfrac{m}{n}=p$ prime 
\begin{align*}
    \mathrm{norm}^{\mathbf{Q}(\mu_m)}_{\mathbf{Q}(\mu_n)}(_c\Xi_{\mathrm{mot},m}(\underline{\delta}_1))&=
        \mathcal{P}_p^{'}(\mathrm{Frob}_p^{-1})\cdot\ _c\Xi_{\mathrm{mot},n}\left(\underline{\delta}_1\right).
\end{align*} 
\end{enumerate}
\end{theorem}
  \noindent 
  We refer the reader to \Cref{def integral lattice} and \Cref{sec from and back} for more details on local integral data, and to \Cref{thm euler system norm relations} for the definition of the family $\underline{\delta}_1$. For the sake of completeness, we also mention that the other generator of the ideal in part $(1)$, i.e. $p-1$, is also trivially obtained. Additionally, part $(2)$ of the theorem is a special instance of part $(1)$ and it is the corresponding version of Theorem A in \cite{Loeffler_2021}, in our setup. It combines the best of both worlds between \cite{Lei_2014} and \cite{grossi2020norm}; i.e. Euler system tame norm relations in motivic cohomology and with the correct Euler factors.

 The local theory developed to show that the tame norm relations are integrally optimal, contains several novel stand-alone integrality results regarding distinction problems and zeta integrals. This brings us to our second main result.
 
 Let $f_1,f_2$ be normalized cuspidal Hecke eigenforms of even integral weights $k_1,k_2$, levels $N_1,N_2$ and Nebentypes $\epsilon_1,\epsilon_2$. Write $\pi_{f_i}$ for the corresponding cuspidal automorphic representation attached to $f_i$. Let $S$ be a finite set of places containing $\infty$ and all the prime divisors of $N_1N_2$. Then, $\pi_{f_1}^S\boxtimes\pi_{f_2}^S$ is unramified and contains a unique normalized spherical vector $W_{\pi_{f_1}}^\mathrm{sph}\otimes W_{\pi_{f_2}}^\mathrm{sph}$. The unramified Rankin-Selberg period associated to $f_1\times f_2$ is the unique normalized linear form
 $$\mathcal{Z}_{f_1\times f_2}\in\mathrm{Hom}_{\GL_2(\mathbf{A}^S)}\left(\mathcal{S}((\mathbf{A}^S)^2)\otimes \pi_{f_1}^S\otimes\pi_{f_2}^S,\mathbf{1}\right)$$
 in this one-dimensional Hom-space, where $\mathcal{S}((\mathbf{A}^S)^2$ is a space of Schwartz functions. It can be realized using the local zeta integrals and $L$-factors of \cite{jacquet1981euler}. Our second main result is concerned with the integral behavior of this period.

\begin{theorem}[\Cref{thm integral JPSS periods}, \Cref{thm global rankin selberg}]\label{thm intro B}
     Let $\ell\nmid N_1N_2$ be a fixed prime and fix $\mathbf{C}\simeq_\iota\overline{\mathbf{Q}}_\ell$. Let $\mathbf{L}_{f_1,f_2}$ be the smallest $\ell$-adic number field containing the composite of the number fields of $f_1$ and $f_2$ under $\iota$. Then the following are true: 
    \begin{enumerate}
        \item For any $g_1,g_2\in\GL_2(\mathbf{A}^S)$ and any decomposable Schwartz function $\Phi=\otimes_{p\not\in S}\Phi_p\in \mathcal{S}((\mathbf{A}^S)^2)$ where each $\Phi_p$ is valued in $\vol_{H_p}(\mathrm{Stab}_{H_p}(\Phi_p)\cap g_pG_p^\circ g_p^{-1})^{-1}\mathcal{O}_{\mathbf{L}_{f_1,f_2}}$, the period $\mathcal{Z}_{f_1\times f_2}$ satisfies $$\mathcal{Z}_{f_1\times f_2}\left(\Phi,g_1W_{\pi_{f_1}^S}^\mathrm{sph},g_2W_{\pi_{f_2}^S}^\mathrm{sph}\right)\in\mathcal{O}_{\mathbf{L}_{f_1,f_2}}.$$ 
        \item Suppose, moreover, that $S_0$ denotes a finite set of primes disjoint from $S$, and for each $p\in S_0$, $\Phi_p$ is valued in $\vol_{H_p}(\mathrm{Stab}_{H_p}(\Phi_p)\cap g_pG_p^\circ[p]g_p^{-1})^{-1}\mathcal{O}_{\mathbf{L}_{f_1,f_2}}$ \emph{(note the addition of the determinant level $G_p^\circ[p]$ instead of $G_p^\circ$)}. Then, the period $\mathcal{Z}_{f_1\times f_2}$ also satisfies
        $$ \mathcal{Z}_{f_1\times f_2}\left(\Phi,g_1W_{\pi_{f_1}^S}^\mathrm{sph},g_2W_{\pi_{f_2}^S}^\mathrm{sph}\right)\cdot \left(\prod_{\substack{p\in S_0\ \mathrm{s.t:}\\ \Phi_p(0,0)\neq 0}} L_p(\epsilon_1\epsilon_2,0)^{-1}\right)\in \prod_{\substack{p\in S_{0}}}\left(p-1,L_p(\pi_{f_1}\boxtimes \pi_{f_2},0)^{-1}\right)\subseteq\mathcal{O}_{\mathbf{L}_{f_1,f_2}}.$$
        In particular, if the levels of $f_1,f_2$ satisfy $(\#(\mathbf{Z}/N_1N_2\mathbf{Z})^\times,\ell)=1$, then the containment holds without the bracketed product of abelian $L$-factors.
    \end{enumerate}
    
\end{theorem}
\noindent 
Here $\vol_{H_p}$ denotes the unramified Haar measure on $H_p$, and $G_p^\circ[p]\subseteq G_p^\circ$ is the open compact subgroup given in \eqref{eq: det levl}.

It is important to note that \Cref{thm intro B} (and more generally \Cref{sec in lattice in S x pi1 x pi2}, \Cref{sec integral lattice in pi1 x pi2 x pi3}, \Cref{rankin selber integrals}) treats \textit{arbitrary} integral test vectors as per \Cref{def integral lattice}. There are related results in the literature, mostly resembling \Cref{sec in lattice in S x pi1 x pi2}, but only for \textit{specific} nice integral input data (see for example \cite[\S $4$]{grossi2020norm}). Overall, it does not seem feasible to directly study these periods for arbitrary integral test vectors and expect these general integrality results. This is mostly due to the lack of compatibility between the additive structure of the local space of Schwartz functions, and the normalisation volume factors of integral data arising from global considerations. The way we work around this issue is by translating our problem to an integral problem in $\mathrm{Ind}_{P_p}^{G_p}\mathbf{1}$ where $P$ is the mirabolic subgroup of $\GL_2$. This is also crucial in proving \Cref{thm intro A} since in order to realize the local factors for general integral input data and study them, one essentially needs to closely examine how local Rankin-Selberg periods interact with integral structures in unramified principal-series representations of $\GL_2\times \GL_2$. 

\subsection{Outline of the paper}
After setting up some preliminaries in \Cref{sec prelims}; \Cref{section structure theorems} and \Cref{sec Hecke forms} are mainly dedicated to reproducing the original local construction of \cite{Loeffler_2021} and then ``embedding it'' into a different setup. This transition captures and realizes integrality in a more natural and approachable manner. It also bypasses the problem mentioned above regarding the local Rankin-Selberg zeta integral evaluated on arbitrary integral elements. There are a few main ingredients to this: Firstly, we verify that a general cyclicity result of \cite{Sakellaridis_2008} can be applied to our setup, and proceed to link the two as in \cite{Loeffler_2021}. Then, we construct and study a Hecke equivariant morphism to a certain space of compactly supported functions that are invariant under left translations by the diagonal mirabolic subgroup $P\subseteq\GL_2$. Afterwards, using a splitting of measures inspired by \cite{kurinczuk2017rankin}, together with certain local zeta integrals, we construct non-zero Hecke equivariant forms from this space to $(\Pi_p^\vee)^{G_p^\circ}$ for a sufficiently dense family of Whittaker type principal-series $\Pi_p$. These are easier to work with at an integral level and with general input data. Finally, we prove a main integrality result (\Cref{thm conj 1 for P-invariant}) within this $P$-invariant setting. We also establish a (integral) structure result for the aforementioned space which even though is not crucial for norm relations it offers a more complete version of the theory and it links our work with integral structures in unramified Whittaker type $G$-representations distinguished by the diagonal mirabolic, and with branching laws for $\ell$-modular representations (\Cref{sec relation to l-modular}).

In \Cref{sec local norm relations}, by ``increasing the level'' in our previous constructions and using certain trace maps, we transfer our integral results back to the original Rankin-Selberg setting. This is how we obtain the \textit{local abstract integral analogue} of \Cref{thm intro A} (\Cref{thm local norm relations}). The rest of \Cref{sec euler systems for gl2 x gl2} is dedicated to introducing the setup of \Cref{thm intro A} and defining the Rankin-Eisenstein map $\mathcal{RE}$, following \cite{Loeffler_2021}. The next step is to apply our integral local result to this sophisticated example. Doing this, we obtain the integrally optimal norm-compatibility of our classes in motivic cohomology attached to general integral data. 

In the final section (\Cref{sec int periods}), we link our local Hecke theoretic results to the integral behavior of the unique Rankin-Selberg for a product of modular forms, which culminates in \Cref{thm intro B}.
\subsection{Acknowledgments}
I am vastly grateful to my PhD supervisor David Loeffler, without whom this paper would not have been possible. In particular, I thank him for his constant support and invaluable mathematical input while this work was taking place. Additionally, I would like to thank Ju-Feng Wu, Robert Kurinczuk,  Dipendra Prasad, Yiannis Sakellaridis and Johannes Girsch for useful discussions. 
 
{
  \hypersetup{linkcolor=black}
  \tableofcontents
}\section{Preliminaries}\label{sec prelims}
\subsection{Group schemes and Hecke algebras}
Throughout, $p$ will always denote a prime number. We fix once and for all the algebraic groups $$G:=\GL_2\times \GL_2,\  H:=\GL_2.$$
We identify $H$ with a subgroup of $G$ via the diagonal embedding $h\mapsto (h,h)$. 
  For a group $\Gamma$ in $\{H,G\}$, we write $\Gamma_p$ for the locally pro-finite group $\Gamma(\mathbf{Q}_p)$ and $\Gamma_p^\circ$ for the maximal open compact subgroup  $\Gamma(\mathbf{Z}_p)$. We fix a $\mathbf{Q}$-valued Haar measure $d\gamma$ on $\Gamma_p$, normalized to give $\Gamma_p^\circ$ volume one. We write $\mathcal{H}(\Gamma_p):=C_c^\infty(\Gamma_p)$ for the full Hecke algebra of smooth, compactly supported $\mathbf{C}$-valued functions on $\Gamma_p$, under the usual convolution, which we denote by $$(\theta_1\cdot_\Gamma \theta_2)(x):=\int_{\Gamma_p} \theta_1(\gamma)\theta_2(\gamma^{-1}x)\ d\gamma, \ x\in\Gamma_p, \theta_1,\theta_2\in\mathcal{H}(\Gamma_p).$$ The subscript $\Gamma$ will be used when it's necessary to distinguish between different subgroups, but when it is clear from context, it will be omitted. We write $\mathcal{H}_{\Gamma_p}^\circ:=C_c^\infty(\Gamma_p^\circ\backslash \Gamma_p/\Gamma_p^\circ)$ for the spherical (or unramified) Hecke algebra of (smooth), compactly supported, $\Gamma_p^\circ$-bi-invariant, $\mathbf{C}$-valued functions on $\Gamma_p$, under the same convolution as before, and we also write $\mathcal{H}_{\Gamma_p}^\circ(\mathbf{Z}[1/p])$ for the subalgebra of functions valued in $\mathbf{Z}[1/p]$. It is well known that $\mathcal{H}_{\Gamma_p}^\circ$ is a commutative $\mathbf{C}$-algebra of Weyl group invariant polynomials in the Satake parameters. Explicitly, we have:
  \begin{align*}\mathcal{H}_{H_p}^\circ\simeq\mathbf{C}[\mathcal{S}_p^{\pm 1},\mathcal{T}_p]&\longhookrightarrow\mathcal{H}_{G_p}^\circ\simeq\mathbf{C}[\mathcal{S}_{p,1}^{\pm1},\mathcal{T}_{p,1},\mathcal{S}_{p,2}^{\pm 1},\mathcal{T}_{p,2}]\\
\mathcal{H}_{H_p}^\circ(\mathbf{Z}[1/p])\simeq\mathbf{Z}[1/p][\mathcal{S}_p^{\pm 1},\mathcal{T}_p]&\longhookrightarrow\mathcal{H}_{G_p}^\circ(\mathbf{Z}[1/p])\simeq\mathbf{Z}[1/p][\mathcal{S}_{p,1}^{\pm1},\mathcal{T}_{p,1},\mathcal{S}_{p,2}^{\pm 1},\mathcal{T}_{p,2}]\\
\mathcal{S}_p& \longmapsto \mathcal{S}_{p,1}\mathcal{S}_{p,2}\\
\mathcal{T}_{p}&\longmapsto\mathcal{T}_{p,1}\mathcal{T}_{p,2}\end{align*}
where
\begin{equation*}
\begin{split}
\mathcal{S}_p&:=\ch\left(\left[\begin{smallmatrix}
    p & \\
    & p
\end{smallmatrix}\right]H_p^\circ\right)\\
\mathcal{S}_{p,1}&:=\ch\left(\left(\left[\begin{smallmatrix}
    p & \\
    & p
\end{smallmatrix}\right],1\right)G_p^\circ\right)\\
\mathcal{S}_{p,2}&:=\ch\left(\left(1,\left[\begin{smallmatrix}
    p & \\
    & p
\end{smallmatrix}\right]\right)G_p^\circ\right)
\end{split}\ \ \ \ 
\begin{split}
\mathcal{T}_p&:=\ch\left(H_p^\circ\left[\begin{smallmatrix}
    p & \\
    & 1
\end{smallmatrix}\right] H_p^\circ\right)\\
\mathcal{T}_{p,1}&:= \ch\left(G_p^\circ\left(\left[\begin{smallmatrix}
    p & \\
    & 1
\end{smallmatrix}\right],1\right)G_p^\circ\right)\\
\mathcal{T}_{p,1}&:=\ch\left(G_p^\circ\left(1,\left[\begin{smallmatrix}
    p & \\
    & 1
\end{smallmatrix}\right]\right)G_p^\circ\right).
\end{split}
\end{equation*}
  We will denote by $(-)^{'}$ the automorphism of $\mathcal{H}_{\Gamma_p}^\circ$ induced by inversion in $\Gamma_p$, i.e. for any $\theta\in\mathcal{H}_{\Gamma_p}^\circ$, we have $\theta^{'}(\gamma):=\theta(\gamma^{-1})$ for all $\gamma\in \Gamma_p$. 
For an open subgroup $U_p$ of $\Gamma_p$, we write $\ch(U_p)$ for the characteristic function of $U_p$. All representations considered will be $\mathbf{C}$-linear (or $\overline{\mathbf{Q}}_\ell$-linear after fixing $\mathbf{C}\simeq\overline{\mathbf{Q}}_\ell$) unless otherwise stated (e.g \Cref{sec relation to l-modular}), and every smooth $\Gamma_p$-representation $V$ will also be regarded as a $\mathcal{H}(\Gamma_p)$-module in the usual way, via $$\theta\cdot v:=\int_{\Gamma_p}\theta(\gamma)\ \gamma\cdot v\ d\gamma,\ \theta\in\mathcal{H}(\Gamma_p), v\in V.$$ 
In particular, if $v\in V$ is $U_p$-invariant, then for any $g\in\Gamma_p$, $\ch(g U_p)\cdot v=\vol_{\Gamma_p}(U_p,d\gamma) g\cdot v$.
\subsection{Whittaker models, equivariant maps and volume factors}\label{sec group schemes eq maps}
We write ``$\mathrm{Ind}$'' (resp. ``$c\text{-}\mathrm{Ind}$'') for unnormalized (resp. compact unnormalized) induction, and ``$\mathrm{ind}$'' for the usual normalized induction.
We will usually denote by $B_p$ the upper triangular Borel of $H_p$, $N_p$ its unipotent radical of upper unitriangular matrices, and by $T_p$ the diagonal torus. Then $B_p=T_p\ltimes N_p$ and we also have the well-known Iwasawa decomposition $H_p=B_pH_p^\circ$.
 Let $\psi:\mathbf{Q}_p\rightarrow \mathbf{C}^\times$ be an additive character which we regard as a character of $N_p$ in the usual way by identifying $N_p$ with $\mathbf{Q}_p$ through its top right entry. As in \cite{loeffler2021gross}, we say that a smooth $H_p$-representation $\pi_p$ is of \textit{Whittaker type} if it is either irreducible and generic or a twist of the reducible principle series, which contains the Steinberg representation as a co-dimension one submodule. By a classic yet deep theorem of Gel'fand-Kazhdan (in a more general setting, \cite{gel1972representation}) such a representation is isomorphic to a space of  $\mathbf{C}$-valued functions on $H_p$ transforming by $\psi$ under left translations by $N_p$, i.e. $\pi_p\hookrightarrow \mathrm{Ind}_{N_p}^{H_p}\psi$. This space is unique and we call it the \textit{Whittaker model} of $\pi_p$, denoted by $\mathcal{W}(\pi_p,\psi)$. We will thus identify such a representation with its Whittaker model. 
Additionally, we say that such a representation is \textit{unramified} if it has a non-zero $H_p^\circ$-fixed vector. It is well-known that for such representation the space $(\pi_p)^{H_p^\circ}$ is one-dimensional and is generated by a canonical spherical vector $W_{\pi_p}^\mathrm{sph}$ such that $W_{\pi_p}^\mathrm{sph}(1)=1$. Thus, the spherical Hecke algebra of $H_p$ acts on this space via a character which we call the \textit{spherical Hecke eigensystem} of $\pi_p$ and denote it by $\Theta_{\pi_p}$.

We write $\mathcal{S}(\mathbf{Q}_p^2)$ for the space of $\mathbf{C}$-valued Schwartz functions (i.e locally constant and compactly supported) on $\mathbf{Q}_p^2$, regarded as a smooth $H_p$-representation under right translations. We also write $\mathcal{S}_0(\mathbf{Q}_p^2)$ for the subrepresentation of $\mathcal{S}(\mathbf{Q}_p^2)$ consisting of functions that vanish at $(0,0)$.

\begin{defn}\emph{(\cite[Definition $3.1.1$]{Loeffler_2021})}\label{def equiv maps}
    Let $V$ be a smooth $G_p$-representation. A linear map 
    $$\mathfrak{Z}:\mathcal{S}(\mathbf{Q}_p^2)\otimes \mathcal{H}(G_p)\longrightarrow{V}$$
    is $(G_p\times H_p)$-equivariant, if it is equivariant with respect to the following actions: 
\begin{enumerate}
    \item $G_p$ acts on the left hand side by $g\cdot(\phi\otimes\xi):=\phi\otimes(\xi((-)g)$ and on the right-hand side by its assigned action on $V$.
    \item $H_p$ acts on the left-hand side by $h\cdot(\phi\otimes\xi):=\phi((-)h)\otimes\xi(h^{-1}(-))$ and trivially on the right-hand side
\end{enumerate}
    
\end{defn}

\noindent 
The family of such equivariant maps can be canonically identified with $\Hom_{H_p}(\mathcal{S}(\mathbf{Q}_p^2)\otimes V^\vee,\mathbf{1})$, where $V^\vee$ denotes the smooth dual of $V$ \cite[\S $4.9$]{Loeffler_Skinner_Zerbes_2021}. The latter is often non-zero and very interesting to study. For an open subgroup $U_p\subseteq G_p^\circ$, we write $\mathcal{I}(G_p/U_p)$, for the $H_p$-coinvariants of $\mathcal{S}(\mathbf{Q}_p^2)\otimes C_c^\infty(G_p/U_p)$, and $[\cdot]$ for the class of an element in this space of coinvariants. We will eventually drop the $[\cdot]$ notation once it's no longer necessary. It follows from the definitions that for every such pair $(V,\mathfrak{Z})$, we get an induced map $\mathcal{I}(G_p/U_p)\rightarrow V^{U_p}$. We equip $\mathcal{S}(\mathbf{Q}_p^2)\otimes C_c^\infty(G_p/G_p^\circ)$ with the structure of an $\mathcal{H}_{G_p}^\circ$-module via the $G_p$-action defined above. It can be easily checked that this induces a well-defined action of $\mathcal{H}_{G_p}^\circ$ on $\mathcal{I}(G_p/G_p^\circ)$. A formal unravelling of definitions, gives the following two identities as in \cite{Loeffler_2021}: \begin{itemize}
    \item Let $\phi\in \mathcal{S}(\mathbf{Q}_p^2)$ and $\xi_1,\xi_2\in \mathcal{H}(G_p)$. Then
    $[\phi\otimes (\xi_1\cdot_G\xi_2^{'})]=\xi_2\cdot[\phi\otimes \xi_1].$
    \item Let $\phi\in \mathcal{S}(\mathbf{Q}_p^2)$, $\xi_1\in \mathcal{H}(H_p)$ and $\xi_2\in\mathcal{H}(G_p)$. Then
    $[(\xi_1\cdot \phi)\otimes \xi_2]=[\phi\otimes (\xi_1^{'}\cdot_H \xi_2)].$
\end{itemize}
These two identities will be used throughout without further mention. 

\begin{defn}\label{def integral lattice} We introduce the following integral lattices in $\mathcal{I}(G_p/G_p^\circ)$ as in \emph{\cite[Definition $3.2.1$]{Loeffler_2021}}. Let $U_p$ be an open subgroup of $G_p^\circ$. 
   The \textit{integral} lattice \emph{(}at level $U_p$\emph{)} $\mathcal{I}_{(0)}(G_p/U_p,\mathbf{Z}[1/p])$ consists of all $\mathbf{Z}[1/p]$-linear combinations of elements of the form $[\phi\otimes \ch(gU_p)]$, where $\phi\in\mathcal{S}_{(0)}(\mathbf{Q}_p^2)$, $g\in G_p$, and $\phi$ is valued in 
   $$\frac{1}{\vol_{H_p}\left(\mathrm{Stab}_{H_p}(\phi)\cap g U_p g^{-1}\right)}\mathbf{Z}[1/p].$$
   Here $\mathcal{S}_{(0)}$ is shorthand notation to mean either $\mathcal{S}$ or $\mathcal{S}_0$ and $\mathcal{I}_{(0)}$ is then defined accordingly. Additionally, $\vol_{H_p}$ denotes the normalized $\mathbf{Q}$-valued Haar measure which gives $H_p^\circ$ volume $1$.
\end{defn}
\begin{rem}
    At first glance this definition seems a bit odd. However from a global viewpoint as pointed out in \cite{Loeffler_2021}, this is precisely the data which gives integral classes in the \'etale cohomology of our Shimura variety (see \Cref{sec etale realisation}).  As we will see later on, these elements do not only possess global integral properties, but also deep \textit{local} integral properties (see \Cref{prop integrally from coinvariants to C^infinity} and \Cref{thm local norm relations}).
\end{rem}
\noindent Of particular interest to us, will be the lattices $\mathcal{I}_{(0)}(G_p/G_p^\circ,\mathbf{Z}[1/p])$ and $\mathcal{I}_{(0)}(G_p/G_p^\circ[p],\mathbf{Z}[1/p])$, where $G_p^\circ[p]$ is the open level subgroup of $G_p^\circ$ given by 
\begin{align}\label{eq: det levl}G_p^\circ[p]:=\{(g_1,g_2)\in G_p^\circ\ |\ \det(g_2)\equiv 1\mod p\}.\end{align}
\begin{rem}
    One might think that there is a certain asymmetry in this definition in the sense that we've imposed the determinant condition on the second factor. However, this choice is irrelevant. The results are identical if one imposes the condition on the first factor instead. In fact, the existence of this choice can be eliminated all together if one works with the larger group $G\times \GL_1$ and then imposes this determinant condition on the $\GL_1$ factor. We do not do that here to avoid unnecessarily complicating the notation everywhere else.
\end{rem}
\noindent The inclusion $G_p^\circ[p]\subseteq G_p^\circ$ induces trace maps which descend to an integral level \cite[Proposition $3.2.3$]{Loeffler_2021}:
\[\begin{tikzcd}[sep=scriptsize]
	{\mathcal{I}_{(0)}(G_p/G_p^\circ[p])} && {\mathcal{I}_{(0)}(G_p/G_p^\circ)} & {} \\
	{\mathcal{I}_{(0)}(G_p/G_p^\circ[p],\mathbf{Z}[1/p])} && {\mathcal{I}_{(0)}(G_p/G_p^\circ,\mathbf{Z}[1/p])}
	\arrow["{\mathrm{Tr}^{G_p^\circ[p]}_{G_p^\circ}}", from=1-1, to=1-3]
	\arrow[hook', from=2-1, to=1-1]
	\arrow["{\mathrm{Tr}^{G_p^\circ[p]}_{G_p^\circ}}", from=2-1, to=2-3]
	\arrow[hook, from=2-3, to=1-3]
\end{tikzcd}\]
where $\mathrm{Tr}(\phi\otimes \xi):=\mathrm{Tr}^{G_p^\circ[p]}_{G_p^\circ}(\phi\otimes \xi):=\phi\otimes\sum_{\gamma\in G_p^\circ/G_p^\circ[p]}\xi((-)\gamma)$.

\section{Structure theorems}\label{section structure theorems}
\subsection{Connection to spherical varieties and a result of Sakellaridis} 
We introduce the space of $\mathbf{C}$-valued functions $C_c^\infty(H_p^\circ\backslash G_p/ G_p^\circ)$, which we regard as a module over $\mathcal{H}_{G_p}^\circ\otimes \mathcal{H}_{H_p}^\circ$ via
$$\left((\theta_1\otimes \theta_2)\cdot\xi\right)(x):=\int_{G_p}\int_{H_p} \theta_1(g)\theta_2(h)\xi(h^{-1}xg)\ dg\ dh, \ \theta_1\otimes \theta_2\in \mathcal{H}_{G_p}^\circ\otimes \mathcal{H}_{H_p}^\circ,\ \xi\in C_c^\infty(H_p^\circ\backslash G_p/ G_p^\circ).$$
The goal of this section is to prove that the module $C_c^\infty(H_p^\circ\backslash G_p/G_p^\circ)$ is cyclic over $\mathcal{H}_{G_p}^\circ \otimes \mathcal{H}_{H_p}^\circ$, and generated by the characteristic function of $G_p^\circ$. To do this, we heavily rely on an application of a general theorem of Sakellaridis \cite{sakellaridis2013spherical} regarding function spaces on spherical varieties. Sakellaridis' result \cite[Corollary $8.0.4$]{sakellaridis2013spherical}  shows that for any split reductive group $\mathscr{G}$ over $\mathbf{Z}_p$, and any homogeneous affine spherical $\mathscr{G}$-variety $\mathcal{X}$ that satisfies a long list of conditions, the module $C_c^\infty(\mathcal{X}(\mathbf{Q}_p))^{\mathscr{G}_p^\circ}$, of $\mathscr{G}_p^\circ$-invariant Schwartz functions on $\mathcal{X}(\mathbf{Q}_p)$, is cyclic over $\mathcal{H}_{\mathscr{G}_p}^\circ$, generated by the characteristic function of $\mathcal{X}(\mathbf{Z}_p)$. In fact, Sakellaridis' results also give us the annihilator of this action \cite[Theorem $8.0.2$]{sakellaridis2013spherical}, but we won't be concerned with that here.
\begin{thm}\label{cyclicity via sakellaridis}
    The module $C_c^\infty(H_p^\circ\backslash G_p/G_p^\circ)$ is cyclic over $\mathcal{H}_{G_p}^\circ\otimes \mathcal{H}_{H_p}^\circ$, generated by the characteristic function $\ch(G_p^\circ)$.
    \begin{proof}
        In order to be able to apply Sakellaridis' result, we need some preparation. We write $\mathscr{G}$ for the group scheme $G\times H=\GL_2^3$, and $\mathscr{H}$ for the diagonal copy of $H$ in $\mathscr{G}$. We define three different Borel subgroups of $\GL_2$. We write $B_1$ for the upper triangular Borel, $B_2$ for the lower triangular Borel, and $B_3$ for $nB_2n^{-1}$, where $n=\left[\begin{smallmatrix}
            1 & 1 \\
            & 1
        \end{smallmatrix}\right]$. Then $\mathscr{B}:=B_1\times B_2\times B_3$ is a Borel of $\mathscr{G}$. With this choice, the $\mathscr{B}$-orbit of the identity is the unique open $\mathscr{B}$-orbit on the homogeneous affine variety $\mathcal{X}:=\mathscr{H}\backslash \mathscr{G}$. We write $\mathcal{Y}$ for this open orbit. We write $T$ for the diagonal split torus of $\GL_2$, and $\mathscr{T}:=T\times T\times nTn^{-1}\subseteq\mathscr{B}$. Write $w_1,w_2,w_3$ for the simple reflections associated to each copy of $\GL_2$ in $\mathscr{G}$. The Weyl group $W$, of $\mathscr{G}$ is given by $\langle w_1\rangle \times \langle w_2\rangle\times  \langle w_3\rangle\simeq (\mathbf{Z}/2\mathbf{Z})^3$. For each simple reflection we write $\mathscr{P}_i$ for the associated parabolic with respect to our fixed choice of $\mathscr{T}\subseteq \mathscr{B}$ (e.g $\mathscr{P}_1=\GL_2\times B_2\times B_3$). In \cite{knop1995set}, Knop defines an action of the Weyl group of $\mathscr{G}$ on the set of Borel orbits (over $\overline{\mathbf{Q}}_p$) which is summarized in Section $2.2$ of \cite{Sakellaridis_2008}. We want to compute the action of each simple reflection on the open orbit $\mathcal{Y}$. We do this for $i=1$ and the other two cases are handled in the same fashion. For this, we once again follow \cite[\S $2.2$]{Sakellaridis_2008}. Consider the quotient 
        $$\mathscr{P}_1\rightarrow \mathrm{PGL}_2=\mathrm{Aut}(\mathbf{P}_1)=\mathrm{Aut}(\mathscr{B}\backslash\mathscr{P}_1).$$
        The image of $\mathrm{Stab}_{\mathscr{P}_1}(1)$ (where we regard $1$ as an element of $\mathscr{H}\backslash\mathscr{G}$) in $\mathrm{PGL}_2$ is a spherical subgroup. It is readily seen that this image is given by the non-trivial torus $B_2\cap B_3$ in $\PGL_2$. Thus it
        is of ``type T'' as per the classification in \cite[\S $2.1.1$]{Sakellaridis_2008}. From the description of Knop's action, we deduce that $w_1$ fixes the open orbit $\mathcal{Y}$. Repeating this for $w_2$ and $w_3$, we see that $W$ fixes the open orbit $\mathcal{Y}$. Now, in our case, the standard parabolic $P(\mathcal{X}):=\{g\in \mathscr{G}\ |\ \mathcal{Y}\cdot g=\mathcal{Y}\}$ is simply equal to $\mathscr{B}$, and hence the Weyl group and modular quasi-character of its Levi are trivial. Since we've shown that $\mathcal{Y}$ is a fixed point for the Knop action, it follows from \cite[Theorem $2.2.2$]{Sakellaridis_2008} (originally due to \cite{knop1995set}) that the ``little Weyl group'', $W_\mathcal{X}$, of $\mathcal{X}$ (originally introduced by \cite{brion1990vers}) coincides with the whole Weyl group $W$, of $\mathscr{G}$. We now turn our attention to verifying the assumptions in \cite[\S $8$]{sakellaridis2013spherical}. Using, the discussion on simple root types above, and \cite[\S $7$]{sakellaridis2013spherical}, we see that the assumptions of \cite[Theorem $7.2.1$]{sakellaridis2013spherical} are satisfied. Furthermore, it is clear that $\mathcal{Y}(\mathbf{Q}_p)$ has a unique $\mathscr{B}(\mathbf{Q}_p)$-orbit, since $\mathrm{Stab}_{\mathscr{B}}(1)=Z_{\GL_2}^\mathrm{diag}$ is smooth and by \cite[Proposition $1$]{bhargava2014arithmetic}, the set of such orbits is in bijection with the kernel of the map$$H^1\left(\mathrm{Gal}(\overline{\mathbf{Q}}_p/\mathbf{Q}_p),\mathrm{Stab}_\mathscr{B}(1)(\overline{\mathbf{Q}}_p)\right)\longrightarrow H^1\left(\mathrm{Gal}(\overline{\mathbf{Q}}_p/\mathbf{Q}_p),\mathscr{B}(\overline{\mathbf{Q}}_p)\right).$$But the left-most term is trivial by an application of Hilbert $90$ (in fact both terms are). We write $\mathscr{T}_\mathcal{X}$ for $\mathscr{T}/(\mathscr{T}\cap \mathscr{H})=\mathscr{T}/Z_{\GL_2}^{\mathrm{diag}}$. The little Weyl group $W_\mathcal{X}=W$ acts on $X_*(\mathscr{T}_X)$ and it remains to verify that the natural map
        \begin{align}\label{eq:natural map}\mathbf{C}[X_*(\mathscr{T})]^W\longrightarrow \mathbf{C}[X_*(\mathscr{T}_\mathcal{X})]^{W_\mathcal{X}}\end{align}
        is surjective. But $W_\mathcal{{X}}=W$, and $\mathbf{C}[X_*(\mathscr{T})]$ is a semisimple $\mathbf{C}[W]$-module. Thus, surjectivity of \eqref{eq:natural map} follows from surjectivity of the map $\mathbf{C}[X_*(\mathscr{T})]\rightarrow \mathbf{C}[X_*(\mathscr{T}_\mathcal{X})]$ which is simply the quotient from the algebra $\mathbf{C}[x_1^{\pm1},y_1^{\pm1},x_2^{\pm1},y_2^{\pm1},x_3^{\pm1},y_3^{\pm1}]$ down to $\mathbf{C}[x_1^{\pm1},y_1^{\pm1},x_2^{\pm1},y_2^{\pm1},x_3^{\pm1},y_3^{\pm1}]/(1-x_1y_1x_2y_2x_3y_3)$. Finally, we are in a position to apply \cite[Corollary $8.0.4$]{sakellaridis2013spherical} that $C_c^\infty(\mathcal{X}(\mathbf{Q}_p))^{\mathscr{G}_p^\circ}$ is a cyclic $\mathcal{H}_{\mathscr{G}_p}^\circ$-module generated by $\ch(\mathcal{X}(\mathbf{Z}_p))$, where the Hecke action is the one induced from regarding $C_c^\infty(\mathcal{X}(\mathbf{Q}_p))$ as a smooth $\mathscr{G}_p$-representation via right translations. Finally, an unravelling of definitions, shows that we have a canonical isomorphism
        $$C_c^\infty(\mathcal{X}(\mathbf{Q}_p))^{\mathscr{G}_p^\circ}\simeq C_c^\infty(H_p^\circ\backslash G_p/G_p^\circ)\ ,\ F\mapsto\left(f_F:g\mapsto F(g,1)\right)$$
        which is equivariant with respect to the Hecke action of $\mathcal{H}_{\mathscr{G}_p}^\circ\simeq \mathcal{H}_{G_p}^\circ\otimes \mathcal{H}_{H_p}^\circ$, and maps $\ch(\mathcal{X}(\mathbf{Z}_p))$ to $\ch(G_p^\circ)$. This concludes the proof of the theorem.
    \end{proof}
\end{thm}
\subsection{Cyclicity of $\mathcal{I}(G_p/G_p^\circ)$} We will now use \Cref{cyclicity via sakellaridis} to determine the structure of $\mathcal{I}(G_p/G_p^\circ)$ over the spherical Hecke algebra of $G_p$. Our approach draws heavily from \cite[\S $4.3$]{Loeffler_2021}.  

\begin{thm}\label{thm cyclicity on coinvariants}
    Let $\phi_0=\ch(\mathbf{Z}_p^2)$ and $\xi_0=\ch(G_p^\circ)$. The module $\mathcal{I}(G_p/G_p^\circ)$ is cyclic over $\mathcal{H}_{G_p}^\circ$ generated by $[\phi_0\otimes \xi_0]$.
    \begin{proof}
        The proof is essentially translating \cite[Lemma $4.3.3$ \& Proposition $4.3.4$]{Loeffler_2021}, and then applying the formalism of \cite[Theorem $4.3.6$]{Loeffler_2021} together with \Cref{cyclicity via sakellaridis}. But since it's quite involved and combines quite a few results, we give the details here. We write $A$ for the rank one torus given by $Z_H$. We have embeddings of algebras \begin{align*}&\mathcal{H}_{A_p}^\circ\overset{\Delta_H}{\longrightarrow} \mathcal{H}_{H_p}^\circ,\ \ch(\left[\begin{smallmatrix}
            p& \\
            & p
        \end{smallmatrix}\right] A_p^\circ)\mapsto \ch(\left[\begin{smallmatrix}
            p& \\
            & p
        \end{smallmatrix}\right] H_p^\circ)\\
        &\mathcal{H}_{A_p}^\circ\overset{\Delta_G}{\longrightarrow} \mathcal{H}_{G_p}^\circ,\ \ch(\left[\begin{smallmatrix}
            p& \\
            & p
        \end{smallmatrix}\right] A_p^\circ)\mapsto \ch((\left[\begin{smallmatrix}
            p& \\
            & p
        \end{smallmatrix}\right],\left[\begin{smallmatrix}
            p& \\
            & p
        \end{smallmatrix}\right])G_p^\circ)
        \end{align*}
        and that $\mathcal{H}_{A_p}^\circ\simeq \mathbf{C}[\ch(\left[\begin{smallmatrix}
            p & \\
            & p
        \end{smallmatrix}\right]A_p^\circ)^{\pm 1}]$ as $\mathbf{C}$-algebras.
        We write $\phi_0=\ch(\mathbf{Z}_p^2)$. By the proof of \cite[Lemma $4.3.3$]{Loeffler_2021}, there exists an algebra homomorphism $\mathcal{H}_{H_p}^\circ\overset{\zeta_H}{\longrightarrow} \mathcal{H}_{A_p}^\circ$, such that $\theta\cdot \phi_0=(\Delta_H\circ \zeta_H)(\theta)\cdot \phi_0$ for every $\theta\in\mathcal{H}_{H_p}^\circ$. The definition of this map can be found in \textit{op.cit}, and it is not hard to check that it satisfies the aforementioned property. We now consider an arbitrary element in $\mathcal{I}(G_p/G_p^\circ)$ of the form $[\delta]=[\phi\otimes \xi]$. Up to a non-zero scalar multiple, such an element is equal to $[(\ch(A_p^\circ)\cdot_A \phi)\otimes \xi]$ in $\mathcal{I}(G_p/G_p^\circ)$, and thus we may assume that $[\delta]=[\phi\otimes \xi]$ with $\phi\in (\mathcal{S}(\mathbf{Q}_p^2))^{A_p^\circ}$. By the proof of \cite[Proposition $4.3.4$]{Loeffler_2021}, the $\mathbf{C}[H_p]$-module $\mathcal{S}(\mathbf{Q}_p^2)^{A_p^\circ}$ is cyclic, generated by $\phi_0$. Since the vector $\phi_0$ is unramified, we may assume that $[\delta]=[(\theta\cdot\phi_0)\otimes \xi]$ with $\theta\in C_c^\infty(H_p/H_p^\circ)$. But this is nothing more than $[\phi_0\otimes (\theta^{'}\cdot_H \xi)]$. Since $\xi$ is right $G_p^\circ$-invariant and $\theta^{'}\in \mathcal{H}_{H_p}^\circ$, the convolution $\theta^{'}\cdot_H\xi$ is an element of $C_c^\infty(H_p^\circ\backslash G_p/G_p^\circ)$. By \Cref{cyclicity via sakellaridis}, we can express $\theta^{'}\cdot_H \xi$ as a linear combination of elements of the form $(\theta_1\otimes \theta_2)\cdot \ch(G_p^\circ)$ with $\theta_1\otimes \theta_2\in \mathcal{H}_{G_p}^\circ\otimes \mathcal{H}_{H_p}^\circ$. A quick integral computation shows that $(\theta_1\otimes \theta_2)\cdot \ch(G_p^\circ)$ is equal to $\theta_2\cdot_H \theta_1^{'}$. Thus, we may assume that $[\delta]=[\phi_0\otimes(\alpha\cdot_H \beta)]$ where $\alpha\in\mathcal{H}_{H_p}^\circ$ and $\beta\in\mathcal{H}_{G_p}^\circ$. As in the proof of \cite[ Theorem $4.3.6$]{Loeffler_2021}, we have the chain of equalities
        \begin{align*}
            [\delta]&=[\phi_0\otimes (\alpha\cdot_H\beta)]\\
            &=[(\alpha^{'}\cdot \phi_0)\otimes \beta]\\
            &=[(\Delta_H(\alpha_1)\cdot \phi_0)\otimes \beta]\\
            &=[\phi_0\otimes (\Delta_H(\alpha_1)^{'} \cdot_H \beta)]\\
            &=[\phi_0\otimes (\Delta_G(\alpha_1)^{'} \cdot_G \beta)]
        \end{align*}
        where $\alpha_1:=\zeta_H(\alpha^{'})\in \mathcal{H}_{A_p}^\circ$, and the last equality follows from the fact that $$\Delta_H(\ch(\left[\begin{smallmatrix}
            p & \\
            & p
        \end{smallmatrix}\right] A_p^\circ))\cdot_H \beta=\beta(\left[\begin{smallmatrix}
            p & \\
            & p
        \end{smallmatrix}\right]^{-1}(-))=\Delta_G(\ch(\left[\begin{smallmatrix}
            p & \\
            & p
        \end{smallmatrix}\right] A_p^\circ))\cdot_G \beta.$$
        If we set $\theta_3:= \Delta_G(\alpha_1)^{'}\cdot_G \beta$, which is an element of $\mathcal{H}_{G_p}^\circ$, we see that $[\delta]$ is equal to $[\phi_0\otimes (\xi_0\cdot_G \Lambda)]$, which is in turn equal to $\Lambda^{'}*_G[\phi_0\otimes \xi_0].$ This concludes the proof.
    \end{proof}
\end{thm}

\begin{rem}
    Later on, we will actually show that $\mathcal{I}(G_p/G_p^\circ)$ is free over $\mathcal{H}_{G_p}^\circ$. But that will require a different approach, so we postpone it for now. 
\end{rem}

\noindent We now showcase how the structure of $\mathcal{I}(G_p/G_p^\circ)$ can be used to recover, unify and adapt certain results of Jacquet, Shalika, Piatetski-Shapiro and Prasad to unramified representations of Whittaker type.

\subsubsection{Relation to results of Jacquet, Piatetskii-Shapiro, Shalika}

\begin{thm}[Jacquet, Piatetskii-Shapiro, Shalika]\label{prop JSPS}
    For unramified $H_p$-representations $\pi_{p,1}, \pi_{p,2}$ of Whittaker type, we have
    $$\dim\ \Hom_{H_p}(\mathcal{S}(\mathbf{Q}_p^2)\otimes \pi_{p,1}\otimes \pi_{p,2},\mathbf{1})\leq 1$$
    and every non-zero such linear form, is non-vanishing on $\phi_0\otimes W_{\pi_p,1}^\mathrm{sph}\otimes W_{\pi_p,2}^\mathrm{sph}$.
\end{thm}
This is well-known to experts. Below, we give a short algebraic proof of \Cref{prop JSPS}, which gives uniqueness, and existence of good test vectors, relying on the structure of $\mathcal{I}(G_p/G_p^\circ)$ as a module over the spherical Hecke algebra. This Hecke approach is what will allow us to study the \textit{integral} behaviour of this period later on.

\begin{proof}
    There exists a surjective linear map from $\mathcal{S}(\mathbf{Q}_p^2)\otimes C_c^\infty(G_p/G_p^\circ)$ to $\mathcal{S}(\mathbf{Q}_p^2)\otimes \pi_{p,1}\otimes \pi_{p,2}$ that is given by sending $\phi\otimes \xi$ to $\phi\otimes\left(\xi\cdot (W_{\pi_{p,1}}^\mathrm{sph}\otimes W_{\pi_{p,2}}^\mathrm{sph})\right)$. This clearly induces a map from $\mathcal{I}(G_p/G_p^\circ)$ to $\left(\mathcal{S}(\mathbf{Q}_p^2)\otimes \pi_{p,1}\otimes \pi_{p,2}\right)_{H_p}$. By \Cref{thm cyclicity on coinvariants}, every element of $\mathcal{I}(G_p/G_p^\circ)$ is of the form $\theta*[\phi_0\otimes \xi_0]=[\phi_0\otimes \theta^{'}]$ for $\theta\in\mathcal{H}_{G_p}^\circ$. Hence the space $\left(\mathcal{S}(\mathbf{Q}_p^2)\otimes \pi_{p,1}\otimes \pi_{p,2}\right)_{H_p}$ is at most one dimensional and is generated by $[\phi_0\otimes W_{\pi_{p,1}}^\mathrm{sph}\otimes W_{\pi_{p,2}}^\mathrm{sph}]$. This concludes the proof.
\end{proof}

\subsubsection{Relation to results of Prasad}

\begin{thm}[Prasad]\label{thm trilinear forms prasad}
Let $\pi_{p,1},\pi_{p,2},\pi_{p,3}$ be unramified $H_p$-representations of Whittaker type. Then
$$\dim\ \Hom_{H_p}(\pi_{p,1}\otimes \pi_{p,2}\otimes \pi_{p,3},\mathbf{1})\leq 1$$
and every non-zero such linear form is non-vanishing on $W_{\pi_{p,1}}^\mathrm{sph}\otimes W_{\pi_{p,2}}^\mathrm{sph}\otimes W_{\pi_{p,3}}^\mathrm{sph}$.
\end{thm}
For $\pi_{p,1},\pi_{p,2},\pi_{p,3}$ irreducible, the uniqueness is (part of) \cite[Theorem $1.1$]{prasad1990trilinear} and is proved using invariant distributions. For the cases where some of the $\pi_{p,i}$ are reducible of Whittaker type, the uniqueness is treated in three results of \cite{harris2001note}. For good test vectors in the unramified case, this is Theorem $1.3$ in \cite{prasad1990trilinear}. Below we give a short alternative proof of this result using \Cref{prop JSPS}. This approach will later allow us to make \textit{integral} statements for the non-zero period of this $\Hom$-space using results for the lattice $\mathcal{I}(G_p/G_p^\circ,\mathbf{Z}[1/p])$.
\begin{proof}
    By twisting, we may assume that $\pi_{p,1}$ is the normalized induction $I(\chi_p,|\cdot|_p^{-1/2})$ with $\chi_p$ an unramified quasi-character of $\mathbf{Q}_p^\times$ with $I(\chi_p,|\cdot|_p^{-1/2})$ of Whittaker type. Here $I(\chi_p,|\cdot|_p^{-1/2})$ denotes the representation $\mathrm{ind}_{B_p}^{G_p}\left[\begin{smallmatrix}
        \chi_p & \\
        & |\cdot|_p^{-1/2}
    \end{smallmatrix}\right]$ where $B_p$ is the upper triangular Borel, and the diagonal character acts trivially on the unipotent radical. Using \cite[Proposition $3.2.2$ \& Lemma $3.2.5$]{Loeffler_Skinner_Zerbes_2021}, and \cite[Proposition $1.5.2$]{grossi2020norm}, one has a non-zero $H_p$-equivariant map $\mathcal{S}(\mathbf{Q}_p^2)\rightarrow \pi_{p,1}$. Since $\pi_{p,1}$ is unramified of Whittaker type, this map is surjective and maps $\phi_0$ to $W_{\pi_{p,1}}^\mathrm{sph}$. This induces a surjective $H_p$-equivariant map $\mathcal{S}(\mathbf{Q}_p^2)\otimes C_c^\infty(G_p/G_p^\circ)\rightarrow \pi_{p,1}\otimes \pi_{p,2}\otimes \pi_{p,3}$, which in turn induces a surjective linear map $\mathcal{I}(G_p/G_p^\circ)\rightarrow (\pi_{p,1}\otimes \pi_{p,2}\otimes \pi_{p,3})_{H_p}$. The result then follows at once in the same fashion as the proof of \Cref{prop JSPS} given above, once again utilizing \Cref{cyclicity via sakellaridis}.
\end{proof}

    We would now like to realise the cyclicity of the Hecke module $\mathcal{I}(G_p/G_p^\circ)$ and study its properties when we specialise to integral data as in \Cref{def integral lattice}. The right way forward as mentioned in the introduction is \textit{not} to directly use the local Rankin-Selberg zeta integral as in \cite{Loeffler_2021}. Indeed, it does not seem possible to treat \textit{general} integral test vectors that way and obtain meaningful results at an integral level. Instead, we proceed as in the following sections.

\subsection{Rephrasing using the mirabolic subgroup}\label{sec from and back}

From now on, we will drop the notation $[\cdot]$ on elements of $\mathcal{I}(G_p/G_p^\circ)$. We denote by $P$ the mirabolic subgroup of $H=\GL_2$, given by $\left\{\left[\begin{smallmatrix}
    * & * \\
    & 1
\end{smallmatrix}\right]\right\}$, and we will also regard it as a subgroup of $G$ through the diagonal embedding $H\hookrightarrow G$. As usual, we write $P_p$ for its $\mathbf{Q}_p$-points and 
$P_p^\circ$ for its $\mathbf{Z}_p$-points. There's an identification
$$P_p \backslash H_p\overset{\simeq}{\longrightarrow} \mathbf{Q}_p^2-\{(0,0)\}\ ,\ h\mapsto (0,1)h.$$
This induces an embedding of $H_p$-representations (where recall that $\mathrm{Ind}$ denotes unnormalized, non-compact induction)
$$\mathcal{S}(\mathbf{Q}_p^2)\hookrightarrow \mathrm{Ind}_{P_p}^{H_p}\mathbf{1}\ ,\ \phi\mapsto \left(f_\phi: h\mapsto \phi((0,1)h)\right)$$
under which $\mathcal{S}_0(\mathbf{Q}_p^2)$ gets identified with $\mathrm{c}\text{-}\mathrm{Ind}_{P_p}^{H_p}\mathbf1$. 
\begin{defn}We define the linear map
\begin{align*}\Xi:\mathcal{S}(\mathbf{Q}_p^2)\otimes C_c^\infty(G_p/G_p^\circ)\longrightarrow C^\infty(P_p\backslash G_p/G_p^\circ)=\left(\mathrm{Ind}_{P_p}^{G_p}\mathbf{1}\right)^{G_p^\circ}\\
\phi\otimes \xi\mapsto \left(\Xi_{\phi\otimes \xi}:g\mapsto\int_{H_p}\xi(h^{-1}g)f_\phi(h)\ dh\right)\end{align*}
\end{defn}
\noindent A note should be made on the ``convergence'' of the integral used to define the map $\Xi$, which follows from the compactness of $H_p\cap g_1 G_p^\circ g_2$ in $H_p$, for any $g_1,g_2\in G_p$.
It follows by construction, that $\Xi_{h(\phi\otimes \xi)}=\Xi_{\phi\otimes \xi}$ for any $h\in H_p$, and hence $\Xi$ factors through $\mathcal{I}(G_p/G_p^\circ)$. An unraveling of definitions also shows that this map is Hecke equivariant with respect to $\mathcal{H}_{G_p}^\circ$, where the action on the right is given by the natural action of $\mathcal{H}_{G_p}^\circ$ on $(\mathrm{Ind}_{P_p}^{H_p}\mathbf{1})^{G_p^\circ}$. 

Once again, we \textit{emphasize} that for $\phi\otimes \xi$ with $\phi\in\mathcal{S}_0(\mathbf{Q}_p^2)$, the corresponding function $\Xi_{\phi\otimes \xi}$ belongs in $C_c^\infty(P_p\backslash G_p/G_p^\circ)$, i.e. it is already compactly supported. This will be important later on. We write $C^\infty(P_p\backslash G_p/G_p^\circ,\mathbf{Z}[1/p])$ for the $\mathcal{H}_{G_p}^\circ(\mathbf{Z}[1/p])$-submodule of $C^\infty(P_p\backslash G_p/ G_p^\circ)$ consisting of $\mathbf{Z}[1/p]$-valued functions, and similarly for compactly supported functions.

\begin{prop}\label{prop integrally from coinvariants to C^infinity}
    The image of $\mathcal{I}(G_p/G_p^\circ,\mathbf{Z}[1/p])$ under the map $\Xi$, is contained in $C^\infty(P_p\backslash G_p/G_p^\circ,\mathbf{Z}[1/p])$.
    \begin{proof}
        Let $\phi\otimes \ch(gG_p^\circ)$ be an element of $\mathcal{I}(G_p/G_p^\circ,\mathbf{Z}[1/p])$. Recall that this means by definition that $\phi$ is valued in 
        $$\frac{1}{\vol_{H_p}\left(\mathrm{Stab}_{H_p}(\phi)\cap g G_p^\circ g^{-1}\right)}\mathbf{Z}[1/p].$$
        For ease of notation, write $\delta=\phi\otimes \ch(gG_p^\circ)$.
        For $x\in G_p$, we have by construction that
        \begin{align}\label{eq:7}
            \Xi_\delta(x)&=\int_{H_p}\left(h\cdot \ch(gG_p^\circ)\right)(x)\ (h\cdot \phi)(0,1)\ dh
        \end{align}
        Let $K$ be the open compact of $H_p$ given by $\mathrm{Stab}_{H_p}(\phi)\cap g G_p^\circ g^{-1}$. We can split the integral in \eqref{eq:7} into a finite sum of the form  
        \begin{align}\label{eq: 8}\Xi_\delta(x)=\sum_\gamma \int_{\gamma K}\left(h\cdot \ch(gG_p^\circ)\right)(x)\ (h\cdot \phi)(0,1)\ dh=\vol_{H_p}(K)\sum_\gamma \ch(gG_p^\circ)(\gamma^{-1}x)\ \phi((0,1)\gamma).\end{align}
        By the assumption on the values of $\phi$, \eqref{eq: 8} is an element of $\mathbf{Z}[1/p]$ which concludes the proof.
    \end{proof}
\end{prop}
\subsection{Mapping to compact induction}
We first start by noting that the element $\Xi_0=\Xi(\ch(\mathbf{Z}_p^2)\otimes \ch(G_p^\circ))$, even though it is unramified in the sense that is $G_p^\circ$-fixed, it does not coincide with the trivial unramified element in $\mathrm{c}\text{-}\mathrm{Ind}_{P_p}^{G_p}\mathbf{1}\subseteq\mathrm{Ind}_{P_p}^{G_p}\mathbf{1}$ given by $\ch(P_pG_p^\circ)$. For $i\in\mathbf{Z}$, we write $Z_{p,i}$ for the central subset of $H_p$ given by
$\left\{\left[\begin{smallmatrix}
    p^n & \\
    & p^n
\end{smallmatrix}\right]\ |\ n\geq i\right\}.$
\begin{lem}\label{lem Xi_0}
    The function $\Xi_0$ is equal to the characteristic function $\ch(Z_{p,0} P_pG_p^\circ)$ in $C^\infty(P_p\backslash G_p/G_p^\circ)$.
    \begin{proof}
        For an element $g=(g_1,g_2)\in G_p$, we have
        \begin{align*}
            \Xi_0(g)=\int_{H_p\cap g G_p^\circ} \ch(\mathbf{Z}_p^2)((0,1)h)\ dh
            =\begin{cases}
                0,\ &\mathrm{if}\ g_1H_p^\circ\neq g_2 H_p^\circ\\
                \ch(\mathbf{Z}_p^2)((0,1)g_1),\ &\mathrm{if}\ g_1H_p^\circ= g_2H_p^\circ.
            \end{cases}
        \end{align*}
        The result then follows from the decomposition $H_p=\bigsqcup_{n\in\mathbf{Z}}\left[\begin{smallmatrix}
            p^n & \\
            & p^n
        \end{smallmatrix}\right] P_pH_p^\circ.$
    \end{proof}
\end{lem}
\begin{prop}\label{prop mapping to compact induction}
    There exists $\Psi\in\mathrm{End}_{\mathcal{H}_{G_p}^\circ}\left(C^\infty(P_p\backslash G_p/G_p^\circ)\right)$ for which the composition $\Xi_c:=\Psi\circ \Xi$ gives rise to the following well-defined commutative diagram, where the top map is $\mathcal{H}_{G_p}^\circ$-equivariant.
    \[\begin{tikzcd}[ampersand replacement=\&,row sep=scriptsize]
	{\mathcal{I}(G_p/G_p^\circ)} \&\& {C_c^\infty(P_p\backslash G_p/G_p^\circ)=\left(\mathrm{c\text{-}Ind}_{P_p}^{G_p}\mathbf{1}\right)^{G_p^\circ}} \\
	{\mathcal{I}(G_p/G_p^\circ,\mathbf{Z}[1/p])} \&\& {C_c^\infty(P_p\backslash G_p/G_p^\circ,\mathbf{Z}[1/p]).}
	\arrow["{\Xi_c}", from=1-1, to=1-3]
	\arrow["{\Xi_c}", from=2-1, to=2-3]
	\arrow[hook', from=2-1, to=1-1]
	\arrow[hook, from=2-3, to=1-3]
\end{tikzcd}\]
Finally, the map $\Xi_c$ maps $\ch(\mathbf{Z}_p^2)\otimes \ch(G_p^\circ)$ to $\ch(P_pG_p^\circ)$.
\begin{proof}
    We write $\mathcal{S}_p=\mathcal{S}_{p,1}\mathcal{S}_{p,2}$ for the Hecke operator in $(\mathcal{H}_{G_p}^\circ)^\times$ associated to the element $\left(\left[\begin{smallmatrix}
        p & \\
        & p
    \end{smallmatrix}\right],\left[\begin{smallmatrix}
        p & \\
        & p
    \end{smallmatrix}\right]\right)$ in $G_p$. We define the $\mathcal{H}_{G_p}^\circ$-equivariant endomorphism $\Psi$ to be the map
    $$\Psi: C^\infty(P_p\backslash G_p/G_p^\circ)\rightarrow C^\infty(P_p\backslash G_p/G_p^\circ),\ \xi\mapsto (1-\mathcal{S}_p^{-1})\cdot\xi.$$
    As in the statement of the proposition, we put $\Xi_c=\Psi\circ \Xi$. By Hecke equivariance and \Cref{thm cyclicity on coinvariants}, it suffices to show that $\Xi_c$ maps $\ch(\mathbf{Z}_p^2)\otimes \ch(G_p^\circ)$ to an element in the compact induction. In particular we will show that $\ch(\mathbf{Z}_p^2)\otimes \ch(G_p^\circ)$ is mapped to $\ch(P_pG_p^\circ)$. By \Cref{lem Xi_0}, $\Xi_c\left(\ch(\mathbf{Z}_p^2)\otimes \ch(G_p^\circ)\right)$ is given by $\Psi\left(\ch(Z_{p,0}P_p G_p^\circ)\right)$. This in turn is given by the function
    \begin{align*}
        \Psi\left(\ch(Z_{p,0}P_p G_p^\circ)\right)(g)&=\ch(Z_{p,0}P_p G_p^\circ)(g)-\int_{G_p} \mathcal{S}_p^{-1}(\gamma)\ \ch(Z_{p,0}P_p G_p^\circ)(g\gamma)\ d\gamma\\
        &= \ch(Z_{p,0}P_p G_p^\circ)(g)- \ch(Z_{p,0}P_p G_p^\circ)\left(g\left[\begin{smallmatrix}
            p & \\
            & p
        \end{smallmatrix}\right]^{-1}\right).
    \end{align*}
    From this it is clear that $\Xi_c\left(\ch(\mathbf{Z}_p^2)\otimes \ch(G_p^\circ)\right)=\ch(P_pG_p^\circ)$ as claimed. Finally, the fact that $\Xi_c$ induces a map on integral level follows immediately from \Cref{prop integrally from coinvariants to C^infinity}.
\end{proof}
\end{prop}
\noindent Our constructions so far can be summarised in the following commutative diagrams, where the square on the left is equivariant with respect to the action of the spherical Hecke algebra $\mathcal{H}_{G_p}^\circ$ and $\mathcal{I}(G_p/G_p^\circ)$ is cyclic over $\mathcal{H}_{G_p}^\circ$ generated by $\delta_0:=\ch(\mathbf{Z}_p^2)\otimes \ch(G_p^\circ)$.
\[\begin{tikzcd}[column sep=small,row sep=scriptsize]
	{\delta_0} & {\mathcal{I}(G_p/G_p^\circ)} & {C^\infty(P_p\backslash G_p/G_p^\circ)} & {\mathcal{I}(G_p/G_p^\circ,\mathbf{Z}[1/p])} & {C^\infty(P_p\backslash G_p/G_p^\circ,\mathbf{Z}[1/p])} \\
	{\ch(P_pG_p^\circ)} & {C_c^\infty(P_p\backslash G_p/G_p^\circ)} & {C^\infty(P_p\backslash G_p/G_p^\circ)} & {C_c^\infty(P_p\backslash G_p/G_p^\circ,\mathbf{Z}[1/p])} & {C^\infty(P_p\backslash G_p/G_p^\circ,\mathbf{Z}[1/p])}
	\arrow["{\Xi_c}", maps to, from=1-1, to=2-1]
	\arrow["\Xi", from=1-2, to=1-3]
	\arrow["{\Xi_c}", from=1-2, to=2-2]
	\arrow["{{\cdot\left(1-\mathcal{S}_p^{-1}\right)}}"', from=1-3, to=2-3]
	\arrow["\Xi", from=1-4, to=1-5]
	\arrow["{\Xi_c}", from=1-4, to=2-4]
	\arrow["{{\cdot\left(1-\mathcal{S}_p^{-1}\right)}}"', from=1-5, to=2-5]
	\arrow[hook, from=2-2, to=2-3]
	\arrow[hook, from=2-4, to=2-5]
\end{tikzcd}\]

\section{Hecke equivariant forms on $C_c^\infty(P_p\backslash G_p/G_p^\circ)$}\label{sec Hecke forms}
We fix once and for all an additive character $\psi:\mathbf{Q}_p\rightarrow\mathbf{C}^\times$ of conductor one. It is standard that such a character exists (e.g \cite[\S $7$]{ramakrishnan2013fourier}). We consider the family of $G_p$-representations $\pi_{p,1}\boxtimes \pi_{p,2}$ where $\pi_{p,1}$ and $\pi_{p,2}$ are unramified $H_p$-representations of Whittaker type. As it is often the case, we identify $\pi_{p,1}$ with its Whittaker model $\mathcal{W}(\pi_{p,1},\psi)$ and similarly, $\pi_{p,2}$ with $\mathcal{W}(\pi_{p,2},\psi^{-1})$. For further theory surrounding Whittaker models for this family of representations, we refer the reader to \cite{bump_1997}. 
\subsection{Splitting of measures}\label{sec measures}
It follows from a standard computation, that the modular quasi-character of the mirabolic subgroup $P_p$ is simply given by $\delta_P\left(\left[\begin{smallmatrix}
    a & b \\
    & 1
\end{smallmatrix}\right]\right)=|a|_p$. We write $d^Rx$ for the normalized right Haar measure on $P_p$. Then $\delta_P^{-1} d^Rx$ is the normalized left Haar measure on $P_p$.  We write $C_c^\infty(P_p\backslash G_p;\delta_P)$ for the space of smooth functions on $G_p$, that are compactly supported modulo $P_p$ on the left, transform by $\delta_P$ under left $P_p$-translations, and are fixed under right translations by some open compact subgroup of $G_p^\circ$. Following \cite[\S $2.2$]{kurinczuk2017rankin} and \cite[I $2.8$]{vigneras1996representations}, there exists a surjective map
\begin{align}\label{eq: surjective map}\mathcal{H}(G_p)\longrightarrow C_c^\infty(P_p\backslash G_p;\delta_P),\ f\mapsto \left(f^P:g\mapsto \int_{P_p}f(xg) \delta_P(x)^{-1}\ d^Rx\right)\end{align}
\and a unique, up to scalars, linear form $d_{P_p\backslash G_p}g$ on $C_c^\infty(P_p\backslash G_p;\delta_P)$, that is right invariant by $G_p$. We write $$d_{P_p\backslash G_p}g(f):=\int_{P_p\backslash G_p} f(g)\ d_{P_p\backslash G_p},\ f\in C_c^\infty(P_p\backslash G_p;\delta_P) g.$$
Up to correct normalization of $d_{P_p\backslash G_p}g$, we have \begin{align}\label{eq:quotient measure}\int_{G_p}h(g)\ dg=\int_{P_p\backslash G_p} h^P(g)\ d_{P_p\backslash G_p}g,\ h\in \mathcal{H}(G_p)\end{align}
where $dg$ is the normalized Haar measure on $G_p$. We want to be able to unfold the integral $\int_{P_p\backslash G_p} f(g)\ d_{P_p\backslash G_p}g$ for $f\in C_c^\infty(P_p\backslash G_p/G_p^\circ;\delta_P)$.
\begin{lem}\label{lem unfolding}
    Let $f$ be a function in $C_c^\infty(P_p\backslash G_p/G_p^\circ;\delta_P)$. We write $Z_p$ for the center of $H_p$ embedded diagonally in $G_p$. Then,
    $$\int_{P_p\backslash G_p} f(g)\ d_{P_p\backslash G_p}g=\int_{Z_p}\int_{\GL_2(\mathbf{Q}_p)}f(z(\gamma,1))\ dz\ d\gamma.$$
    \begin{proof}
    By \eqref{eq: surjective map}, we can write $f$ as $h^P$ for some $h$ in $\mathcal{H}(G_p)$.
        It is clear that $G_p$ splits as a semidirect product of $ \GL_2(\mathbf{Q}_p)\times 1$ and $H_p$. Hence by a slightly modified version of \cite[Lemma $2.4$]{kurinczuk2017rankin}, using unimodularity, and \eqref{eq:quotient measure}, we see that the integral in question is equal to $\int_{H_p}\int_{\GL_2(\mathbf{Q}_p)}h\left(x(\gamma,1)\right)\ dx\ d\gamma.$ Using  Cartier's argument and \textit{op.cit} one more time, this is nothing more than $$\scalemath{0.98}{\int_{P_p}\int_{Z_p}\int_{H_p^\circ}\int_{\GL_2(\mathbf{Q}_p)}h\left(yzk(\gamma,1)\right)\ d^Ly\ dz\ dk\ d\gamma=\int_{P_p}\int_{Z_p}\int_{H_p^\circ}\int_{\GL_2(\mathbf{Q}_p)}h\left(yzk(\gamma,1)\right)\delta_P(y)^{-1}\ d^Ry\ dz\ dk\ d\gamma}.$$
        By definition of $h^P$, this reduces to
        $$\int_{Z_p}\int_{\GL_2(\mathbf{Z}_p)}\int_{\GL_2(\mathbf{Q}_p)}f(z(\kappa\gamma,\kappa))\ dz\ d\kappa\ d\gamma=\int_{Z_p}\int_{\GL_2(\mathbf{Q}_p)}f(z(\gamma,1))\ dz\ d\gamma$$
        where the last equality follows from the fact that $f$ is right $G_p^\circ$-invariant, together with a change of variables.
    \end{proof}
\end{lem}
\subsection{The linear form $\Lambda_{\Pi_p}$}\label{sec the linear form} The goal of this section is to construct, for each unramified $G_p$-representation $\Pi_p$ of Whittaker type, a linear form $\Lambda_{\Pi_p}$ on $C_c^\infty(P_p\backslash G_p/G_p^\circ)$, satisfying a certain Hecke equivariance condition. This will allow us to ``realize'', in a way, the cyclicity of the module $\mathcal{I}(G_p/G_p^\circ)$. 
\begin{defn}\label{def W}
    Let $\pi_{p,1}$ and $\pi_{p,2}$ be two unramified $H_p$-representations of Whittaker type and let $\Pi_p=\pi_{p,1}\boxtimes \pi_{p,2}$. For $f\in C_c^\infty(P_p\backslash G_p/G_p^\circ)$ and $g\in G_p$, we define 
    $$\mathscr{W}_{\Pi_p}(s;g,f):=f(g)\ |\det(g_2)|_p^s\int_{\mathbf{Q}_p^\times}\left(W_{\pi_{p,1}}^\mathrm{sph}\boxtimes W_{\pi_{p,2}}^\mathrm{sph}\right)\left(\left[\begin{smallmatrix}
        a & \\
        & 1
    \end{smallmatrix}\right]g\right)\ |a|_p^{s-1}\ d^\times a\in \mathbf{C}(p^s,p^{-s})$$
    where $g_2$ denotes the projection onto the second factor, and $s\in\mathbf{C}$ is a complex parameter. For fixed $f$ and $g$ and $\Re(s)$ large enough, (independently of $f$ and $g$), using one of the main results of \cite{jacquet1983rankin}, the integral appearing above converges and admits unique meromorphic continuation as a rational function of $p^s$.
\end{defn}
\noindent  
As a function onf $g\in G_p$, $\mathscr{W}_{\Pi_p}(s;g,f)$ is compactly supported modulo $P_p$ on the left by construction. For $x=\left[\begin{smallmatrix}
    x_1 & x_2\\
     & 1
\end{smallmatrix}\right]\in P_p$, we have 
\begin{align}\label{eq: 13}
    \mathscr{W}_{\Pi_p}(s;xg,f)&=f(xg)\ |\det(xg_2)|_p^s\int_{\mathbf{Q}_p^\times} \left(W_{\pi_{p,1}}^\mathrm{sph}\boxtimes W_{\pi_{p,2}}^\mathrm{sph}\right)\left(\left[\begin{smallmatrix}
        ax_1 & ax_2 \\
        & 1
    \end{smallmatrix}\right]g\right)\ |a|_p^{s-1}\ d^\times a\\
    \nonumber &=f(g)\ |\det(xg_2)|_p^s\  |x_1|_p^{1-s}\int_{\mathbf{Q}_p^\times} \left(W_{\pi_{p,1}}^\mathrm{sph}\boxtimes W_{\pi_{p,2}}^\mathrm{sph}\right)\left(\left[\begin{smallmatrix}
        a & \\
        & 1
    \end{smallmatrix}\right]g\right)\ |a|_p^{s-1}\ d^\times a\\
    \nonumber &=\delta_P(x)\  \mathscr{W}_{\Pi_p}(s;g,f)
\end{align}
where we have also used the fact that $\pi_{p,1}$ is identified with is $\psi$-Whittaker model, and $\pi_{p,2}$ is identified with its $\psi^{-1}$-Whittaker model.
\begin{defn}
    For $f\in C_c^\infty(P_p\backslash G_p/G_p^\circ)$ and notation as in \emph{\Cref{def W}}, we define
    $$\Lambda_{\Pi_p}(s;f):=\int_{P_p\backslash G_p} \mathscr{W}_{\Pi_p}(s;g,f)\ d_{P_p\backslash G_p} g\in \mathbf{C}(p^s,p^{-s}).$$
\end{defn}
\begin{prop}\label{prop meromorphic cont}
    For $f\in C_c^\infty(P_p\backslash G_p/G_p^\circ)$, the integral $\Lambda_{\Pi_p}(s;f)$ converges for $\Re(s)$ large enough and admits unique meromorphic continuation as a rational function of $p^s$, in the fractional ideal $L(\Pi_p,s)\mathbf{C}[p^s,p^{-s}]$. 
    \begin{proof}
        For $\Re(s)$ sufficiently large, we can unfold the integral using \Cref{lem unfolding}. The result then follows from \cite{jacquet1983rankin}.
    \end{proof}
\end{prop}
\begin{defn} With notation as in \emph{\Cref{def W}} we define the linear form 
$$\Lambda_{\Pi_p}:C_c^\infty(P_p\backslash G_p/G_p^\circ)\longrightarrow \mathbf{C},\ \Lambda_{\Pi_p}(f):= \lim_{s\rightarrow 0}\frac{\Lambda_{\Pi_p}(s;f)}{L(\Pi_p,s)}$$
which is well-defined by \emph{\Cref{prop meromorphic cont}}.
\end{defn}
\subsubsection{Equivariance properties}

As shown in the following proposition, the linear form $\Lambda_{\Pi_p}$ becomes $\mathcal{H}_{G_p}^\circ$-equivariant if we interpret $\mathbf{C}$ as the one dimensional subspace of $G_p^\circ$-fixed vectors in the smooth dual of $\Pi_p$.
\begin{prop}\label{prop equivariance of linear form} Let $\Pi_p=\pi_{p,1}\boxtimes \pi_{p,2}$, with $\pi_{p,i}$ unramified $H_p$-representations of Whittaker type. We write $\omega_{\Pi_p}$ for its central character and $\Theta_{\Pi_p}$ for its spherical Hecke eigensystem.
    \begin{enumerate}
    \item The linear form $\Lambda_{\Pi_p}$ defines an element of $$\Hom_{\mathcal{H}_{G_p}^\circ}\left(C_c^\infty (P_p\backslash G_p/G_p^\circ), \left(\Pi_p^\vee\right)^{G_p^\circ}\right).$$
        \item For $f_0=\ch(P_pG_p^\circ)$, we have $\Lambda_{\Pi_p}\left(f_0\right)=L(\omega_{\Pi_p},0)^{-1}=\Theta_{\Pi_p}(1-\mathcal{S}_p)$. 
    \end{enumerate}
    \begin{proof}
        We firstly prove the first part. Let $f\in C_c^\infty(P_p\backslash G_p/G_p^\circ)$ and $\mathcal{P}\in \mathcal{H}_{G_p}^\circ$.
        \begin{align*}
            \Lambda_{\Pi_p}(s;\mathcal{P}\cdot f)&=\int_{P_p\backslash G_p}\int_{\mathbf{Q}_p^\times} (\mathcal{P}\cdot f)(g)\ |\det(g_2)|_p^s\left(W_{\pi_{p,1}}^\mathrm{sph}\boxtimes W_{\pi_{p,2}}^\mathrm{sph}\right)\left(\left[\begin{smallmatrix}
                a & \\
                & 1
            \end{smallmatrix}\right]g\right)\ |a|_p^{s-1}\ d^\times a\ d_{P_p\backslash G_p}g\\
            &= \int_{P_p\backslash G_p}\int_{\mathbf{Q}_p^\times} \int_{G_p} \mathcal{P}(\gamma) f(g\gamma)\ |\det(g_2)|_p^s\left(W_{\pi_{p,1}}^\mathrm{sph}\boxtimes W_{\pi_{p,2}}^\mathrm{sph}\right)\left(\left[\begin{smallmatrix}
                a & \\
                & 1
            \end{smallmatrix}\right]g\right)\ |a|_p^{s-1}\ d\gamma \ d^\times a\ d_{P_p\backslash G_p}g
        \end{align*}
        From \Cref{sec measures}, $d_{P_p\backslash G_p}g$ is right $G_p$-invariant, hence the above is equal to
        $$ \int_{P_p\backslash G_p}\int_{\mathbf{Q}_p^\times} \int_{G_p} \mathcal{P}(\gamma) f(g)\ |\det(g_2\gamma_2^{-2})|_p^s\left(W_{\pi_{p,1}}^\mathrm{sph}\boxtimes W_{\pi_{p,2}}^\mathrm{sph}\right)\left(\left[\begin{smallmatrix}
                a & \\
                & 1
            \end{smallmatrix}\right]g\gamma^{-1}\right)\ |a|_p^{s-1}\ d\gamma \ d^\times a\ d_{P_p\backslash G_p}g.$$
            Since $G_p$ is unimodular, the Haar measure $d\gamma$ is invariant under the change of variables $\gamma\rightsquigarrow \gamma^{-1}$. Thus, this is equal to
            $$\int_{P_p\backslash G_p}\int_{\mathbf{Q}_p^\times}  f(g)\ |\det(g_2)|_p^s\left(\int_{G_p}|\det(\gamma_2)|_p^s\ \mathcal{P}^{'}(\gamma)\left(W_{\pi_{p,1}}^\mathrm{sph}\boxtimes W_{\pi_{p,2}}^\mathrm{sph}\right)\left(\left[\begin{smallmatrix}
                a & \\
                & 1
        \end{smallmatrix}\right]g\gamma\right)\  d\gamma\right)|a|_p^{s-1}\ \ d^\times a\ d_{P_p\backslash G_p}g.$$
            The first part then follows upon using Cartan decomposition to express $\mathcal{P}$ as a linear combination of $G_p^\circ$ double coset operators. For the second part, using \Cref{lem unfolding}, we have
            \begin{align*}
                \Lambda_{\Pi_p}(s;f_0)&=\int_{Z_p}\int_{\GL_2(\mathbf{Q}_p)}\mathscr{W}_{\Pi_p}(s;(z\gamma,z), f_0)\ dz\ d\gamma.\\
                &=\int_{Z_p}\int_{\GL_2(\mathbf{Q}_p)}\int_{\mathbf{Q}_p^\times}f_0((z\gamma,z))|\det(z)|_p^s \left(W_{\pi_{p,1}}^\mathrm{sph}\boxtimes W_{\pi_{p,2}}^\mathrm{sph}\right)\left(\left[\begin{smallmatrix}
                a & \\
                & 1
        \end{smallmatrix}\right](z\gamma,z)\right)\ |a|_p^{s-1}\ d^\times a\ d\gamma\ dz.
            \end{align*}
            Now, $f_0((z\gamma,z))$ vanishes for $z$ outside $Z_p^\circ$, and is equal to $f_0((\gamma,1))$ for $z\in Z_p^\circ$. Similarly, $f_0((\gamma,1))$ vanishes for $\gamma$ outside $\GL_2(\mathbf{Z}_p)$ and is identically equal to one on $\GL_2(\mathbf{Z}_p)$. Thus, the triple integral above collapses to give
            $$\int_{\mathbf{Q}_p^\times} \left(W_{\pi_{p,1}}^\mathrm{sph}\boxtimes W_{\pi_{p,2}}^\mathrm{sph}\right)\left(\left[\begin{smallmatrix}
                a & \\
                & 1
        \end{smallmatrix}\right]\right)\ |a|_p^{s-1}\ d^\times a.$$
        It is well-known that the latter is given by $L(\Pi_p,s)L(\omega_{\Pi_p},2s)^{-1}$ and thus the proof of the second part is complete.
    \end{proof}
\end{prop}
\begin{rem}
    It is clear from \Cref{prop equivariance of linear form} that the linear form $\Lambda_{\pi_{p,1}\boxtimes \pi_{p,2}}$ is non-zero for $\omega_{\pi_{p,1}}\neq \omega_{\pi_{p,2}}^{-1}$, where $\omega_{\pi_{p,i}}$ denotes the central character of $\pi_i.$
\end{rem}
\begin{cor}
     The Hecke equivariant map $\Xi_c$ gives a canonical isomorphism onto its image, of free rank one $\mathcal{H}_{G_p}^\circ$-modules.
    \begin{proof}
        Suppose that some $\mathcal{P}\in\mathcal{H}_{G_p}^\circ$ annihilates $\ch(\mathbf{Z}_p^2)\otimes \ch(G_p^\circ)$ in $\mathcal{I}(G_p/G_p^\circ)$. By the Hecke equivariance of $\Xi_c$, $\mathcal{P}$ also annihilates $\Xi_c(\ch(\mathbf{Z}_p^2)\otimes \ch(G_p^\circ))=\ch(P_pG_p^\circ)$. Now, applying $\Lambda_{\Pi_p}$, and using \Cref{prop equivariance of linear form} part $(2)$, we obtain 
        \begin{align}
            \Theta_{\Pi_p}\left(\mathcal{P}^{'}(1-\mathcal{S}_p)\right)=0.
        \end{align}
        The equality in \eqref{eq:7} holds for every unramified $\Pi_p=\pi_{p,1}\boxtimes\pi_{p,2}$ with $\pi_{p,i}$ of Whittaker type. This family of representations is dense in $\mathrm{Spec}(\mathcal{H}_{G_p}^\circ)$ and thus $\mathcal{P}^{'}(1-\mathcal{S}_p)$ is identically zero in $\mathcal{H}_{G_p}^\circ$. From here it follows instantly that $\mathcal{P}=0$. This combined with \Cref{thm cyclicity on coinvariants}, shows that $\mathcal{I}(G_p/G_p^\circ)$ is free of rank one and that $\Xi_c$ is injective, completing the proof.
    \end{proof}
\end{cor}
\begin{defn}
     Let $\delta_0=\ch(\mathbf{Z}_p^2)\otimes \ch(G_p^\circ)$. For $\delta\in\mathcal{I}(G_p/G_p^\circ)$, the Hecke operator $\mathcal{P}_\delta$ is the unique element of $\mathcal{H}_{G_p}^\circ$ that satisfies $\mathcal{P}_\delta\cdot\delta_0=\delta$ in $\mathcal{I}(G_p/G_p^\circ)$. 
\end{defn}
\begin{rem}
    Equivalently, the Hecke operator $\mathcal{P}_\delta$ is the unique element of $\mathcal{H}_{G_p}^\circ$ for which $\mathcal{P}_\delta\cdot \ch(P_pG_p^\circ)$ is equal to $\Xi_c(\delta)$.
\end{rem}
\subsection{Explicit formulas for $\Lambda_{\Pi_p}$ in terms of $L$-factors and Satake parameters}\label{sec explicit formulae} Now that we have constructed the Hecke equivariant linear form $\Lambda_{\Pi_p}$, we want to directly compute the integrals involved in its definition, for general input data.
 \begin{defn}\label{def Euler polynomial factor}
The polynomial $\mathcal{P}_p\in\mathcal{H}_{G_p}^\circ[x]$, is characterized by the property 
    $$\Theta_{\Pi_p}(\mathcal{P}_p)(p^{-s})=L(\Pi_p,s)^{-1}$$
    for every unramified $G_p$-representation $\Pi_p=\pi_{p,1}\boxtimes \pi_{p,2}$,
    where $\Theta_{\Pi_p}$ denotes the spherical Hecke eigensystem of $\Pi_p$ and $L(\Pi_p,s)$ is the usual Rankin-Selberg local $L$-factor attached to such an unramified principal-series representation $\Pi_p$, (e.g. \cite{Jacquet1972GL2Part2}).
\end{defn}
\noindent We note that such a polynomial does exists as the local $L$-factor can be expressed as a Weyl-group invariant polynomial in the Satake parameters of $\Pi_p$.

\begin{defn} \label{def generalized L factors}
    Let $\Pi_p=\pi_{p,1}\boxtimes \pi_{p,2}$ with $\pi_{p,i}$ unramified $H_p$-representations of Whittaker type. Write $\alpha_{\pi_{p,i}}$ and $\beta_{\pi_{p,i}}$ for the Satake parameters of $\pi_{p,i}$. For $r\in\mathbf{Z}_{\geq 1}$, we define
    \begin{align*}
       \scalemath{0.9}{ L_1^{(r)}(\omega_{\Pi_p},s)^{-1}:=\frac{(\alpha_{\pi_{p,1}}^r-\beta_{\pi_{p,1}}^r)-\alpha_{\pi_{p,1}}\beta_{\pi_{p,1}}(\alpha_{\pi_{p,2}}+\beta_{\pi_{p,2}})(\alpha_{\pi_{p,1}}^{r-1}-\beta_{\pi_{p,1}}^{r-1})p^{-s}+(\alpha_{\pi_{p,2}}\beta_{\pi_{p,2}}\alpha_{\pi_{p,1}}^2\beta_{\pi_{p,1}}^2)(\alpha_{\pi_{p,1}}^{r-2}-\beta_{\pi_{p,1}}^{r-2})p^{-2s}}{\alpha_{\pi_{p,1}}-\beta_{\pi_{p,1}}}}
    \end{align*}
    \begin{align*}
       \scalemath{0.9}{ L_2^{(r)}(\omega_{\Pi_p},s)^{-1}:=\frac{(\alpha_{\pi_{p,2}}^r-\beta_{\pi_{p,2}}^r)-\alpha_{\pi_{p,2}}\beta_{\pi_{p,2}}(\alpha_{\pi_{p,1}}+\beta_{\pi_{p,1}})(\alpha_{\pi_{p,2}}^{r-1}-\beta_{\pi_{p,2}}^{r-1})p^{-s}+(\alpha_{\pi_{p,1}}\beta_{\pi_{p,1}}\alpha_{\pi_{p,2}}^2\beta_{\pi_{p,2}}^2)(\alpha_{\pi_{p,2}}^{r-2}-\beta_{\pi_{p,2}}^{r-2})p^{-2s}}{\alpha_{\pi_{p,2}}-\beta_{\pi_{p,2}}}}
    \end{align*}
\end{defn}
\begin{rem}\label{rem :4.3.4}
    The $L$-factors considered above are natural generalisations of $L(\omega_{\Pi_p},s)$ and will show up in our computation of $\Lambda_{\Pi_p}$. In fact, one has $L_1^{(1)}(\omega_{\Pi_p},s)=L_2^{(1)}(\omega_{\Pi_p},s)=L(\omega_{\Pi_p},2s)$.
\end{rem}

\begin{lem}
    For $r\in\mathbf{Z}_{\geq 1}$, there exist polynomials $\mathcal{P}_{p,1}^{(r)}(x)$ and $\mathcal{P}_{p,2}^{(r)}(x)$ in $\mathcal{H}_{G_p}^\circ[x]$ such that
    \begin{align*}
\Theta_{\Pi_p}\left(\mathcal{P}_{p,1}^{(r)}\right)(p^{-s})= L_1^{(r)}(\omega_{\Pi_p},s)^{-1}\ \mathrm{and}\ 
\Theta_{\Pi_p}\left(\mathcal{P}_{p,2}^{(r)}\right)(p^{-s})=L_2^{(r)}(\omega_{\Pi_p},s)^{-1}
    \end{align*}
    for every $\Pi_p$ as in \emph{\Cref{def generalized L factors}}.
    \begin{proof}
        The inverse $L$-factors $L_i^{(r)}(\omega_{\Pi_p},s)^{-1}$ are polynomials in $p^{-s}$ with coefficients that are themselves Weyl-group invariant polynomials in the Satake parameters of $\Pi_p$. The result then follows at once.
    \end{proof}
\end{lem}
\begin{lem}\label{lem two integrals}
   For $n\in\mathbf{Z}_{\geq 0}$ and with notation as above, we have 
\begin{align*}\int_{\mathbf{Q}_p^\times}\left(W_{\pi_{p,1}}^\mathrm{sph}\boxtimes W_{\pi_{p,2}}^\mathrm{sph}\right)\left(\left[\begin{smallmatrix}
        ap^n & \\
        & 1
\end{smallmatrix}\right],\left[\begin{smallmatrix}
        a & \\
        & 1
    \end{smallmatrix}\right]\right)\ |a|_p^{s-1}\ d^\times a &=L(\Pi_p,s)\cdot p^\frac{-n}{2}L_1^{(n+1)}(\omega_{\Pi_p},s)^{-1}\\
    \int_{\mathbf{Q}_p^\times}\left(W_{\pi_{p,1}}^\mathrm{sph}\boxtimes W_{\pi_{p,2}}^\mathrm{sph}\right)\left(\left[\begin{smallmatrix}
        a & \\
        & 1
\end{smallmatrix}\right],\left[\begin{smallmatrix}
        ap^n & \\
        & 1
    \end{smallmatrix}\right]\right)\ |a|_p^{s-1}\ d^\times a &=L(\Pi_p,s)\cdot p^\frac{-n}{2}L_2^{(n+1)}(\omega_{\Pi_p},s)^{-1}
    \end{align*}
    \begin{proof}
        The proof is a direct algebraic computation in the same style as \cite[Lemma $3.3.2$]{grossi2020norm}, once again again using \cite{shintani1976explicit} to express the values of $W_{\pi_{p,i}}^\mathrm{sph}(\left[\begin{smallmatrix}
            a & \\
            & 1
        \end{smallmatrix}\right])$ as polynomials in the Satake parameters.
    \end{proof}
\end{lem}

\noindent We now have all the tools needed to evaluate the linear form $\Lambda_{\Pi_p}$ on $f\in C_c^\infty(P_p\backslash G_p/G_p^\circ)$. By Iwasawa decomposition, we may clearly assume that $f$ is given by a characteristic function $\ch\left(P_pz_0(\gamma_0,1)G_p^\circ\right)$ with $z_0\in Z_p$ and $\gamma_0\in\GL_2(\mathbf{Q}_p)$ . 
\noindent We recall that $\mathcal{S}_{p,1}$ and $\mathcal{T}_{p,1}$ denote the spherical Hecke operators in $\mathcal{H}_{G_p}^\circ$, corresponding to the elements $\left(\left[\begin{smallmatrix}
    p & \\
     & p
\end{smallmatrix}\right],1\right)$ and $\left(\left[\begin{smallmatrix}
    p & \\
     & 1
\end{smallmatrix}\right],1\right)$ respectively. Similarly, $\mathcal{S}_{p,2}$ and $\mathcal{T}_{p,2}$ denote the Hecke operators corresponding to $\left(1,\left[\begin{smallmatrix}
    p & \\
     & p
\end{smallmatrix}\right]\right)$ and $\left(1,\left[\begin{smallmatrix}
    p & \\
     & 1
\end{smallmatrix}\right]\right)$ respectively. Additionally, $\mathcal{S}_p$ will once again denote $\mathcal{S}_{p,1}\mathcal{S}_{p,2}$. We write $\mathfrak{s}_n(x,y)$ for the Schur polynomial $(x^{n+1}-y^{n+1})/(x-y)$, and $\mathfrak{s}_n^\circ(X,Y)$ for the polynomial that satisfies $$\Theta_{\pi_{p,i}}\left(\mathfrak{s}_n^\circ(\mathcal{S}_{p,i},\mathcal{T}_{p,i})\right)=\mathfrak{s}_n\left(\alpha_{\pi_{p,i}},\beta_{\pi_{p,i}}\right),\ i\in\{1,2\}.$$Finally, we put $f_{z_0,\gamma_0}:=\ch\left(P_pz_0(\gamma_0,1)G_p^\circ\right)$ in $C_c^\infty(P_p\backslash G_p/G_p^\circ)$, with 
$$z_0=\left[\begin{smallmatrix}
    p^{r_0} & \\
     & p^{r_0}
\end{smallmatrix}\right]\in Z_p,\ \gamma_0=\left[\begin{smallmatrix}
    p^{r_1} & \\
     & p^{r_1}
\end{smallmatrix}\right]\left[\begin{smallmatrix}
    p^m &  \\
     & 1
\end{smallmatrix}\right]\left[\begin{smallmatrix}
    1 & y\\
    & 1
\end{smallmatrix}\right] k\in \GL_2(\mathbf{Q}_p)=Z_pP_p\GL_2(\mathbf{Z}_p)$$
\begin{prop}\label{prop explicit formula for Lambda}  With notation as above, we have
$$\Lambda_{\Pi_p}(f_{z_0,\gamma_0})=
    \Theta_{\Pi_p}\left( \mathscr{V}_{z_0,\gamma_0}\left[\mathscr{P}_{\gamma_0} + \mathscr{Q}_{\gamma_0}\mathcal{P}_p(1) \right] \right)
$$
where:
\begin{align*}
    \mathscr{V}_{z_0,\gamma_0}&:=\vol_{\GL_2(\mathbf{Q}_p)}(P_p^\circ \gamma_0 \GL_2(\mathbf{Z}_p))\cdot \mathcal{S}_p^{r_0}\mathcal{S}_{p,1}^{r_1}\\
    \mathscr{P}_{\gamma_0}&:=p^\frac{-m}{2}\begin{cases}
        \mathcal{P}_{p,1}^{(m+1)}(1),\ &\mathrm{if}\ m\geq 0\\
        \mathcal{P}_{p,2}^{(-m+1)}(1),\ &\mathrm{if}\ m<0
    \end{cases}\\
    \mathscr{Q}_{\gamma_0}&:=p^\frac{-m}{2}\begin{cases}
        \sum_{i=0}^{-v_p(y)-m-1}\varepsilon_i(yp^m)\mathfrak{s}_{i+m}^\circ(\mathcal{S}_{p,1},\mathcal{T}_{p,1})\mathfrak{s}_{i}^\circ(\mathcal{S}_{p,2},\mathcal{T}_{p,2}),\ &\mathrm{if}\ m\geq 0, v_p(y)<-m\\
        \sum_{i=0}^{-v_p(y)-1}\varepsilon_i(y)\mathfrak{s}_{i}^\circ(\mathcal{S}_{p,1},\mathcal{T}_{p,1})\mathfrak{s}_{i-m}^\circ(\mathcal{S}_{p,2},\mathcal{T}_{p,2}),\ &\mathrm{if}\ m< 0, v_p(y)<0\\
        0,\ &\mathrm{otherwise}.
    \end{cases}\\
    \varepsilon_i(w)&:=\begin{cases}
        -1,\ &\mathrm{if}\ 0\leq i<-v_p(w)-1\\
        \frac{-p}{p-1},\ &\mathrm{if}\ i=-v_p(w)-1.
    \end{cases}
\end{align*}
    \begin{proof}
        Using \Cref{lem unfolding}, we once again unfold $\Lambda_{\Pi_p}(s;f_{z_0,\gamma_0})$ as
        \begin{align}\label{eq: 14}
            \nonumber \Lambda_{\Pi_p}(s;f_{z_0,\gamma_0})&=\int_{Z_p}\int_{\GL_2(\mathbf{Q}_p)}\int_{\mathbf{Q}_p^\times} f_{z_0,\gamma_0}((z\gamma,z))\ |\det(z)|_p^s\ \left(W_{\pi_{p,1}}^\mathrm{sph}\boxtimes W_{\pi_{p,2}}^\mathrm{sph}\right)\left(\left[\begin{smallmatrix}
                a & \\
                 & 1
            \end{smallmatrix}\right](z\gamma,z)\right)\ |a|_p^{s-1}\ d^\times a\ d\gamma \ dz\\
             &=\omega_{\Pi_p}(z_0)\ |\det(z_0)|_p^s\ \int_{P_p^\circ \gamma_0 \GL_2(\mathbf{Z}_p)}\int_{\mathbf{Q}_p^\times} \left(W_{\pi_{p,1}}^\mathrm{sph}\boxtimes W_{\pi_{p,2}}^\mathrm{sph}\right)\left(\left[\begin{smallmatrix}
                a & \\
                 & 1
            \end{smallmatrix}\right](\gamma,1)\right)\ |a|_p^{s-1}\ d^\times a\ d\gamma
        \end{align}
          where the second equality follows from the following fact: As a function of $z\in Z_p$, $f_{z_0,\gamma_0}((z\gamma,z))$ is equal to $f_{1,\gamma_0}(\gamma,1)$ on $z_0Z_p^\circ$, and zero everywhere else. Similarly, as a function of $\gamma\in \GL_2(\mathbf{Q}_p)$, $f_{1,\gamma_0}(\gamma,1)$ is equal to $1$ on $P_p^\circ \gamma_0 \GL_2(\mathbf{Z}_p)$, and zero everywhere else. We now claim that the innermost integral in \eqref{eq: 14} is independent of $\gamma\in P_p^\circ \gamma_0 \GL_2(\mathbf{Z}_p)$ and is completely determined by its value on $\gamma_0$. Indeed, let $\gamma$ be an element of $P_p^\circ \gamma_0 \GL_2(\mathbf{Z}_p)$ written as $\left[\begin{smallmatrix}
            x_1 & x_2 \\
            & 1
        \end{smallmatrix}\right]\gamma_0 k$, where by definition, $x_1\in\mathbf{Z}_p^\times$ and $x_2\in \mathbf{Z}_p$. Then the innermost integral in \eqref{eq: 14} is given by 
        $$\int_{\mathbf{Q}_p^\times\cap \mathbf{Z}_p} \psi(ax_2) \left(W_{\pi_{p,1}}^\mathrm{sph}\boxtimes W_{\pi_{p,2}}^\mathrm{sph}\right)\left(\left[\begin{smallmatrix}
            ax_1 & \\
             & 1
        \end{smallmatrix}\right]\gamma_0,\left[\begin{smallmatrix}
            a & \\
             & 1
        \end{smallmatrix}\right]\right)\ |a|_p^{s-1}\ d^\times a.$$
        Using the fact that $\psi$ has conductor one, $x_2\in\mathbf{Z}_p$, $x_1\in\mathbf{Z}_p^\times$ and a change of variables $a\rightsquigarrow ax_1$, we see that the claim holds. We thus have
        \begin{align}\label{eq:15}
            \scalemath{0.95}{\Lambda_{\Pi_p}(s;f_{z_0,\gamma_0})=\omega_{\Pi_p}(z_0)\ |\det(z_0)|_p^s\ \vol_{\GL_2(\mathbf{Q}_p)}\left(P_p^\circ \gamma_0 \GL_2(\mathbf{Z}_p)\right)\int_{\mathbf{Q}_p^\times} \left(W_{\pi_{p,1}}^\mathrm{sph}\boxtimes W_{\pi_{p,2}}^\mathrm{sph}\right)\left(\left[\begin{smallmatrix}
                a & \\
                & 1
            \end{smallmatrix}\right]\gamma_0,\left[\begin{smallmatrix}
                a & \\
                & 1
            \end{smallmatrix}\right]\right)\ |a|_p^{s-1}\ d^\times a.}
        \end{align}
        The integral appearing in \eqref{eq:15} is treated in a very similar way to \cite[\S $6.3.1$]{groutides2024integral}, once again using \cite{shintani1976explicit} to evaluate the spherical Whittaker functions. Writting $C_{\gamma_0}:=\omega_{\Pi_p}\left(\left[\begin{smallmatrix}
            p^{r_1} & \\
             & p^{r_1}
        \end{smallmatrix}\right],1\right)$, this approach together with \Cref{lem two integrals}, allows us to split the integral into a ``main'' term of the form 
        $$\begin{cases}
            C_{\gamma_0}L(\Pi_p,s)\cdot p^\frac{-m}{2}L_1^{(m+1)}(\omega_{\Pi_p},s)^{-1},\ &\mathrm{if}\ m\geq 0\\
        \\
        C_{\gamma_0}L(\Pi_p,s)\cdot p^{\frac{-m}{2}+ms}L_2^{(-m+1)}(\omega_{\Pi_p},s)^{-1},\ &\mathrm{if}\ m<0
        \end{cases}$$
        and an``error'' term of the form
        $$\scalemath{1}{\begin{cases}
            C_{\gamma_0}\sum_{i=0}^{-v_p(y)-m-1}\left(\int_{p^i\mathbf{Z}_p^\times}\psi(ayp^m)-1\ d^\times a\right)\cdot p^\frac{-m}{2}\mathfrak{s}_{i+m}(\alpha_{\pi_{p,1}},\beta_{\pi_{p,1}})\mathfrak{s}_{i}(\alpha_{\pi_{p,2}},\beta_{\pi_{p,2}})p^{-is},\ &\mathrm{if}\ m\geq 0\\
        \\
            C_{\gamma_0}\sum_{i=0}^{-v_p(y)-1}\left(\int_{p^i\mathbf{Z}_p^\times}\psi(ay)-1\ d^\times a\right)\cdot p^{\frac{-m}{2}+ms}\mathfrak{s}_{i}(\alpha_{\pi_{p,1}},\beta_{\pi_{p,1}})\mathfrak{s}_{-m+i}(\alpha_{\pi_{p,2}},\beta_{\pi_{p,2}})p^{-is},\ &\mathrm{if}\ m<0.
        \end{cases}}$$
        All it remains to do is to use the expression found in \cite[\S $6.3.1$]{groutides2024integral} for the integrals involving $\psi$, combine these two terms with \eqref{eq:15}, apply $\lim_{s\rightarrow0}\frac{(-)}{L(\Pi_p,s)}$ and pass to the spherical Hecke algebra through $\Theta_{\Pi_p}$. 
    \end{proof}
\end{prop}

\begin{lem}\label{lem integrality of operators appearing in explicit formula}
    With notation as in \emph{\Cref{prop explicit formula for Lambda}}, the elements $\mathscr{V}_{z_0,\gamma_0}\left[\mathscr{P}_{\gamma_0}+\mathscr{Q}_{\gamma_0}\mathcal{P}_p(1)\right]$ are contained in $\mathcal{H}_{G_p}^\circ(\mathbf{Z}[1/p])$ for any $f_{z_0,\gamma_0}$.
    \begin{proof}
         There is a number of things we need to check. Firstly, we need to make sure that the operators $\mathscr{P}_{\gamma_0}$ and $\mathscr{Q}_{\gamma_0}$ involve no odd powers of $p^\frac{-1}{2}$. Indeed, this follows from a simple parity consideration on $m$ (and $i$ for $\mathscr{Q}_{\gamma_0}$), using the expansion of the polynomials $\mathfrak{s}_n(x,y)$ found in \cite[\S $6.3$]{groutides2024integral}, together with the well-known fact that $\Theta_{\Pi_p}(\mathcal{S}_{p,i})=\alpha_{\pi_{p,i}}\beta_{\pi_{p,i}}$ and $\Theta_{\Pi_p}(\mathcal{T}_{p,i})=p^\frac{1}{2}(\alpha_{\pi_{p,i}}+\beta_{\pi_{p,i}})$. Hence $\mathscr{P}_{\gamma_0}$ and $\mathscr{Q}_{\gamma_0}$ are elements of $\mathcal{H}_{G_p}^\circ(\mathbf{Z}[1/p])$ and $\mathcal{H}_{G_p}^\circ(\frac{1}{p-1}\mathbf{Z}[1/p])$ respectively. It follows from a direct computation that the Euler factor $\mathcal{P}_p(1)$ is also an element of $\mathcal{H}_{G_p}^\circ(\mathbf{Z}[1/p])$. We now claim that for every $\gamma_0$,
    $$\vol_{\GL_2(\mathbf{Q}_p)}(P_p^\circ \gamma_0 \GL_2(\mathbf{Z}_p))\mathscr{Q}_{\gamma_0}$$
    is an element of $\mathcal{H}_{G_p}^\circ(\mathbf{Z}[1/p])$. By Cartan decomposition, $\gamma_0$ is contained in $z_1\GL_2(\mathbf{Z}_p)\left[\begin{smallmatrix}
        p^\lambda & \\
        & 1
    \end{smallmatrix}\right]\GL_2(\mathbf{Z}_p)$, with $z_1\in Z_p$ and $\lambda\in\mathbf{Z}_{\geq 0}$, uniquely determined by $\gamma_0$. We now recall the decomposition \begin{align}
        \scalemath{0.9}{\GL_2(\mathbf{Z}_p) \left[\begin{smallmatrix}
            p^\lambda & \\
            & 1
        \end{smallmatrix}\right] \GL_2(\mathbf{Z}_p)=\left(\bigsqcup_{\beta\in \mathbf{Z}/p^{\lambda}\mathbf{Z}}\left[\begin{smallmatrix}
            p^{\lambda} & \beta\\
            & 1
        \end{smallmatrix}\right] \GL_2(\mathbf{Z}_p)\right)  \sqcup \left(\bigsqcup_{\substack{0<i<\lambda\\ \\ \beta\in(\mathbf{Z}/p^{i}\mathbf{Z})^\times}}\left[\begin{smallmatrix}
            p^{i} & \beta\\
            &p^{\lambda-i}
        \end{smallmatrix}\right] \GL_2(\mathbf{Z}_p)\right)\sqcup\left[\begin{smallmatrix}
            1 & \\
            &p^{\lambda}
        \end{smallmatrix}\right] \GL_2(\mathbf{Z}_p).}
    \end{align}
    If $\gamma_0$ is contained in $z_1\left[\begin{smallmatrix}
            p^\lambda & \beta_0\\
            & 1
        \end{smallmatrix}\right] \GL_2(\mathbf{Z}_p)=z_1\left[\begin{smallmatrix}
            p^\lambda & \\
            & 1
        \end{smallmatrix}\right]\left[\begin{smallmatrix}
            1 & \beta_0p^{-\lambda}\\
            & 1
        \end{smallmatrix}\right]\GL_2(\mathbf{Z}_p)$ or $z_1 \left[\begin{smallmatrix}
            1 & \\
            &p^{\lambda}
        \end{smallmatrix}\right] \GL_2(\mathbf{Z}_p)$, then by the definition of $\mathscr{D}_{\gamma_0}$ in \Cref{prop explicit formula for Lambda}, we see that $\mathscr{Q}_{\gamma_0}=0$ hence the claim holds trivially. If on the other hand $\gamma_0$ is contained in $z_1 \left[\begin{smallmatrix}
            p^{i} & \beta_0\\
            &p^{\lambda-i}
        \end{smallmatrix}\right] \GL_2(\mathbf{Z}_p)$ with $\beta_0$ in $(\mathbf{Z}/p^i\mathbf{Z})^\times$, then $\gamma_0=z_1\left[\begin{smallmatrix}
            p^{\lambda-i} & \\
             & p^{\lambda-i}
        \end{smallmatrix}\right]\left[\begin{smallmatrix}
            p^{2i-\lambda} &  \\
            & 1
        \end{smallmatrix}\right]\left[\begin{smallmatrix}
            1 & \beta_0 p^{-i} \\
            & 1
        \end{smallmatrix}\right]k$ with $i>0$ and $k\in \GL_2(\mathbf{Z}_p)$. Thus in this case, \Cref{prop explicit formula for Lambda} shows that $\mathscr{Q}_{\gamma_0}$ is contained in $\mathcal{H}_{G_p}^\circ(\frac{1}{p-1}\mathbf{Z}[1/p])$. However in this case, it is straightforward to see that
        $$P_p^\circ\gamma_0 \GL_2(\mathbf{Z}_p)=z_1P_p^\circ\left[\begin{smallmatrix}
            p^{i} & \beta_0 \\
            & p^{\lambda-i}
        \end{smallmatrix}\right] \GL_2(\mathbf{Z}_p)=\bigsqcup_{\beta\in (\mathbf{Z}/p^i\mathbf{Z})^\times}z_1\left[\begin{smallmatrix}
            p^i & \beta \\
            & p^{\lambda-i}
        \end{smallmatrix}\right]\GL_2(\mathbf{Z}_p).$$
        Thus $\vol_{\GL_2(\mathbf{Q}_p)}(P_p^\circ \gamma_0 \GL_2(\mathbf{Z}_p))=\#(\mathbf{Z}/p^i\mathbf{Z})^\times$, which is a multiple of $p-1$.
    \end{proof}
\end{lem}
\begin{cor}\label{cor integrality for lambda}
    Let $\Pi_p=\pi_{p,1}\boxtimes \pi_{p,2}$ with $\pi_{p,i}$ unramified $H_p$-representations of Whittaker type with $\omega_{\pi_{p,1}}\neq \omega_{\pi_{p,2}}^{-1}$. Suppose that $\Theta_{\Pi_p}$ restricts to a morphism $\mathcal{H}_{G_p}^\circ(\mathcal{R})\rightarrow \mathcal{R}$ for some $\mathbf{Z}[1/p]$-algebra $\mathcal{R}\subseteq\mathbf{C}$. Then $\Lambda_{\Pi_p}$ induces a non-zero Hecke equivariant morphism
    $$\Lambda_{\Pi_p}:C_c^\infty(P_p\backslash G_p/G_p^\circ,\mathcal{R})\longrightarrow\mathcal{R}\cdot W_{\Pi_p^\vee}^\mathrm{sph}.$$
    \begin{proof}
       Hecke equivariance is part of \Cref{prop equivariance of linear form} using the assumption on $\Pi_p$.
    The result then follows at once using \Cref{lem integrality of operators appearing in explicit formula} and the assumption on $\Theta_{\Pi_p}$.
    \end{proof}
\end{cor}
\subsection{Inverting $1-\mathcal{S}_p$}\label{sec inverting} The Hecke action by $1-\mathcal{S}_p$ induces an injective endomorphism of $C_c^\infty(P_p\backslash G_p/G_p^\circ)$. Thus, the localisation $C_c^\infty(P_p\backslash G_p/G_p^\circ)_{1-\mathcal{S}_p}$ is canonically isomorphic to \[\begin{tikzcd}[ampersand replacement=\&]
	{C_c^\infty(P_p\backslash G_p/G_p^\circ)[\tfrac{1}{1-\mathcal{S}_p}]:=\mathrm{coker}(C_c^\infty(P_p\backslash G_p/G_p^\circ)[x]} \&\& {C_c^\infty(P_p\backslash G_p/G_p^\circ)[x])}
	\arrow["{1-(1-\mathcal{S}_p)x}", from=1-1, to=1-3]
\end{tikzcd}\]
where $C_c^\infty(P_p\backslash G_p/G_p^\circ)[x]$ denotes the natural $\mathcal{H}_{G_p}^\circ[x]$-module of formal polynomials in a single variable $x$ with coefficients in $C_c^\infty(P_p\backslash G_p/G_p^\circ)$, or equivalently the tensor product $\mathcal{H}_{G_p}^\circ[x]\otimes_{\mathcal{H}_{G_p}^\circ} C_c^\infty(P_p\backslash G_p/G_p^\circ)$. For this reason, it is clear that $C_c^\infty(P_p\backslash G_p/G_p^\circ)$ embeds into its localisation $C_c^\infty(P_p\backslash G_p/G_p^\circ)[\tfrac{1}{1-\mathcal{S}_p}]$. As usual, we consider representations $\Pi_p=\pi_{p,1}\boxtimes\pi_{p,2}$ where $\pi_{p,i}$ are unramified of Whittaker type. From \Cref{prop explicit formula for Lambda}, and the density of this family of representations in $\mathrm{Spec}(\mathcal{H}_{G_p}^\circ)$, we can deduce that for each function $f$ in $C_c^\infty(P_p\backslash G_p/G_p^\circ)$, there exists a unique Hecke operator which we call $\mathcal{Q}_f$ in $\mathcal{H}_{G_p}^\circ$, for which $\Lambda_{\Pi_p}(f)=\Theta_{\Pi_p}(\mathcal{Q}_f)$ for every such $\Pi_p$. \Cref{prop equivariance of linear form} then implies that \begin{align}\label{eq: inverting}\Lambda_{\Pi_p}((1-\mathcal{S}_p)\cdot f)=\Theta_{\Pi_p^\vee}(1-\mathcal{S}_p)\Lambda_{\Pi_p}(f)=\Theta_{\Pi_p}(-\mathcal{S}_p^{-1}(1-\mathcal{S}_p)\mathcal{Q}_f))=\Lambda_{\Pi_p}(-\mathcal{S}_p\mathcal{Q}_f^{'}\cdot f_0).\end{align}
\begin{thm}\label{thm freeness of P-invariant vectors}
    The module $C_c^\infty(P_p\backslash G_p/G_p^\circ)[\tfrac{1}{1-\mathcal{S}_p}]$ is free of rank one over $\mathcal{H}_{G_p}^\circ[\tfrac{1}{1-\mathcal{S}_p}]$, generated by the characteristic function $f_0:=\ch(P_pG_p^\circ)$. In particular, for $f\in C_c^\infty(P_p\backslash G_p/G_p^\circ)$, we have $f=\tfrac{-\mathcal{S}_p\mathcal{Q}_f^{'}}{1-\mathcal{S}_p}\cdot f_0$, where $\mathcal{Q}_f$ is the Hecke operator described above.
    \begin{proof} The proof is long and quite technical, so we give it in \Cref{appendix} in order to ease exposition.
    \end{proof}
\end{thm}
\begin{rem}
    \Cref{thm freeness of P-invariant vectors}, strictly speaking, is not necessary for our proof of optimal integral norm relations, however it definitely makes it cleaner and more complete. In addition, it showcases the obstruction regarding the invertibility of the operators $1-\mathcal{S}_p$. In fact, later on we will see that \Cref{thm freeness of P-invariant vectors} is simply not true without the localisation. We also like to stress the fact that \Cref{thm freeness of P-invariant vectors} by no means replaces the need for \Cref{thm cyclicity on coinvariants} and the two results are not equivalent. 
\end{rem}
\begin{cor}\label{cor freeness at integral level for P-invariants}
    The module $C_c^\infty(P_p\backslash G_p/G_p^\circ,\mathbf{Z}[1/p])[\frac{1}{1-\mathcal{S}_p}]$ is free of rank one over $\mathcal{H}_{G_p}^\circ(\mathbf{Z}[1/p])[\frac{1}{1-\mathcal{S}_p}]$ generated by $\ch(P_pG_p^\circ)$.
    \begin{proof}
        This follows from \eqref{eq: inverting}, \Cref{lem integrality of operators appearing in explicit formula} and \Cref{thm freeness of P-invariant vectors}.
    \end{proof}
\end{cor}
\begin{defn}
    Let $f_0=\ch(P_pG_p^\circ)$. For $f\in C_c^\infty(P_p\backslash G_p/G_p^\circ)[\frac{1}{1-\mathcal{S}_p}]$, the Hecke operator $\mathcal{P}_{f,\mathrm{loc}}$ is the unique element of $\mathcal{H}_{G_p}^\circ[\frac{1}{1-\mathcal{S}_p}]$ that satisfies $\mathcal{P}_{f,\mathrm{loc}}\cdot f_0=f.$
\end{defn}
\noindent Such a unique Hecke operator clearly exists by \Cref{thm freeness of P-invariant vectors}. We put 
\begin{align}\label{eq:111}
\scalemath{0.95}{C_c^{\infty,1}(P_p\backslash G_p/G_p^\circ,\mathbf{Z}[1/p])=\left\{f\in C_c^\infty(P_p\backslash G_p/G_p^\circ,\mathbf{Z}[1/p])\ |\ f\ \text{is valued in}\ \begin{cases}
    (p-1)\mathbf{Z}[1/p],\ \text{on}\ Z_p(T_p\times \{1\}) \\
    \mathbf{Z}[1/p],\ \text{everywhere else.}
\end{cases}\right\}}\end{align}
\noindent We also consider a non-compactly supported version of this which we denote by $C^{\infty,1}(P_p\backslash G_p/G_p^\circ,\mathbf{Z}[1/p])$ and is defined in the same way, simply by dropping the compact support condition.\\
\\
We write $C_c^{\infty,1}(P_p\backslash G_p/G_p^\circ,\mathbf{Z}[1/p])[\frac{1}{1-\mathcal{S}_p}]$ for the $\mathcal{H}_{G_p}^\circ(\mathbf{Z}[1/p])[\frac{1}{1-\mathcal{S}_p}]$-submodule generated by the lattice $C_c^{\infty,1}(P_p\backslash G_p/G_p^\circ,\mathbf{Z}[1/p])$ in $C_c^\infty(P_p\backslash G_p/G_p^\circ,\mathbf{Z}[1/p])[\frac{1}{1-\mathcal{S}_p}]$. We now obtain the following main integrality result, on the $P_p$-invariant side of things.

\begin{thm}\label{thm conj 1 for P-invariant}
Let $f$ be an element of $C_c^{\infty,1}(P_p\backslash G_p/G_p^\circ,\mathbf{Z}[1/p])[\tfrac{1}{1-\mathcal{S}_p}]$. Then $f=\mathcal{P}_{f,\mathrm{loc}}\cdot f_0$ and the Hecke operator $\mathcal{P}_{f,\mathrm{loc}}$ is contained in the ideal
$$\mathfrak{h}_p=\left(p-1,\mathcal{P}_p(1)^{'}\right)\subseteq \mathcal{H}_{G_p}^\circ(\mathbf{Z}[1/p])[\tfrac{1}{1-\mathcal{S}_p}].$$
If moreover $\tfrac{f}{1-\mathcal{S}_p}$ is an element of $ C_c^{\infty,1}(P_p\backslash G_p/G_p^\circ,\mathbf{Z}[1/p])$, then the containment holds without inverting $1-\mathcal{S}_p$.
\begin{proof}
    We may assume that $f\in C_c^{\infty,1}(P_p\backslash G_p/G_p^\circ,\mathbf{Z}[1/p])$ and that is is given by a characteristic function $c_{z,\gamma}\ch(P_pz(\gamma,1) G_p^\circ)$ with $z\in Z_p$ and $\gamma\in H_p$. Then $\mathcal{P}_{f,\mathrm{loc}}=c_{z,\gamma}\tfrac{-\mathcal{S}_p\mathcal{Q}_f^{'}}{1-\mathcal{S}_p}$ where $\mathcal{Q}_f=\mathscr{V}_{z,\gamma}\left[\mathscr{P}_{\gamma} + \mathscr{Q}_{\gamma}\mathcal{P}_p(1) \right]$.  We now proceed in the same style as in the final step of the proof of \Cref{lem integrality of operators appearing in explicit formula}. If $\gamma$ is contained in $y\left[\begin{smallmatrix}
        p^\lambda &\beta\\
        & 1
    \end{smallmatrix}\right] \GL_2(\mathbf{Z}_p)$ with $y\in Z_p$, $\lambda\geq 0$ and $\beta\in \mathbf{Z}/p^\lambda\mathbf{Z}$, then $\mathscr{Q}_{\gamma}=0$ and $\ch(P_pz(\gamma,1)G_p^\circ)$ coincides with $\ch(P_pz(y\left[\begin{smallmatrix}
        p^\lambda & \\
         & 1
    \end{smallmatrix}\right],1) G_p^\circ)$. Thus, by the assumption on the values of $f$, we must have $c_{z,\gamma}\in(p-1)\mathbf{Z}[1/p]$ and $\mathcal{Q}_f\in(p-1)\mathcal{H}_{G_p}^\circ(\mathbf{Z}[1/p])$. If $\gamma$ is contained in $y\left[\begin{smallmatrix}
        1 & \\
        & p^{\lambda}
    \end{smallmatrix}\right]\GL_2(\mathbf{Z}_p)$ for some $y\in Z_p$ and $\lambda\geq 0$, then the situation is practically the same.  Finally if $\gamma$ is contained in $y\left[\begin{smallmatrix}
        p^i & \beta \\
        & p^{\lambda-i}
    \end{smallmatrix}\right]\GL_2(\mathbf{Z}_p)$ for some $y\in Z_p$, $ 0<i<\lambda, \lambda>0$ and $\beta\in(\mathbf{Z}/p^i\mathbf{Z})^\times$, then as in the last part of the proof of \Cref{lem integrality of operators appearing in explicit formula}, using that in this case $\vol_{\GL_2(\mathbf{Q}_p)}(P_p^\circ \gamma \GL_2(\mathbf{Z}_p))$ is a multiple of $p-1$, we deduce that $\mathscr{V}_{z,\gamma}\left[\mathscr{P}_{\gamma} + \mathscr{Q}_{\gamma}\mathcal{P}_p(1) \right]$ is an element of $\mathfrak{h}_p\subseteq\mathcal{H}_{G_p}^\circ(\mathbf{Z}[1/p])$ as required. Moreover, if $f=(1-\mathcal{S}_p)\cdot f_0$ with $f_0\in C_c^{\infty,1}(P_p\backslash G_p/G_p^\circ,\mathbf{Z}[1/p])$ as in the statement of the theorem, then by freeness, $\mathcal{P}_{f,\mathrm{loc}}$ is given by $(1-\mathcal{S}_p)\cdot \tfrac{-\mathcal{S}_p\mathcal{Q}_{f_0}^{'}}{1-\mathcal{S}_p}$ and thus we are done.
\end{proof}
\end{thm}

\section{Euler systems for $\GL_2\times \GL_2$}\label{sec euler systems for gl2 x gl2}
We now finally have all the necessary tools in order to state and prove our optimal abstract norm relations at a local integral level and use them in order to obtain \Cref{thm intro A} of the introduction.
\subsection{Local abstract integral norm relations}\label{sec local norm relations} Let $V$ be a smooth $G_p$-representation and let $\mathfrak{Z}:\mathcal{S}(\mathbf{Q}_p^2)\otimes \mathcal{H}(G_p)\rightarrow V$ be \textit{any} $(G_p\times H_p)$-equivariant map as in \Cref{def equiv maps}. Consider the induced diagram
    \[\begin{tikzcd}[ampersand replacement=\&,cramped,sep=small]
	{\mathcal{I}(G_p/G_p^\circ[p],\mathbf{Z}[1/p])} \\
	{\mathcal{I}(G_p/G_p^\circ,\mathbf{Z}[1/p])} \& {V.}
	\arrow["{\mathrm{Tr}}"', from=1-1, to=2-1]
	\arrow["\mathfrak{Z}", from=1-1, to=2-2]
	\arrow["\mathfrak{Z}"', from=2-1, to=2-2]
\end{tikzcd}\]
It follows straight from the definitions that the horizontal map is equivariant with respect to $\mathcal{H}_{G_p}^\circ$. Furthermore, for $\delta\in \mathcal{I}(G_p/G_p^\circ[p],\mathbf{Z}[1/p])$, we have (using the ``norm'' notation of \cite{Loeffler_2021})
$$\mathfrak{Z}\left(\mathrm{Tr}\left(\delta\right)\right)=\mathrm{norm}^{G_p^\circ[p]}_{G_p^\circ}\left(\mathfrak{Z}\left(\delta\right)\right):=\sum_{\gamma\in G_p^\circ/G_p^\circ[p]}\gamma\cdot \mathfrak{Z}\left(\delta\right).$$
Finally, we write $\delta_0$ for the unramified vector $\ch(\mathbf{Z}_p^2)\otimes \ch(G_p^\circ)$, which implies that $\mathrm{norm}^{G_p^\circ[p]}_{G_p^\circ}\left(\mathfrak{Z}\left(\delta_0\right)\right)$ is nothing more than $(p-1)\cdot \mathfrak{Z}(\delta_0).$ Recall that $\mathcal{S}_0(\mathbf{Q}_p^2)$ denotes the space of Schwartz functions that vanish at the origin and $\mathcal{I}_0(G_p/G_p^\circ)$ denotes the submodule of $\mathcal{I}(G_p/G_p^\circ)$ given by the image of $\mathcal{S}_0(\mathbf{Q}_p^2)\otimes C_c^\infty(G_p/G_p^\circ)$ down to the $H_p$-coinvariants. Similarly, recall the definitions of $\mathcal{I}_0(G_p/G_p^\circ,\mathbf{Z}[1/p])$ and $\mathcal{I}_0(G_p/G_p^\circ[p],\mathbf{Z}[1/p])$ given in \Cref{sec group schemes eq maps}
\begin{thm}[Integral abstract norm relations]\label{thm local norm relations}
    Let $V$ and $\mathfrak{Z}$ be as above. Then:
    \begin{enumerate}
        \item For $\delta\in\mathcal{I}(G_p/G_p^\circ,\mathbf{Z}[1/p])$, we have $\mathcal{P}_\delta\in\mathcal{H}_{G_p}^\circ(\mathbf{Z}[1/p])$ and $\mathfrak{Z}(\delta)=\mathcal{P}_\delta\cdot \mathfrak{Z}(\delta_0)$.
        \item For $\delta\in\mathcal{I}_0(G_p/G_p^\circ[p],\mathbf{Z}[1/p])$, we have $$ \mathcal{P}_{\mathrm{Tr}(\delta)}\in\left(p-1,\mathcal{P}_p(1)^{'}\right)\subseteq \mathcal{H}_{G_p}^\circ(\mathbf{Z}[1/p])\ \ \text{and}\ \ \mathrm{norm}^{G_p^\circ[p]}_{G_p^\circ}\left(\mathfrak{Z}(\delta)\right)=\mathcal{P}_{\mathrm{Tr}(\delta)}\cdot \mathfrak{Z}(\delta_0).$$
        \item For $\delta\in\mathcal{I}(G_p/G_p^\circ[p],\mathbf{Z}[1/p])$, we have $\mathcal{P}_{\mathrm{Tr}(\delta)}\in\left(p-1,\mathcal{P}_p^{'}(1)\right)\subseteq\mathcal{H}_{G_p}^\circ(\mathbf{Z}[1/p])[\tfrac{1}{1-\mathcal{S}_p}]$ and $\mathcal{P}_{\mathrm{Tr}(\delta)}$ has degree at most one as a polynomial in $\tfrac{1}{1-\mathcal{S}_p}$. If furthermore $\scalemath{0.9}{1-\mathcal{S}_p}$ acts invertibly as a linear endomorphism of $V^{G_p^\circ}$, then the norm relation of part two is also satisfied.
    \end{enumerate}
    \end{thm}
\noindent This theorem can be interpreted as the local integral essence of the tame norm relations in motivic cohomology
for integral test data, as presented later on in \Cref{thm euler system norm relations}
    \begin{proof}
        Let $\delta\in\mathcal{I}(G_p/G_p^\circ,\mathbf{Z}[1/p])$. By \Cref{prop integrally from coinvariants to C^infinity} and \Cref{thm cyclicity on coinvariants}, $$\mathcal{P}_\delta\cdot \ch(P_pG_p^\circ)=\Xi_c(\delta)\in C_c^\infty(P_p\backslash G_p/G_p^\circ,\mathbf{Z}[1/p]).$$By \Cref{cor freeness at integral level for P-invariants}, we deduce that $\mathcal{P}_\delta=\mathcal{P}_{\Xi_c(\delta),\mathrm{loc}}\in\mathcal{H}_{G_p}^\circ(\mathbf{Z}[1/p])[\tfrac{1}{1-\mathcal{S}_p}]\cap \mathcal{H}_{G_p}^\circ$, and thus in particular that $\mathcal{P}_\delta\in\mathcal{H}_{G_p}^\circ(\mathbf{Z}[1/p])$. This proves the first part when combined with Hecke equivariance of $\mathfrak{Z}$. For the other two parts, we need some more work.
        One checks that we obtain a commutative diagram
     \[\begin{tikzcd}[column sep=small,row sep=scriptsize]
	{\mathcal{I}(G_p/G_p^\circ[p])} & {C^\infty(P_p\backslash G_p/G_p^\circ[p])} & \phi\otimes\xi & {\int_{H_p}\xi(h(-))f_\phi(h)\ dh} \\
	{\mathcal{I}(G_p/G_p^\circ)} & {C^\infty(P_p\backslash G_p/G_p^\circ)} & {\phi\otimes\sum_{\gamma\in G_p^\circ/G_p^\circ[p]}\xi((-)\gamma)} & {\int_{H_p}\sum_{\gamma\in G_p^\circ/G_p^\circ[p]}\xi(h(-)\gamma)f_\phi(h)\ dh}
	\arrow["{{\Xi[p]}}", from=1-1, to=1-2]
	\arrow["{{\mathrm{Tr}}}"', from=1-1, to=2-1]
	\arrow["{{\mathrm{Tr}}}"', from=1-2, to=2-2]
	\arrow[maps to, from=1-3, to=1-4]
	\arrow[maps to, from=1-3, to=2-3]
	\arrow[from=1-4, to=2-4]
	\arrow["\Xi", from=2-1, to=2-2]
	\arrow[maps to, from=2-3, to=2-4]
\end{tikzcd}\]
where $\Xi[p]$ is defined in the same way as $\Xi$, simply by replacing $G_p^\circ$ with $G_p^\circ[p]$, and the trace maps are the usual ones with respect to $G_p^\circ[p]\subseteq G_p^\circ$. We now take $\delta\in\mathcal{I}(G_p/G_p^\circ[p],\mathbf{Z}[1/p])$. The proof of \Cref{prop integrally from coinvariants to C^infinity} goes through in the same way at level $G_p^\circ[p]$ and thus $\Xi[p](\delta)\in C^\infty(P_p\backslash G_p/G_p^\circ[p],\mathbf{Z}[1/p])$. It is easy to check that for an element in $C^\infty(P_p\backslash G_p/G_p^\circ[p],\mathbf{Z}[1/p])$, its image under the trace map to $C^\infty(P_p\backslash G_p/G_p^\circ,\mathbf{Z}[1/p])$ is valued in $(p-1)\mathbf{Z}[1/p]$ on $Z_p(T_p\times \{1\})$. This is because for $g\in Z_p(T_p\times \{1\})$, $P_pgG_p^\circ[p]=P_pgG_p^\circ$ and hence $\mathrm{Tr}$ is simply multiplication by $[G_p^\circ:G_p^\circ[p]]=p-1$ on such a characteristic function. Thus $\mathrm{Tr}(\Xi[p](\delta))$ is contained in $C^{\infty,1}(P_p\backslash G_p/G_p^\circ,\mathbf{Z}[1/p])$. By commutativity of the diagram above, we deduce that $\Xi(\mathrm{Tr}(\delta))$ belongs in $ C^{\infty,1}(P_p\backslash G_p/G_p^\circ,\mathbf{Z}[1/p])$ and hence $\Xi_c(\mathrm{Tr}(\delta))$ belongs in $C_c^{\infty,1}(P_p\backslash G_p/G_p^\circ,\mathbf{Z}[1/p])$. Thus, by \Cref{thm freeness of P-invariant vectors} and \Cref{thm conj 1 for P-invariant}, $\mathcal{P}_{\mathrm{Tr}(\delta)}=\mathcal{P}_{\Xi_c(\mathrm{Tr}(\delta)),\mathrm{loc}}$ is an element of the ideal $\left(p-1,\mathcal{P}_p(1)^{'}\right)\subseteq\mathcal{H}_{G_p}^\circ(\mathbf{Z}[1/p])[\frac{1}{1-\mathcal{S}_p}]$ and has degree at most one as a polynomial in $\frac{1}{1-\mathcal{S}_p}.$ This proves the third part of the theorem. For the second part, if $\delta\in\mathcal{I}_0(G_p/G_p^\circ[p],\mathbf{Z}[1/p])$ then $\mathrm{Tr}(\delta)\in\mathcal{I}_0(G_p/G_p^\circ,\mathbf{Z}[1/p]) $. We have already remarked that for an element in the latter, its image under $\Xi$ is already compactly supported. Hence $\Xi_c(\mathrm{Tr}(\delta))$ is divisible by $\scalemath{0.9}{1-\mathcal{S}_p}$ in $C_c^{\infty,1}(P_p\backslash G_p/G_p^\circ,\mathbf{Z}[1/p])$. The second part of the theorem now follows from the ``moreover'' part of \Cref{thm conj 1 for P-invariant}.
    \end{proof}
\subsection{Cohomology classes and norm compatibility} We now  showcase how we may use our representation theoretic results to obtain norm-compatible classes in the cohomology of the infinite level Shimura variety $Y_G$. This approach was introduced in \cite{Loeffler_2021} and thus we draw heavily from \textit{op.cit}. 
\subsubsection{Modular varieties and motivic cohomology}
We follow the setup and conventions of \cite{Loeffler_Skinner_Zerbes_2021} and we write $Y_{\GL_2}=\varprojlim_U Y_{\mathrm{GL}_2}(U)$ for the infinite level modular (Shimura) curve regarded as a pro-variety over $\mathbf{Q}$ with a right action of $\GL_2(\mathbf{A}_f)$, whose $\mathbf{C}$-points are $\GL_2^+(\mathbf{Q})\backslash (\GL_2(\mathbf{A}_f)\times \mathbf{H})$ (as usual $\mathbf{H}$ denotes the upper-half plane). Similarly, we define $Y_G$ as $Y_{\GL_2}\times Y_{\GL_2}$ which is the Shimura variety associated to $G$. The compatibility of our Shimura data, induces a morphism of $\mathbf{Q}$-varieties 
\[\begin{tikzcd}[sep=scriptsize]
	{Y_{\mathrm{GL}_2}(\mathrm{GL}_2(\mathbf{A}_f)\cap gUg^{-1})} && {Y_G(U)} \\
	&& {Y_G(gUg^{-1})}
	\arrow["{\iota_{gU}}", dotted, from=1-1, to=1-3]
	\arrow[from=1-1, to=2-3]
	\arrow["\simeq"', from=2-3, to=1-3]
\end{tikzcd}\]
for any $U\subseteq G(\mathbf{A}_f)$ open compact and $g\in G(\mathbf{A}_f)/U$. If we temporarily let $\Gamma$ denote any of the groups $\{\GL_2,G\}$, then the motivic cohomology of $Y_\Gamma$ is given by
$$H_\mathrm{mot}^i(Y_\Gamma,\mathbf{Q}(j)):=\varinjlim_{\substack{U\ \text{open compact}}} H_{\mathrm{mot}}^i(Y_\Gamma(U),\mathbf{Q}(j))$$
where at each finite level $U$ the groups $H_{\mathrm{mot}}^i(Y_\Gamma(U),\mathbf{Q}(j))$ are defined in the sense of \cite{voevodsky2000triangulated} and \cite{friedlander2000bivariant}. The spaces $H_\mathrm{mot}^i(Y_\Gamma,\mathbf{Q}(j))$ are naturally smooth $\mathbf{Q}$-linear, $\Gamma(\mathbf{A}_f)$-representations.
\begin{rem}
    It is worth noting that everything in this section extends to the case of ``non-trivial coefficient sheaves''. See for example \cite{Loeffler_2021}, \cite{Loeffler_Skinner_Zerbes_2021} and \cite{grossi2020norm} for some of the formalisms required to extend the theory.
\end{rem}
\noindent By functoriality of motivic cohomology (with a shift in bi-degree) the map $\iota_{gU}$ induces a morphism
$$\iota_{gU,*}:H_\mathrm{mot}^1(Y_{\GL_2}(\GL_2(\mathbf{A}_f)\cap gUg^{-1}),\mathbf{Q}(1))\longrightarrow H_\mathrm{mot}^3(Y_G(U),\mathbf{Q}(2)).$$
In the same way as \cite[\S $9.1$]{Loeffler_2021} and \cite[\S $8.2$]{Loeffler_Skinner_Zerbes_2021},
we have:
\begin{prop}\label{prop mpa iota start}
    Let $\mathrm{vol}_{\GL_2(\mathbf{A}_f)}$ denote the normalized Haar measure on $\GL_2(\mathbf{A}_f)$ that gives $\GL_2(\hat{\mathbf{Z}})$ volume $1$. Then, there exists a morphism
    $$\iota_*:H^1_\mathrm{mot}(Y_{\GL_2},\mathbf{Q}(1))\otimes_\mathbf{Q} \mathcal{H}(G(\mathbf{A}_f),\mathbf{Q})\longrightarrow H_\mathrm{mot}^3(Y_G,\mathbf{Q}(2))$$
    characterized as follows: If $U\subseteq G(\mathbf{A}_f)$ is an open compact subgroup, $g\in G(\mathbf{A}_f)$ and $x$ is an element of $ H^1_\mathrm{mot}(Y_{\GL_2}(\GL_2(\mathbf{A}_f)\cap gUg^{-1}),\mathbf{Q}(1))$, then $$\iota_*(x\otimes \ch(gU))=\vol_{\GL_2(\mathbf{A}_f)}\left(\GL_2(\mathbf{A}_f)\cap gUg^{-1}\right)\cdot \iota_{gU,*}(x).$$
\end{prop}

\subsubsection{Eisenstein classes and the Rankin-Eisenstein map}
 
\begin{thm}\label{thm eisenstein map}
There exists a canonical $\GL_2(\mathbf{A}_f)$-equivariant map
$$S_0(\mathbf{A}_f^2,\mathbf{Q})\longrightarrow H^1_\mathrm{mot}(Y_{\GL_2},\mathbf{Q}(1))=\mathcal{O}(Y_{\GL_2})^\times\otimes \mathbf{Q},\ \phi\mapsto g_\phi$$
characterized by the following: If $\phi=\ch((\alpha,\beta)+N\hat{\mathbf{Z}}^2)$ for some $N\in\mathbf{Z}_{\geq 1}$ and $(\alpha,\beta)\in \mathbf{Q}^2-N\mathbf{Z}^2$, then $g_\phi=g_{\alpha/N,\beta/N}$ is the Siegel unit in the notation of \emph{\cite[\S $1.4$]{kato2004p}\emph}.
\begin{proof}
    This is \cite[Theorem $1.8$]{colmez2004conjecture}.
\end{proof}
\end{thm}
\begin{defn}\label{def rankin eisenstein}
    We define the Rankin-Eisenstein map
    $$\mathcal{RE}:\mathcal{S}_0(\mathbf{A}_f^2,\mathbf{Q})\otimes_\mathbf{Q} \mathcal{H}(G(\mathbf{A}_f),\mathbf{Q})\longrightarrow H_\mathrm{mot}^3(Y_G,\mathbf{Q}(2)),\ \phi\otimes \xi\mapsto \iota_*(g_\phi\otimes \xi)$$
    where $\iota_*$ is the map of \emph{\Cref{prop mpa iota start}}.
\end{defn}
\noindent It is not hard to see that the map $\mathcal{RE}$ is $\left(\GL_2(\mathbf{A}_f)\times G(\mathbf{A}_f)\right)$-equivariant in the sense of \Cref{def equiv maps}. The detailed argument in \cite[\S $6.1$]{grossi2020norm} works verbatim. 
\subsubsection{The classes and norm compatibility}\label{sec classes and norm relatiosn}

 Let $S$ be a fixed finite set of primes containing $2$ and write $\mathbf{Q}_S=\prod_{p\in S} \mathbf{Q}_p$. Let $\delta_S\in \mathcal{I}_0(G(\mathbf{Q}_S)/K_S,\mathbf{Z})$ where $K_S$ is some open compact subgroup. These will form the local data at the ``bad'' primes and will be fixed throughout. 
 \begin{defn}For the data at primes $p\notin S$, we take any family $\underline{\delta}=(\delta_p)_{p\notin S}$ contained in the product $ \prod_{p\notin S}\left(\mathcal{S}(\mathbf{Q}_p^2,\mathbf{Q})\otimes_\mathbf{Q} \mathcal{H}(G_p/G_p^\circ,\mathbf{Q})\right)$.
For a square free integer $n\in\mathbf{Z}_{\geq 1}$ coprime to $S$ and $p\notin S$, we define the ``$n$-truncated'' element:
$$\delta[n]_p=\begin{dcases}
    \ch(\mathbf{Z}_p^2)\otimes \ch(G_p^\circ),\ \text{if}\ p\nmid n\\
    \delta_p,\ \text{if}\ p|n.
\end{dcases}$$
and we set $\delta[n]:=\delta_S\bigotimes_{p\notin S}\delta[n]_p$.
\end{defn}
\noindent Finally, for such integer $n$, we write $K_S[n]$ for the open compact subgroup given by the following product: $K_S\times\{g=(g_1,g_2)\in G(\hat{\mathbf{Z}}^S)\ |\ \det(g_2)\equiv 1\mod n\}$. We are now ready to define a family of cohomology classes which will depend on $\underline{\delta}$, and implicitly $\delta_S$, however to ease notation, we omit the dependence on $\delta_S$ since we fix it throughout. We write $\mathscr{Z}_S$ for the set of square-free positive integers coprime to $S$.
\begin{defn}\label{def classes}
    Similarly to \emph{\cite{Loeffler_2021}}, we set $$\mathcal{Z}_{\mathrm{mot},n}(\underline{\delta}):=\mathcal{RE}(\delta[n])\in H^3_\mathrm{mot}(Y_G(K_S),\mathbf{Q}(2)).$$
\end{defn}
\noindent We note that for a fixed family $\underline{\delta}$, the class $\mathcal{Z}_{\mathrm{
mot},n}(\underline{\delta})$, for any $n\in\mathscr{Z}_S$, only depends on the elements $\delta_p$ for primes $p|n$. Additionally and most importantly, if the family of input data $\underline{\delta}$ has determinant level, i.e. it lies in $\prod_{p\notin S}\left(\mathcal{S}(\mathbf{Q}_p^2,\mathbf{Q})\otimes_\mathbf{Q} \mathcal{H}(G_p/G_p^\circ[p],\mathbf{Q})\right)$, then
$$\mathcal{Z}_{\mathrm{mot},n}(\underline{\delta})\in H^3_\mathrm{mot}(Y_G(K_S[n]),\mathbf{Q}(2)).$$
\begin{thm}\label{thm euler system norm relations}
     For positive integers $n|m$ let $\mathrm{pr}_n^m$ denote the natural map $Y_G(K_S[m])\rightarrow Y_G(K_S[n])$. 
    \begin{enumerate}
        \item If $\underline{\delta}\in\prod_{p\notin S}\mathcal{I}(G_p/G_p^\circ,\mathbf{Z}[1/p])$, then for any $n,m\in\mathscr{Z}_S$ with $\tfrac{m}{n}=p$ prime\emph{:}
        $$\mathcal{Z}_{\mathrm{mot},m}\left(\underline{\delta}\right)=\mathcal{P}_{\delta_p}\cdot \mathcal{Z}_{\mathrm{mot},n}\left(\underline{\delta}\right),\ \text{with}\ \mathcal{P}_{\delta_p}\in\mathcal{H}_{G_p}^\circ(\mathbf{Z}[1/p]).$$
        \item If $\underline{\delta}\in\prod_{p\notin S}\mathcal{I}_0(G_p/G_p^\circ[p],\mathbf{Z}[1/p])$, then for any $n,m\in\mathscr{Z}_S$ with $\tfrac{m}{n}=p$\emph{:}
        \begin{align}\label{eq: 25}\scalemath{1}{(\mathrm{pr}^{m}_n)_*\left(\mathcal{Z}_{\mathrm{mot},m}\left(\underline{\delta}\right)\right)=\mathcal{P}_{\mathrm{Tr}(\delta_p)}\cdot \mathcal{Z}_{\mathrm{mot},n}\left(\underline{\delta}\right),\ \text{with}\ \mathcal{P}_{\mathrm{Tr}(\delta_p)}\in\left(p-1,\mathcal{P}_p^{'}(1)\right)\subseteq\mathcal{H}_{G_p}^\circ(\mathbf{Z}[1/p]).}\end{align}
        \item 
        For the specific choices of $\underline{\delta}$ below, we have for all $n,m\in\mathscr{Z}_S$ with $\tfrac{m}{n}=p$ prime 
\begin{align*}
(\mathrm{pr}^{m}_n)_*\left(\mathcal{Z}_{\mathrm{mot},m}\left(\underline{\delta}\right)\right)=\begin{dcases}
        (p-1)\cdot \mathcal{Z}_{\mathrm{mot},n}\left(\underline{\delta}\right),\ &\text{if}\ \underline{\delta}=\underline{\delta}_0\\
        \mathcal{P}_p^{'}(1)\cdot \mathcal{Z}_{\mathrm{mot},n}\left(\underline{\delta}\right),\ &\text{if}\ \underline{\delta}=\underline{\delta}_1
    \end{dcases}
\end{align*}
where $\scalemath{0.9}{\underline{\delta}_0=\left((p-1)\cdot \ch(\mathbf{Z}_p^2)\otimes \ch(G_p^\circ[p])\right)_{p\notin S}}$, and
$\scalemath{0.9}{\underline{\delta}_1=\left(n_p\phi_{p,2}\otimes\left(\ch(G_p^\circ[p])- \ch\left(\left(1,\left[\begin{smallmatrix}
            1 & 1/p \\
            & 1
        \end{smallmatrix}\right]\right) G_p^\circ[p]\right)\right)\right)_{p\notin S}}$,\\ $n_p=p(p-1)^2(p+1)$ and $\phi_{p,2}=\ch(p^2\mathbf{Z}_p\times (1+p^2\mathbf{Z}_p))$
    \end{enumerate}
\end{thm}

\begin{proof}
    Firstly note that the data $\delta[np]$ and $\delta[n]$ only differ at the prime $p$ and coincide everywhere else. Write $V$ for $H^3_\mathrm{mot}(Y_G,\mathbf{Q}(2))$ regarded as a smooth $G_p$-representation. It is also clear that we can check the equalities after tensoring with $\mathbf{C}$ and thus we will omit it from our notation. By the global equivariance of the Rankin-Eisenstein map, we can observe that after fixing the input data away from $p$, the composite map at $p$
    $$\mathcal{RE}_{\delta[m]^{(p)}}:\mathcal{S}(\mathbf{Q}_p^2)\otimes \mathcal{H}(G_p)\longrightarrow\mathcal{S}_0(\mathbf{A}_f^2)\otimes_\mathbf{Q}\mathcal{H}(G(\mathbf{A}_f))\overset{\mathcal{RE}}{\longrightarrow} V$$
    is $(G_p\times H_p)$-equivariant, where the first map is given by $x_p\mapsto \delta[m]^{(p)}\otimes x_p$. Thus if $\underline{\delta}$ is as in part one,
    $$\mathcal{Z}_{\mathrm{mot},m}\left(\underline{\delta}\right)=\mathcal{RE}_{\delta[m]^{(p)}}(\delta_p)=\mathcal{P}_{\delta_p}\cdot \mathcal{RE}_{\delta[m]^{(p)}}(\ch(\mathbf{Z}_p^2)\otimes \ch(G_p^\circ)),\ \mathcal{P}_{\delta_p}\in\mathcal{H}_{G_p}^\circ(\mathbf{Z}[1/p]).$$ This follows from the first part of \Cref{thm local norm relations}. But since $(p,nS)=1$, $\mathcal{RE}_{\delta[m]^{(p)}}$ evaluated at the unramified vector $\ch(\mathbf{Z}_p^2)\otimes \ch(G_p^\circ)$ is nothing more than $\mathcal{Z}_{\mathrm{mot},n}(\underline{\delta})$ and the first part of the theorem follows. For the second part, we have a commutative diagram
    \[\begin{tikzcd}
	{\mathcal{S}_0(\mathbf{A}_f^2,\mathbf{Q})\otimes_\mathbf{Q}\mathcal{H}\left(G(\mathbf{A}_f)/K_S[m],\mathbf{Q}\right)} & {H_\mathrm{mot}^3(Y_G(K_S[m]),\mathbf{Q}(2))} & {H_\mathrm{mot}^3(Y_G(K_S[n]),\mathbf{Q}(2))} \\
	{\mathcal{S}_0(\mathbf{Q}_p^2,\mathbf{Q})\otimes_{\mathbf{Q}} \mathcal{H}(G_p/G_p^\circ[p],\mathbf{Q})} & V & V
	\arrow["{\mathcal{RE}[m]}", from=1-1, to=1-2]
	\arrow["{\mathcal{RE}}", from=1-1, to=2-2]
	\arrow["{(\mathrm{pr}^{m}_n)_*}", from=1-2, to=1-3]
	\arrow["{\tau_{m}}", from=1-2, to=2-2]
	\arrow["{\tau_n}", from=1-3, to=2-3]
	\arrow[from=2-1, to=1-1]
	\arrow["{\mathcal{RE}_{\delta[m]^{(p)}}}"', from=2-1, to=2-2]
	\arrow["{\mathrm{norm}_{G_p^\circ}^{G_p^\circ[p]}}", from=2-2, to=2-3]
\end{tikzcd}\]
where $\tau_{n},\tau_{m}$ are the natural maps, and the leftmost vertical map is again given by $x_p\mapsto \delta[m]^{(p)}\otimes x_p$ which is well-defined since we are now assuming that $\underline{\delta}$ is as in the second part of the theorem. The commutativity of the right-most square follows from the fact that pushforward acts on cohomology via coset representatives. Thus, by the $(G_p\times H_p)$-equivariance property of $\mathcal{RE}_{\delta[m]^{(p)}}$ and the second part of \Cref{thm local norm relations}, we have for $\underline{\delta}$ as in the second statement of the theorem,
$$\mathrm{norm}_{G_p^\circ}^{G_p^\circ[p]}\left(\mathcal{RE}_{\delta[m]^{(p)}}(\delta_p)\right)=\mathcal{P}_{\mathrm{Tr}(\delta_p)}\cdot \mathcal{RE}_{\delta[m]^{(p)}}(\ch(\mathbf{Z}_p^2)\otimes \ch(G_p^\circ)),\ \mathcal{P}_{\mathrm{Tr}(\delta_p)}\in\left(p-1,\mathcal{P}_p^{'}(1)\right)\subseteq \mathcal{H}_{G_p}^\circ(\mathbf{Z}[1/p]).$$
The proof of the second statement in the theorem is now complete by following the commutative diagram above and unraveling the definitions.

\noindent We now deal with the third and final part. The $\underline{\delta}_0$ case is trivial. For the $\underline{\delta}_1$ case, we've just seen that it suffices to show that $\mathcal{P}_{\mathrm{Tr}(\delta_{{1,p}})}=\mathcal{P}_p^{'}(1)$ where 
$$\delta_{1,p}:=n_p\phi_{p,2}\otimes  \left(\ch(G_p^\circ[p])- \ch\left(\left(1,\left[\begin{smallmatrix}
            1 & 1/p \\
            & 1
        \end{smallmatrix}\right]\right) G_p^\circ[p]\right)\right).$$
        These local elements are indeed integral in the sense of \Cref{def integral lattice}. Even though the local Rankin-Selberg zeta integral of \cite{jacquet1983rankin} does not allow us to attack the generality of parts $(1)$ and $(2)$, it actually works well for this nice integral input data $\delta_{1,p}$. Thus, one can use it in order to realize $\mathcal{P}_{\mathrm{Tr}(\delta_{1,p})}$ as $\mathcal{P}_p^{'}(1)$ (one could also specialize the more general integral machinery of \Cref{sec from and back} and \Cref{sec Hecke forms}, but for this simple case of explicit input data, that is not really necessary). By the argument of \cite[Proposition $4.4.1$]{Loeffler_2021}, it suffices to show that the local Rankin-Selberg zeta integral of \cite{jacquet1983rankin} for this choice of integral data:
        \begin{align*}
            Z(n_p\phi_{p,2}, W_{\pi_{p,1}}^\mathrm{sph},W_{\pi_{p,2}}^\mathrm{sph}&- \left[\begin{smallmatrix}
                1 & 1/p\\
                & 1
            \end{smallmatrix}\right]W_{\pi_{p,2}}^\mathrm{sph},s)\\
            &:=n_p\int_{N_p\backslash H_p}W_{\pi_{p,1}}^\mathrm{sph}(h) \left(W_{\pi_{p,2}}^\mathrm{sph}- \left[\begin{smallmatrix}
                1 & 1/p\\
                & 1
            \end{smallmatrix}\right]W_{\pi_{p,2}}^\mathrm{sph}\right)(h)\ \phi_{p,2}((0,1)h)\ |\det(h)|_p^s\ dh
        \end{align*}
        is equal to the constant function $1$, for every unramified principal-series $\pi_{p,1}\boxtimes \pi_{p,2}=\mathcal{W}(\pi_{p,1},\psi)\boxtimes \mathcal{W}(\pi_{p,2},\overline{\psi})$ of $G_p$. This computation is a bit lengthy, but can be carried out directly by unfolding. For more details on this zeta integral and its properties, we refer the reader to \Cref{sec local integrals} and \cite{jacquet1983rankin}. 
\end{proof}
\begin{rem}
The strength of this result, and the payoff of the previous representation theoretic sections is having integrality statements \Cref{thm euler system norm relations}$(1)$,$(2)$ for \textit{generic} $\underline{\delta}$. This cannot be addressed with explicit zeta-integral computations. However, adding them in when choosing as input more ``classical'' data, allows us to get \Cref{thm euler system norm relations}$(3)$ which strengthens known results on norm-relations as well.
\end{rem}
\subsubsection{Base change to cyclotomic fields}\label{sec base change}
Following \cite[\S $8.2$]{Loeffler_2021}, we can give an alternative description of the classes $\mathcal{Z}_{\mathrm{mot},n}(\underline{\delta})$. We can regard $G$ as a subgroup of $G\times \mathbf{G}_m$ via the embedding $i:g=(g_1,g_2)\mapsto(g,\mathrm{det}(g_2))$. For every $n\in\mathscr{Z}_S$, this gives an open and closed embedding (which we still denote by $i$)
$$i:Y_G(K_S[n])\hookrightarrow Y_G(K_S)\times_{\mathrm{Spec}(\mathbf{Q})}\mathrm{Spec}(\mathbf{Q}(\mu_n)).$$
The pushforward along this map gives
$$i_*:H_{\mathrm{mot}}^3\left(Y_G(K_S[n]),\mathbf{Q}(2)\right)\longrightarrow H_\mathrm{mot}^3\left(Y_G(K_S)\times_{\mathrm{Spec}(\mathbf{Q})}\mathrm{Spec}(\mathbf{Q}(\mu_n)),\mathbf{Q}(2)\right).$$
which satisfies certain intertwining properties. In particular, for every prime $p\nmid n$, the map $i_*$ satisfies the following intertwining properties with respect to the action of $\mathcal{H}_{G_p}^\circ(\mathbf{Z}[1/p])$ on the left and the action of $\mathcal{H}_{G_p}^\circ(\mathbf{Z}[1/p])[\mathrm{Gal}(\mathbf{Q}(\mu_n)/\mathbf{Q})]$ on the right: For $x\in H_{\mathrm{mot}}^3\left(Y_G(K_S[n]),\mathbf{Q}(2)\right)$, we have
\begin{align}\label{eq: intert. props.}i_*(\theta\cdot x)=\begin{dcases}
    \theta\cdot i_*(x)&,\ \text{if} \ \theta=\mathcal{S}_{p,1}\\
    \theta\cdot i_*(x)&,\ \text{if} \ \theta=\mathcal{T}_{p,1}\\
    (\theta \mathrm{Frob}_p^{2})\cdot i_*(x)&,\ \text{if} \ \theta=\mathcal{S}_{p,2}\\
    (\theta\mathrm{Frob}_p)\cdot i_*(x)&,\ \text{if} \ \theta=\mathcal{T}_{p,2}
\end{dcases}\end{align}
where $\mathrm{Frob}_p$ denotes arithmetic Frobenius at $p$, regarded as an element of $\mathrm{Gal}(\mathbf{Q}(\mu_n)/\mathbf{Q})$. This follows from \cite[\S $8.2$]{Loeffler_2021} using the same normalization for the global Artin reciprocity map, and our definition of $i$.
\begin{rem}
    As remarked earlier, at first glance it might appear that there's a certain asymmetry lingering in the background since we've made a choice in which slot to apply the determinant map to $\mathbf{G}_m$. However the results are independent of this choice. Additionally, for Euler system applications, one can simply work with the group $\GL_2\times_{\GL_1} \GL_2$ which is essentially what happens in the Asai case at split primes \cite[\S $5$]{Groutides+2026+103+124}. Here the fiber product over $\GL_1$ is taken with respect to the determinant map.
\end{rem}
\noindent The polynomial $\mathcal{P}_p(x)$ of \Cref{def Euler polynomial factor} can be computed explicitly. It is given by
$$\mathcal{P}_p(x)=1-\tfrac{1}{p}\mathcal{T}_{p,1}\mathcal{T}_{p,2}x+\left(\tfrac{1}{p}\mathcal{S}_{p,1}\mathcal{T}_{p,2}^2+\tfrac{1}{p}\mathcal{S}_{p,2}\mathcal{T}_{p,1}^2-2\mathcal{S}_{p,1}\mathcal{S}_{p,2}\right)x^2-\tfrac{1}{p}\mathcal{S}_{p,1}\mathcal{S}_{p,2}\mathcal{T}_{p,1}\mathcal{T}_{p,2}x^3+\mathcal{S}_{p,1}^2\mathcal{S}_{p,2}^2x^4.$$
Combining this with \eqref{eq: intert. props.} it follows that $i_*$ intertwines the action of $\mathcal{P}_p^{'}(1)$ on $H_{\mathrm{mot}}^3\left(Y_G(K_S[n]),\mathbf{Q}(2)\right)$ with that of $\mathcal{P}_p^{'}\left(\mathrm{Frob}_p^{-1}\right)$ on $H_\mathrm{mot}^3\left(Y_G(K_S)\times_{\mathrm{Spec}(\mathbf{Q})}\mathrm{Spec}(\mathbf{Q}(\mu_n)),\mathbf{Q}(2)\right)$. \\
\\
For data as in \Cref{def classes}, we write $\Xi_{\mathrm{mot},n}(\underline{\delta})$ for the image of $\mathcal{Z}_{\mathrm{mot},n}(\underline{\delta})$ under $i_*$. Then we can re-state \Cref{thm euler system norm relations} for the classes $\Xi_{\mathrm{mot},n}(\underline{\delta})$:
\begin{cor}\label{cor base change}
\begin{enumerate}
\item If $\underline{\delta}\in\prod_{p\notin S}\mathcal{I}_0(G_p/G_p^\circ[p],\mathbf{Z}[1/p])$, then for any $n,m\in\mathscr{Z}_S$ with $\tfrac{m}{n}=p$ we have
$$
    \mathrm{norm}^{\mathbf{Q}(\mu_m)}_{\mathbf{Q}(\mu_n)}\left(\Xi_{\mathrm{mot},m}(\underline{\delta})\right)=\mathcal{P}_{\mathrm{Tr}(\delta_p)}^\mathrm{cycl}\cdot \Xi_{\mathrm{mot},n}(\underline{\delta})$$
     $$\mathcal{P}_{\mathrm{Tr}(\delta_p)}^\mathrm{cycl}\in \left(p-1,\mathcal{P}_p^{'}\left(\mathrm{Frob}_p^{-1}\right)\right)\subseteq \mathcal{H}_{G_p}^\circ(\mathbf{Z}[1/p])[\mathrm{Gal}(\mathbf{Q}(\mu_n)/\mathbf{Q})].$$
\item In the special cases $\underline{\delta}=\underline{\delta}_0$ and $\underline{\delta}=\underline{\delta}_1$ we have
\begin{align*}
    \mathrm{norm}^{\mathbf{Q}(\mu_m)}_{\mathbf{Q}(\mu_n)}(\Xi_{\mathrm{mot},m}(\underline{\delta}))&=\begin{dcases}
        (p-1)\cdot \Xi_{\mathrm{mot},n}\left(\underline{\delta}\right),\ &\text{if}\ \underline{\delta}=\underline{\delta}_0\\
        \mathcal{P}_p^{'}(\mathrm{Frob}_p^{-1})\cdot \Xi_{\mathrm{mot},n}\left(\underline{\delta}\right),\ &\text{if}\ \underline{\delta}=\underline{\delta}_1
    \end{dcases}
\end{align*} 
\end{enumerate}
\end{cor}
\begin{proof}
    \Cref{thm euler system norm relations} and the discussion above
\end{proof}

\subsubsection{Integral \'etale realisation}\label{sec etale realisation}
In this section we fix an prime $\ell\notin S$. There is an \'etale regulator map
$$r_{\acute{e}\text{t}}: H_\mathrm{mot}^3(Y_G,\mathbf{Q}(2))\longrightarrow H_{\acute{e}\text{t}}^3(Y_G,\mathbf{Q}_\ell(2))$$
obtained by taking the limit of the regulator maps defined in \cite{huber2000realization} at each finite level. Let $n\in\mathbf{Z}_{\geq 1}$ be a square-free integer coprime to $\ell S$. As in \cite{Loeffler_2021}, but with different notation, for each integer $c>1$ and coprime to $6nS$, we define classes
\begin{align*}_c\mathcal{Z}_{\acute{e}\text{t},n}\left(\underline{\delta}\right)&:=\left(c^2-\prod_{p| c}\mathcal{S}_p^{v_p(c)}\right)\cdot r_{\acute{e}\text{t}}\left(\mathcal{Z}_{\mathrm{mot},n}\left(\underline{\delta}\right)\right)\\
_c\Xi_{\acute{e}\text{t},n}(\underline{\delta})&:= \left(c^2-\prod_{p| c}\mathcal{S}_p^{v_p(c)}\mathrm{Frob}_p^{2v_p(c)}\right)\cdot r_{\acute{e}\text{t}}\left(\Xi_{\mathrm{mot},n}\left(\underline{\delta}\right)\right)\end{align*}
Of course by \eqref{eq: intert. props.} these definitions are compatible with one another. 
These normalisation factors arise from the integral variant of the map of \Cref{thm eisenstein map}, a description of which can be found in \cite[\S $7$]{Loeffler_Skinner_Zerbes_2021}. If $\underline{\delta}$ is as in \Cref{thm euler system norm relations}, then by the argument of \cite[Proposition $9.5.2$]{Loeffler_2021} but simpler since we are dealing with trivial coefficients, the classes $_c\mathcal{Z}_{\acute{e}\text{t},n}\left(\underline{\delta}\right)$ actually lie in $H_{\acute{e}\text{t}}^3(Y_G,\mathbf{Z}_\ell(2))$ (and similarly for the $_c\Xi_{\acute{e}\text{t}}$-classes). Then, \Cref{cor base change} also holds for these integral classes.
   \section{Integrality for Rankin-Selberg type periods}\label{sec int periods}
\subsection{The space $\Hom_{P_p}(\pi_{p,1}\otimes\pi_{p,2},\delta_P)$ and an integral lattice in $\pi_{p,1}\otimes\pi_{p,2}$}\label{sec relation to l-modular}
The space $\Hom_{P_p}(\pi_{p,1}\otimes \pi_{p,2},\delta_P)$, up to twist, has been studied in \cite{jacquet1983rankin} in the more general setting of $\GL(n)$ and for (possibly ramified) Whittaker type representations, where a ``generic'' bound ($\leq 1$) on the dimension of this space is established, using vastly different methods. This step is crucial for the proof of the functional equation for the local zeta integrals of \textit{op.cit}. To link this to our work we will relate the module $C_c^\infty(P_p\backslash G_p/G_p^\circ)$ to a certain module of $P_p$-coinvariants. The natural object to consider here is the module of $\delta_P$-twisted $P_p$-coinvariants of $C_c^\infty(G_p/G_p^\circ)$ given by 
$$C_c^\infty(G_p/G_p^\circ)_{P_p,\delta_P}:=C_c^\infty(G_p/G_p^\circ)/\langle x\cdot\xi-\delta_P(x)\xi\ |\ x\in P_p,\xi\in C_c^\infty(G_p/G_p^\circ)\rangle.$$
This is again a module over $\mathcal{H}_{G_p}^\circ$ under the usual Hecke action induced by right translation.
\begin{prop}\label{prop from inv to coinv}
    There is a Hecke equivariant isomorphism
    $$\Phi: C_c^\infty(P_p\backslash G_p/G_p^\circ)\overset{\simeq}{\longrightarrow} C_c^\infty(G_p/G_p^\circ)_{P_p,\delta_P},\ \ch(P_pgG_p^\circ)\mapsto\frac{1}{\vol_{P_p}\left(P_p\cap g G_p^\circ g^{-1}\right)}\ch( gG_p^\circ)$$
    where $\vol_{P_p}$ denotes the volume with respect to right or left normalised Haar measure and is independent of such choice.
    \begin{proof}
    Firstly, we show that the volume factor in question is independent of the choice of left or right normalised Haar measure. Indeed, 
    \begin{align*}
        \vol_{P_p}(P_p\cap g G_p^\circ g^{-1},d^Rx)&=\int_{P_p\cap gG_p^\circ g^{-1}} d^Rx\\
        &=\int_{P_p\cap g G_p^\circ g^{-1}} \delta_P(x)\ d^Lx=\vol_{P_p}(P_p\cap g G_p^\circ g^{-1},d^Lx)
    \end{align*}
    where the last equality follows from the fact that $\delta_P$ is trivial on $P_p\cap gG_p^\circ g^{-1}$.
        We now check that the map $\Phi$ is well-defined. To do this, it clearly suffices to show that $\Phi(\ch(P_p g G_p^\circ))$ agrees with $\Phi(\ch(P_p yg G_p^\circ))$ for $y\in P_p$. Indeed, $\ch(P_p yg G_p^\circ)$ gets mapped to 
        \begin{align}\label{eq:26}
            \frac{1}{\vol_{P_p}(P_p\cap yg G_p^\circ g^{-1}y^{-1},d^Rx)}\ch(yg G_p^\circ)&=\frac{1}{\vol_{P_p}(P_p\cap yg G_p^\circ g^{-1},d^Rx)}\ch(yg G_p^\circ)\\
           \nonumber &=\frac{\delta_P(y)}{\vol_{P_p}(P_p\cap yg G_p^\circ g^{-1},d^Rx)}\ch(gG_p^\circ).
        \end{align}
        The first equality follows from the fact that $d^Rx$ is right invariant, and the second equality follows from the coinvariant action.
        Now, $\vol_{P_p}(P_p\cap yg G_p^\circ g^{-1},d^Rx)$ is given by $\int_{P_p\cap yg G_p^\circ g^{-1}} d^Rx= \int_{P_p\cap yg G_p^\circ g^{-1}} \delta_P(x) d^Lx$ which is in turn equal to $\delta_P(y) \vol_{P_p}\left(P_p\cap g G_p^\circ g^{-1}\right)$. From \eqref{eq:26} we now see that $\Phi$ is well-defined. The map $\Phi$ is a priori a morphism of $\mathbf{C}$-vector spaces.  Showing directly that $\Phi$ is Hecke equivariant using its current definition appears to be very hard. We will instead construct an inverse morphism to $\Phi$ which can be readily seen to be Hecke equivariant. The result will then follow at once. There is a natural morphism 
        $$\Psi: C_c^\infty(G_p/G_p^\circ)\longrightarrow C_c^\infty(P_p\backslash G_p/G_p^\circ),\ \xi\mapsto \left(f_\xi:g\mapsto \int_{P_p}\xi(xg)\ d^Rx\right)$$
        which is Hecke equivariant by construction. Now, for $y\in P_p$ $$f_{y\cdot\xi}(g)=\int_{P_p}\xi(y^{-1}xg)\ d^Rx=\int_{P_p}\xi(y^{-1}xg)\delta_P(x)\ d^Lx=\delta_P(y)\int_{P_p}\xi(xg)\delta_P(x)\ d^Lx=\delta_P(y)f_\xi(g).$$
        Thus, the map $\Psi$ factors through $C_c^\infty(G_p/G_p^\circ)_{P_p,\delta_P}$. Finally, for $g\in G_p$, the function $f_{\ch(gG_p^\circ)}$ is supported on $P_pg G_p^\circ$ and 
        $f_{\ch(gG_p^\circ)}(g)=\vol_{P_p}\left(P_p\cap g G_p^\circ g^{-1}\right).$ Hence $\Psi(\ch(gG_p^\circ))=\vol_{P_p}\left(P_p\cap g G_p^\circ g^{-1}\right)\ch(P_pgG_p^\circ)$ which shows that $\Phi$ and $\Psi$ are mutually inverse as required.
    \end{proof}
\end{prop}

\begin{defn}
    \begin{enumerate}    
\item The lattice $\mathcal{J}(G_p/G_p^\circ,\mathbf{Z}[1/p])_{P_p,\delta_P}$ inside $C_c^\infty(G_p/G_p^\circ)_{P_p,\delta_P}$ consists of $\mathbf{Z}[1/p]$-linear combinations of elements of the form
    $$\frac{1}{\vol_{P_p}\left(P_p\cap g G_p^\circ g^{-1}\right)}\ch(g G_p^\circ),\ g\in G_p.$$
    \item The lattice $\mathcal{J}_1(G_p/G_p^\circ,\mathbf{Z}[1/p])_{P_p,\delta_P}$ inside $C_c^\infty(G_p/G_p^\circ)_{P_p,\delta_P}$ consists of $\mathbf{Z}[1/p]$-linear combinations of elements of the form 
    $$\frac{1}{\vol_{P_p}\left(P_p\cap gG_p^\circ[p] g^{-1}\right)}\ch(gG_p^\circ),\ g\in G_p.$$
\end{enumerate}
\end{defn}

\begin{lem}\label{lem identification of lattices}
    The map $\Phi$ identifies the lattices 
    \begin{align*}C_c^\infty(P_p\backslash G_p/G_p^\circ,\mathbf{Z}[1/p])&\simeq_{\Phi} \mathcal{J}(G_p/G_p^\circ,\mathbf{Z}[1/p])_{P_p,\delta_P}\\
    C_c^{\infty,1}(P_p\backslash G_p/G_p^\circ,\mathbf{Z}[1/p])&\simeq_{\Phi} \mathcal{J}_1(G_p/G_p^\circ,\mathbf{Z}[1/p])_{P_p,\delta_P}
    \end{align*}
    where $C_c^{\infty,1}$ is defined in \eqref{eq:111}.
\end{lem}
\begin{proof}
    The first identification follows at once from the definitions and \Cref{prop from inv to coinv}. Thus we deal with the second one. It is straightforward to show that every double coset $P_pg G_p^\circ$ is equal to $P_p(z_1\left[\begin{smallmatrix}
        p^m & \\
        & 1
    \end{smallmatrix}\right],z_2)G_p^\circ$ with $m\in \mathbf{Z}$ and $ z_i\in Z_p$, or $P_p(z_1\left[\begin{smallmatrix}
        p^m & p^{-n}\\
        & 1
    \end{smallmatrix}\right],z_2)G_p^\circ$ with $n\in\mathbf{Z}_{\geq 1}, m\in\mathbf{Z}, m>-n $ and $z_i\in Z_p$. By definition of the two lattices and the map $\Phi$,  it suffices to show that 
    \begin{align*}
        \tfrac{1}{p-1}\cdot \vol_{P_p}\left(P_p\cap g G_p^\circ g^{-1}\right)&= \vol_{P_p}\left(P_p\cap g G_p^\circ[p] g^{-1}\right),\ g= \left(z_1\left[\begin{smallmatrix}
        p^m & \\
        & 1
    \end{smallmatrix}\right],z_2\right)\\
     \vol_{P_p}\left(P_p\cap g G_p^\circ g^{-1}\right)&=\vol_{P_p}\left(P_p\cap g G_p^\circ[p] g^{-1}\right),\ g=\left(z_1\left[\begin{smallmatrix}
        p^m & p^{-n}\\
        & 1
    \end{smallmatrix}\right],z_2\right).
    \end{align*}
    We firstly treat the first case, in which we have $P_p\cap gG_p^\circ g^{-1}=P_p^\circ\cap \left[\begin{smallmatrix}
        p^m & \\
        &1
    \end{smallmatrix}\right] H_p^\circ \left[\begin{smallmatrix}
        p^m & \\
        &1
    \end{smallmatrix}\right]^{-1}.$ A quick matrix calculation shows that this is nothing more than $P_p^\circ$ if $m\leq 0$ and $\left[\begin{smallmatrix}
        \mathbf{Z}_p^\times & p^m\mathbf{Z}_p\\
        & 1
    \end{smallmatrix}\right]$ if $m>0$. The former has volume $1$ and the latter has volume $\frac{1}{p^m}$. On the other hand, if $m\leq 0$ then $P_p\cap g G_p^\circ[p] g^{-1}=\left[\begin{smallmatrix}
        1+p\mathbf{Z}_p & \mathbf{Z}_p\\
        & 1
    \end{smallmatrix}\right]$ which clearly has volume $\frac{1}{p-1}$. If $m>0$ then $P_p\cap g G_p^\circ[p] g^{-1}=\left[\begin{smallmatrix}
        1+p\mathbf{Z}_p & p^m\mathbf{Z}_p \\
        &1
    \end{smallmatrix}\right]$ which has volume $\frac{1}{(p-1)p^m}$ as required. For the second case, note that $m+n$ is strictly greater than zero. Another quick matrix computation shows that in this case the two subgroups have identical volume which is given (up to a power of $p$) by the volume of $\left[\begin{smallmatrix}
        1+p^{n+m}\mathbf{Z}_p & \mathbf{Z}_p\\
        & 1
    \end{smallmatrix}\right]$.
\end{proof}
 We now let $\pi_{p,1},\pi_{p,2}$ be unramified Whittaker type $H_p$-representations with $\omega_{\pi_{p,1}}\neq \omega_{\pi_{p,2}}^{-1}$. One then has a surjective linear map by acting on the spherical vector 
$$C_c^\infty(G_p/G_p^\circ)_{P_p,\delta_P}[\tfrac{1}{1-\mathcal{S}_p}]\longrightarrow (\pi_{p,1}\otimes\pi_{p,2})_{P_p,\delta_P}$$
where the object on the right denotes the $\delta_P$-twisted $P_p$-coinvariants of the representation $\pi_{p,1}\boxtimes \pi_{p,2}$ defined in a similar fashion as $C_c^\infty(G_p/G_p^\circ)_{P_p,\delta_P}$.
\Cref{prop from inv to coinv} and \Cref{thm freeness of P-invariant vectors} then imply that 
$$\dim\ \mathrm{Hom}_{P_p}(\pi_{p,1}\otimes \pi_{p,2},\delta_P)\leq 1$$ and a non-zero linear form in this space is non-vanishing on the spherical vector $W_{\pi_{p,1}}^\mathrm{sph}\otimes W_{\pi_{p,2}}^\mathrm{sph}.$ In our setup we can actually explicitly write down a non-zero element of this space using zeta integrals. The construction is similar to that of \Cref{sec the linear form}. Indeed, the assignment
\begin{align}\label{eq: explicit P-invariant period}\mathscr{D}_{\Pi_p}:\Pi_p:=\pi_{p,1}\otimes \pi_{p,2}\longrightarrow\mathbf{C},\ W_1\otimes W_2\mapsto \lim_{s\rightarrow 0}\frac{\int_{\mathbf{Q}_p^\times}\left(W_1\boxtimes W_2\right)(\left[\begin{smallmatrix}
            a & \\
            & 1
        \end{smallmatrix}\right])\ |a|_p^{s-1}\ d^\times a}{L(\pi_{p,1}\boxtimes\pi_{p,2},s)}\end{align}
is well-defined and gives an element of $\Hom_{P_p}(\pi_{p,1}\otimes \pi_{p,2},\delta_P)$, which takes the value $L(\omega_{\pi_{p,1}\boxtimes \pi_{p,2}},0)^{-1}$ on the normalized spherical vector. This is clearly non-zero since $\omega_{\pi_{p,1}}\neq \omega_{\pi_{p,2}}^{-1}$.

It is also worth noting that $\Lambda_{\Pi_p}$ and $\mathscr{D}_{\Pi_p}$ are related through the following commutative diagram 
\[\begin{tikzcd}[ampersand replacement=\&,sep=scriptsize]
	{C_c^\infty(G_p/G_p^\circ)_{P_p,\delta_P}[\tfrac{1}{1-\mathcal{S}_p}]} \& {(\Pi_p)_{P_p,\delta_P}} \& {\mathbf{C}\simeq(\Pi_p^\vee)^{G_p^\circ}} \\
	{C_c^\infty(P_p\backslash G_p/G_p^\circ)[\tfrac{1}{1-\mathcal{S}_p}]}
	\arrow[two heads, from=1-1, to=1-2]
	\arrow["{\mathscr{D}_{\Pi_p}}", from=1-2, to=1-3]
	\arrow["{\simeq_\Phi}", from=2-1, to=1-1]
	\arrow[bend right =10,"{\Lambda_{\Pi_p}}"', from=2-1, to=1-3]
\end{tikzcd}\]
Passing to an integral level, we have the following result.
\begin{prop}\label{thm integral P-lattice for pi1 x pi2 }
    Let $\Pi_p:=\pi_{p,1}\otimes \pi_{p,2}$ be as above, and assume that the spherical Hecke eigensystem $\Theta_{\Pi_{p}}$ takes values in $\mathcal{R}$ when restricted to $\mathcal{H}_{G_p}^\circ(\mathcal{R})$, for some $\mathbf{Z}[1/p]$-algebra $\mathcal{R}$. Let $\mathcal{Z}$ be a non-zero linear form in $\mathrm{Hom}_{P_p}(\Pi_p,\delta_P)$, normalized to take the value $L(\omega_{\Pi_p},0)^{-1}$ on $W_{\pi_{p,1}}^\mathrm{sph}\otimes W_{\pi_{p,2}}^\mathrm{sph}$. Then:
    \begin{enumerate}
    
    \item The linear form $\mathcal{Z}$ is valued in $\mathcal{R}$ on the lattice 
    $$\mathrm{span}_{\mathcal{R}}\left\{\frac{1}{\vol_{P_p}(P_p\cap gG_p^\circ g^{-1})}\ g\cdot (W_{\pi_{p,1}}^\mathrm{sph}\otimes W_{\pi_{p,2}}^\mathrm{sph})\ |\ g\in G_p\right\}\subseteq\pi_{p,1}\otimes\pi_{p,2}.$$
    \item The linear form $\mathcal{Z}$ is valued in the ideal $\left(p-1,L(\Pi_p,0)^{-1}\right)\subseteq\mathcal{R}$ on the lattice 
    $$\mathrm{span}_\mathcal{R}\left\{\frac{1}{\vol_{P_p}(P_p\cap gG_p^\circ[p] g^{-1})}\ g\cdot (W_{\pi_{p,1}}^\mathrm{sph}\otimes W_{\pi_{p,2}}^\mathrm{sph})\ |\ g\in G_p\right\}\subseteq\pi_{p,1}\otimes\pi_{p,2}.$$
    \end{enumerate}
    \begin{proof}
    By multiplicity one described above, $\mathcal{Z}$ is already completely determined and coincides with $\mathscr{D}_{\Pi_p}$. The value of $\mathscr{D}_{\Pi_p}$ on an element of the lattice in question can be expressed as an 
$\mathcal{R}$-linear combination of elements of the form $\mathscr{D}_{\Pi_p}\left(\xi\cdot (W_{\pi_{p,1}}^\mathrm{sph}\otimes W_{\pi_{p,2}}^\mathrm{sph})\right)$ for some $\xi\in\mathcal{J}(G_p/G_p^\circ,\mathbf{Z}[1/p])_{P_p,\delta_P}$. Then $\xi$ corresponds to some $f$ in $C_c^\infty(P_p\backslash G_p/G_p^\circ,\mathbf{Z}[1/p])$ via $\Phi$. The first part then follows from \Cref{cor integrality for lambda} and the commutative diagram above. For the second part, using the same argument as for the first part, the value of $\mathscr{D}_{\Pi_p}$ evaluated on an element of the lattice in question is an $\mathcal{R}$-linear combination of elements of the form $\Lambda_{\Pi_p}(f)$ where now $f\in C_c^{\infty,1}(P_p\backslash G_p/G_p^\circ,\mathbf{Z}[1/p])$. But by \Cref{thm freeness of P-invariant vectors}, and \Cref{thm conj 1 for P-invariant}, the corresponding Hecke operator $\mathcal{P}_{f,\mathrm{loc}}$, is given by $\tfrac{-\mathcal{S}_p\mathcal{Q}_f^{'}}{1-\mathcal{S}_p}$ where $\mathcal{Q}_f\in(p-1,\mathcal{P}_p(1))\subseteq\mathcal{H}_{G_p}^\circ(\mathbf{Z}[1/p])$ . Hence by \Cref{prop equivariance of linear form}, $\Lambda_{\Pi_p}(f)=\Lambda_{\Pi_p}(\mathcal{P}_{f,\mathrm{loc}}\cdot f_0)=\Theta_{\Pi_p}(\mathcal{Q}_f)\Theta_{\Pi_p}(1-\mathcal{S}_p)^{-1}\Lambda_{\Pi_p}(f_0)=\Theta_{\Pi_p}(\mathcal{Q}_f)$. This gives the result.
\end{proof}
\end{prop}

\subsubsection{Necessity of inverting $\scalemath{0.9}{1-\mathcal{S}_p}$ via an $\mathrm{Ext}$-vanishing result of Prasad}\label{sec Ext vanishing} Recall \Cref{thm freeness of P-invariant vectors} which states that the module $C_c^\infty(P_p\backslash G_p/G_p^\circ)[\frac{1}{1-\mathcal{S}_p}]$ is free of rank one over $\mathcal{H}_{G_p}^\circ[\frac{1}{1-\mathcal{S}_p}]$, generated by $\ch(P_pG_p^\circ)$. A natural thing to ask is whether the localisation at the Hecke operator $\scalemath{0.9}{1-\mathcal{S}_p}$ is optimal and necessary. In other words is the module $C_c^\infty(P_p\backslash G_p/G_p^\circ)$ generated by $\ch(P_pG_p^\circ)$ over $\mathcal{H}_{G_p}^\circ$. As the title of this section suggests, the answer is no. Indeed, if $C_c^\infty(P_p\backslash G_p/G_p^\circ)$ was cyclic as a $\mathcal{H}_{G_p}^\circ$-module, then our earlier argument would show that 
$$\dim\ \mathrm{Hom}_{P_p}(\pi_{p,1}\otimes \pi_{p,2},\delta_P)\leq 1$$
for \textit{every} unramified Whittaker type $\pi_{p,1},
\pi_{p,2}$. In particular,
$$\dim\ \mathrm{Hom}_{P_p}(\pi_p\otimes \pi_p^\vee,\delta_P)\leq 1$$
for every unramified irreducible principal-series $\pi_p$ with $\omega_{\pi_p}=1$. By the same argument as the one in \Cref{eq: chain of isos}, one can deduce that the Hom-space above is nothing more than
        $$\mathrm{Hom}_{H_p}\left(\pi_p\otimes \pi_p^\vee\otimes\mathrm{ind}_{B_p}^{H_p}\delta_B^{-1/2},\mathbf{1}\right)= \mathrm{Hom}_{\mathrm{PGL_2(\mathbf{Q}_p})}\left(\pi_p\otimes \pi_p^\vee\otimes\mathrm{ind}_{B_p}^{H_p}\delta_B^{-1/2},\mathbf{1}\right).$$
        It is well-known that the non-Whittaker type principal-series $\mathrm{ind}_{B_p}^{H_p}\delta_B^{-1/2}$ fits into a short exact sequence of $\mathrm{PGL_2(\mathbf{Q}_p})$-modules 
        $$0\longrightarrow \mathbf{1}\longrightarrow \mathrm{ind}_{B_p}^{H_p}\delta_B^{-1/2}\longrightarrow\mathrm{St}\longrightarrow 0.$$
        Tensoring with $\pi_p$ and taking $\mathrm{Ext}_{\mathrm{PGL}_2(\mathbf{Q}_p)}^{*}\left(-,\pi_p\right)$, we obtain a long exact sequence
        
        \begin{align}\label{eq:28}0\longrightarrow \mathrm{Hom}_{\mathrm{PGL_2(\mathbf{Q}_p})}\left(\pi_p\otimes\pi_p^\vee\otimes \mathrm{St},\mathbf{1}\right)&\longrightarrow\mathrm{Hom}_{\mathrm{PGL_2(\mathbf{Q}_p})}\left(\pi_p\otimes \pi_p^\vee\otimes \mathrm{ind}_{B_p}^{H_p}\delta_B^{-1/2},\mathbf{1}\right)\\
        \nonumber&\longrightarrow \mathrm{Hom}_{\mathrm{PGL}_2(\mathbf{Q}_p)}\left(\pi_p\otimes \pi_p^\vee,\mathbf{1}\right)\longrightarrow \mathrm{Ext}_{\mathrm{PGL}_2(\mathbf{Q}_p)}^{1}\left(\pi_p\otimes \mathrm{St},\pi_p\right).
        \end{align}
       By \cite[Proposition $2.6$]{Prasad2018ExtICM} and self-duality of $\mathrm{St}$, the $\mathrm{Ext}$-group appearing above is naturally isomorphic to $\mathrm{Ext}_{\mathrm{PGL}_2(\mathbf{Q}_p)}^{1}\left(\pi_p\otimes \pi_p^\vee,\mathrm{St}\right)$. By the Schneider-Stuhler duality theorem \cite[Theorem $8.1$]{Prasad2018ExtICM} and the fact that $\mathrm{Ext}_{\mathrm{PGL}_2(\mathbf{Q}_p)}^{i}\left(\pi_p\otimes \pi_p^\vee,\mathrm{St}\right)=0$ for $i>1$ (which is the split rank of $\mathrm{PGL_2})$, we have
       $$\mathrm{Ext}_{\mathrm{PGL}_2(\mathbf{Q}_p)}^{1}\left(\pi_p\otimes \pi_p^\vee,\mathrm{St}\right)\simeq \Hom_{\mathrm{PGL}_2(\mathbf{Q}_p)}\left(\mathbf{1},\pi_p\otimes \pi_p^\vee\right).$$
       However, this space vanishes by \cite[Proposition $8.1$]{Prasad2018ExtICM}. The space $
           \mathrm{Hom}_{\mathrm{PGL}_2(\mathbf{Q}_p)}\left(\pi_p\otimes \pi_p^\vee,\mathbf{1}\right)$ is at least one dimensional and $
           \mathrm{Hom}_{\mathrm{PGL}_2(\mathbf{Q}_p)}\left(\pi_p\otimes \pi_p^\vee\otimes \mathrm{St},\mathbf{1}\right)$ is precisely one dimensional,
       where the second fact follows from the main result of \cite{prasad1990trilinear}. Thus going back to \eqref{eq:28}, we finally deduce that 
       $$\dim\ \mathrm{Hom}_{\mathrm{PGL}_2(\mathbf{Q}_p)}\left(\pi_p\otimes \pi_p^\vee\otimes \mathrm{ind}_{B_p}^{H_p}\delta_B^{-1/2},\mathbf{1}\right)\geq 2$$
       Of course this contradicts the bound $\dim\ \mathrm{Hom}_{P_p}\left(\pi_p\otimes \pi_p^\vee,\delta_P\right)\leq 1$. Thus the $\mathcal{H}_{G_p}^\circ$-module $C_c^\infty(P_p\backslash G_p/G_p^\circ)$ is indeed not generated by $\ch(P_pG_p^\circ)$ without localising at the Hecke operator $\scalemath{0.9}{1-\mathcal{S}_p}$.

\subsubsection{Relation to branching laws for $\ell$-modular representations} We fix a prime $\ell\neq p$ and we identify the residue field of $\overline{\mathbf{Z}}_\ell$ with $\overline{\mathbf{F}}_\ell$. We fix a choice of $p^i$ in $\overline{\mathbf
F}_\ell$ for each $i\in\mathbf
{Z}$, compatible with the projection from $\overline{\mathbf{Z}}_\ell$ to $\overline{\mathbf{F}}_\ell$. The character $\delta_P$ is still given (using different notation) by $\delta_P(\left[\begin{smallmatrix}
    a & b\\
     & 1
\end{smallmatrix}\right])=p^{-v_p(a)}$. We have algebra and module structures on $\mathcal{H}_{G_p}^\circ(\overline{\mathbf{F}}_\ell)$, $C_c^\infty(P_p\backslash G_p/G_p^\circ,\overline{\mathbf{F}}_\ell), C_c^\infty(G_p/G_p^\circ,\overline{\mathbf{F}}_\ell)_{P_p,\delta_P}$ by lifting to the corresponding $\overline{\mathbf{Z}}_\ell$-points and then projecting back down. Clearly this is independent of the choice of lifts and thus well-defined. 

As in \cite[\S $3.3$]{groutides2024integral} we say that an unramified $\overline{\mathbf{F}}_\ell$-linear 
 $H_p$-representation $\pi_p$ is \textit{non-degenerate} if $\dim\ (\pi_p)^{H_p^\circ}$ is equal to $1$ and $(\pi_p)^{H_p^\circ}$ generates $\pi_p$. For such a representation we still write $W_{\pi_p}^\mathrm{sph}$ for its spherical vector. The irreducible $\overline{\mathbf{F}}_\ell$-linear unramified principal-series are indeed non-degenerate. This is part of \cite{minguez2014unramified}, but for $\GL_2$, it's been known before that as mentioned in \cite{barthel1995modular}. For the reducible cases, it is a bit more subtle compared to the characteristic zero case (see for example \cite{vigneras1989representations}). 
 \begin{thm}\label{thm modular mult 1 mirabolic}
     Suppose $\ell\nmid p(p-1)$ and let $\pi_{p,1},\pi_{p,2}$ be unramified non-degenerate $\overline{\mathbf{F}}_\ell$-linear $H_p$-representations with $\omega_{\pi_{p,1}}(p)\neq \omega_{\pi_{p,2}}^{-1}(p)$ in $\overline{\mathbf{F}}_\ell^\times$. Then
     $$\dim_{\overline{\mathbf{F}}_\ell}\ \mathrm{Hom}_{P_p}(\pi_{p,1}\otimes\pi_{p,2},\delta_P)\leq 1$$
     and any such non-zero linear form is non-vanishing on $W_{\pi_{p,1}}^\mathrm{sph}\otimes W_{\pi_{p,2}}^\mathrm{sph}$.
     \begin{proof}
         By fixing an isomorphism $\mathbf{C}\simeq\overline{\mathbf{Q}}_\ell$, it follows from \Cref{cor freeness at integral level for P-invariants} that $C_c^\infty(P_p\backslash G_p/G_p^\circ,\overline{\mathbf{Z}}_\ell)[\frac{1}{1-\mathcal{S}_p}]$ is cyclic as a $\mathcal{H}_{G_p}^\circ(\overline{\mathbf{Z}}_\ell)[\frac{1}{1-\mathcal{S}_p}]$-module generated by $\ch(P_pG_p^\circ)$. From this, it is clear that $C_c^\infty(P_p\backslash G_p/G_p^\circ,\overline{\mathbf{F}}_\ell)[\frac{1}{1-\mathcal{S}_p}]$ is also cyclic as a $\mathcal{H}_{G_p}^\circ(\overline{\mathbf{F}}_\ell)[\frac{1}{1-\mathcal{S}_p}]$-module generated by $\ch(P_pG_p^\circ)$. We have a commutative square
         \[\begin{tikzcd}[ampersand replacement=\&,cramped,sep=scriptsize]
	{C_c^\infty(P_p\backslash G_p/G_p^\circ,\overline{\mathbf{Z}}_\ell)} \& {C_c^\infty(G_p/G_p^\circ,\overline{\mathbf{Z}}_\ell)_{P_p,\delta_P}} \\
	{C_c^\infty(P_p\backslash G_p/G_p^\circ,\overline{\mathbf{F}}_\ell)} \& {C_c^\infty(G_p/G_p^\circ,\overline{\mathbf{F}}_\ell)_{P_p,\delta_P}}
	\arrow["{\nu_\ell}", two heads, from=1-1, to=2-1]
	\arrow["{\tau_\ell}"', two heads, from=1-2, to=2-2]
	\arrow["\Phi", hook, from=1-1, to=1-2]
	\arrow["\tilde{\Phi}", from=2-1, to=2-2]
\end{tikzcd}\]
where $\Phi$ is the map from \Cref{prop from inv to coinv} and $\tilde{\Phi}$ is defined by first lifting through $\nu_\ell$ and then applying $\Phi$ and $\tau_\ell$, which is clearly well-defined. From the volume calculations in the proof of \Cref{lem identification of lattices}, the definition of $\Phi$ and the assumption on $\ell\nmid p(p-1)$, it follows that $\tilde{\Phi}$ is surjective. A diagram chase using injectivity of $\Phi$ also shows that $\tilde{\Phi}$ is injective, and finally $\mathcal{H}_{G_p}^\circ(\overline{\mathbf{F}}_\ell)$-equivariance follows from the definitions. Thus, $C_c^\infty(G_p/G_p^\circ,\overline{\mathbf{F}}_\ell)_{P_p,\delta_P}[\frac{1}{1-\mathcal{S}_p}]$ is a cyclic $\mathcal{H}_{G_p}^\circ(\overline{\mathbf{F}}_\ell)[\frac{1}{1-\mathcal{S}_p}]$-module generated by $\ch(G_p^\circ)$. Of course, the linear map $C_c^\infty(G_p/G_p^\circ,\overline{\mathbf{F}}_\ell)\rightarrow \Pi_p=\pi_{p,1}\boxtimes \pi_{p,2}$ can be defined in a completely algebraic manner via $\ch(gG_p^\circ)\mapsto g\cdot W_{\Pi_p}^\mathrm{sph}$. Then the linear map $$C_c^\infty(G_p/G_p^\circ,\overline{\mathbf{F}}_\ell)_{P_p,\delta_P}[\tfrac{1}{1-\mathcal{S}_p}]\longrightarrow (\Pi_p)_{P_p,\delta_P},\ \tfrac{\ch(gG_p^\circ)}{(1-\mathcal{S}_p)^{i}}\mapsto \tfrac{g\cdot W_{\Pi_p}^\mathrm{sph}}{(1-\omega_{\Pi_p}(p))^{i}},\ g\in G_p, i\in\mathbf{Z}_{\geq 0}$$ is surjective. 
By the cyclicity mentioned above and the fact that for such modular representations the space $(\Pi_p)^{G_p^\circ}$ is still one dimensional, its image is generated by the spherical vector and the
result follows.
     \end{proof}
 \end{thm}
 \begin{cor}
     Suppose $\ell\nmid p(p-1)$. Let $\pi_{p,1},\pi_{p,2}$ be unramified non-degenerate $\overline{\mathbf{F}}_\ell$-linear representations of $H_p$ and let $\pi_{p,3}$ be an unramified normalized parabolic induction  $\mathrm{ind}_{B_p}^{H_p}\left[\begin{smallmatrix}
         \chi_{p,1} \\
          & \chi_{p,2}
     \end{smallmatrix}\right]$ with $(\chi_{p,1}/\chi_{p,2})(p)\neq p$ in $\overline{\mathbf{F}}_\ell^\times$. Then
     $$\dim_{\overline{\mathbf{F}}_\ell}\ \mathrm{Hom}_{H_p}(\pi_{p,1}\otimes \pi_{p,2}\otimes \pi_{p,3},\mathbf{1})\leq 1.$$
     If furthermore $\pi_{p,3}$ is also non-degenerate then any such non-zero linear form does not vanish on the spherical vector $W_{\pi_{p,1}}^\mathrm{sph}\otimes W_{\pi_{p,2}}^\mathrm{sph}\otimes W_{\pi_{p,3}}^\mathrm{sph}$.
     \begin{proof}
          Again by twisting we may assume that $\pi_{p,3}$ has trivial central character, i.e. $\chi_{p,2}=\chi_{p,1}^{-1}$ and $\omega_{\pi_{p,1}}\omega_{\pi_{p,2}}=1$. Using the fact that $B_p=Z_p\times P_p$, one can check that $$\pi_{p,3}^\vee\simeq I(\chi_{p,1}^{-1},\chi_{p,1})\simeq \left(\mathrm{Ind}_{P_p}^{H_p}\left(\chi_{p,1}^{-1}\delta_P^{1/2}\right)\right)_1$$where the subscript on the right denotes the eigenspace of the trivial character with respect to the action of the center. Thus, we have a chain of isomorphisms
\begin{align}\label{eq: chain of isos}
     \mathrm{Hom}_{H_p}(\pi_{p,1}\otimes \pi_{p,2}\otimes \pi_{p,3},\mathbf{1})&\simeq\mathrm{Hom}_{H_p}\left(\pi_{p,1}\otimes \pi_{p,2}, I(\chi_{p,1}^{-1},\chi_{p,1})\right)\\
   \nonumber &\simeq \mathrm{Hom}_{P_p}\left(\pi_{p,1}\otimes \pi_{p,2},\chi_{p,1}^{-1}\delta_P^{1/2}\right)\\
\nonumber&\simeq\mathrm{Hom}_{P_p}\left(\pi_{p,1}\otimes\pi_{p,2}\otimes\chi_{p,1}\delta_P^{1/2},\delta_P\right)
\end{align}
where the second isomorphism follows by Frobenius reciprocity for the non-compact induction (e.g \cite[III $2.5$]{renard2010representations}). However, as a $P_p$-representation, $\pi_{p,1}\otimes\pi_{p,2}\otimes \chi_{p,1}\delta_P^{1/2}$ is isomorphic to $\pi_{p,1}\otimes\sigma_{p,2}$ where $\sigma_{p,2}$ is the unramified Whittaker type $H_p$-representation given by $\pi_{p,2}\otimes \chi_{p,1}|\cdot|_p^{1/2}$. By assumption, we have $\chi_{p,1}(p)^2\neq p$ and thus the product of central characters $(\omega_{\pi_{p,1}}\omega_{\sigma_{p,2}})(p)=\chi_{p,1}(p)^2\cdot p^{-1}\neq 1$. Then the bound on the dimension follows from \Cref{thm modular mult 1 mirabolic}. Finally, to show that if such a non-zero linear form exists then it doesn't vanish on the spherical vector (assuming $\pi_{p,3}$ is non-degenerate), one needs to trace through the various isomorphisms of $\mathrm{Hom}$-spaces above (Frobenius reciprocity and duality) while using \cite[Lemme $18$ part $(3)$]{vigneras1989representations} and its proof, and the corresponding result of \Cref{thm modular mult 1 mirabolic}.
     \end{proof}
 \end{cor}

       \begin{rem}
           We would also like to remark that for irreducible representations $\pi_{p,i}$ (even ramified ones) the uniqueness of the triple-product period above, can also be obtained by the proof of \cite{prasad1990trilinear}, which works verbatim in this $\ell\nmid p(p-1)$ modular setting. This has been verified through personal correspondence with Dipendra Prasad.
       \end{rem}

\subsection{Integral lattice in $\mathcal{S}(\mathbf{Q}_p^2)\otimes \pi_{p,1}\otimes \pi_{p,2}$}\label{sec in lattice in S x pi1 x pi2}
 
As usual, we let $\pi_{p,1},\pi_{p,2}$ be unramified $H_p$-representations of Whittaker type. By \Cref{prop JSPS} we know that a non-zero $H_p$-invariant period in $\Hom_{H_p}(\mathcal{S}(\mathbf{Q}_p^2)\otimes \pi_{p,1}\otimes \pi_{p,2},\mathbf{1})$ is unique up to scalars and is non-vanishing on the unramified vector $\ch(\mathbf{Z}_p^2)\otimes W_{\pi_{p,1}}^\mathrm{sph}\otimes W_{\pi_{p,2}}^\mathrm{sph}$. We can thus rescale it so it maps $\ch(\mathbf{Z}_p^2)\otimes W_{\pi_{p,1}}^\mathrm{sph}\otimes W_{\pi_{p,2}}^\mathrm{sph}$ to $1$ and we call such a period \textit{normalized}. In fact we will see later, when we discuss an application of this to global Rankin-Selberg integrals, that such a non-zero period always exists and we will explicitly write it down. For now though, we need only consider an abstract period of this form, without any sort of realization of it. 
\begin{thm}\label{thm integral JPSS periods}
    Let $\mathcal{Z}$ be the non-zero normalized $H_p$-invariant period in $\Hom_{H_p}(\mathcal{S}(\mathbf{Q}_p^2)\otimes \pi_{p,1}\otimes \pi_{p,2},\mathbf{1})$. Assume that the spherical Hecke eigensystem of $\pi_{p,1}\boxtimes \pi_{p,2}$ restricts to a morphism $\mathcal{H}_{G_p}^\circ(\mathcal{R})\rightarrow\mathcal{R}$ for some $\mathbf{Z}[1/p]$-algebra $\mathcal{R}\subseteq\mathbf{C}$. Then:
    \begin{enumerate}
        \item The period $\mathcal{Z}$ is valued in $\mathcal{R}$ on the lattice 
        \begin{align*}
           \scalemath{0.9}{ \mathrm{span}_{\mathcal{R}}\left\{\phi\otimes g\left(W_{\pi_{p,1}}^\mathrm{sph}\otimes W_{\pi_{p,2}}^\mathrm{sph}\right)\ |\ g\in G_p\ \mathrm{and}\ \phi\in\mathcal{S}(\mathbf{Q}_p^2)\ \mathrm{valued}\ \mathrm{in}\ \frac{1}{\vol_{H_p}\left(\mathrm{Stab}_{H_p}(\phi)\cap g G_p^\circ g^{-1}\right)}\mathcal{R}\right\}.}
        \end{align*}
        \item The period $\mathcal{Z}$ is valued in the ideal $\left(p-1,L(\pi_{p,1}\boxtimes \pi_{p,2},0)^{-1}\right)\subseteq\mathcal{R}$ on the lattice 
        \begin{align*}
           \scalemath{0.9}{ \mathrm{span}_{\mathcal{R}}\left\{\phi\otimes g\left(W_{\pi_{p,1}}^\mathrm{sph}\otimes W_{\pi_{p,2}}^\mathrm{sph}\right)\ |\ g\in G_p\ \mathrm{and}\ \phi\in\mathcal{S}_0(\mathbf{Q}_p^2)\ \mathrm{valued}\ \mathrm{in}\ \frac{1}{\vol_{H_p}\left(\mathrm{Stab}_{H_p}(\phi)\cap g G_p^\circ[p] g^{-1}\right)}\mathcal{R}\right\}.}
        \end{align*}
        \item If $\omega_p:=\omega_{\pi_{p,1}}\omega_{\pi_{p,2}}\neq 1$, the period $\mathcal{Z}$ satisfies
        $$L(\omega_p,0)^{-1}\mathcal{Z}(-)\in \left(p-1,L(\pi_{p,1}\boxtimes \pi_{p,2},0)^{-1}\right)\subseteq\mathcal{R}$$
        on the lattice
        \begin{align*}
           \scalemath{0.9}{ \mathrm{span}_{\mathcal{R}}\left\{\phi\otimes g\left(W_{\pi_{p,1}}^\mathrm{sph}\otimes W_{\pi_{p,2}}^\mathrm{sph}\right)\ |\ g\in G_p\ \mathrm{and}\ \phi\in\mathcal{S}(\mathbf{Q}_p^2)\ \mathrm{valued}\ \mathrm{in}\ \frac{1}{\vol_{H_p}\left(\mathrm{Stab}_{H_p}(\phi)\cap g G_p^\circ[p] g^{-1}\right)}\mathcal{R}\right\}.}
        \end{align*}
    \end{enumerate}
    \begin{proof}
        The proof relies on the three parts of \Cref{thm local norm relations}, in particular the integral behavior of the Hecke operators $\mathcal{P}_\delta$ attached to integrals elements $\delta$. To obtain the result one combines this with \cite[Proposition $4.4.1$]{Loeffler_2021}, whose proof is formal and adapts to our setup in an identical fashion. 
    \end{proof}
\end{thm}
\begin{rem}
    Notice how parts $(2)$ and $(3)$ compared to part $(1)$, of \Cref{thm integral JPSS periods}, have the determinant level condition $G_p^\circ[p]$. Also, note that part $(3)$  compared to part $(2)$ allows for Schwartz functions which might not vanish at the origin, at the expense of a potential Dirichlet $L$-factor popping up.
\end{rem}
\begin{rem}
    The assumption on the spherical Hecke eigensystem of $\pi_{p,1}\boxtimes \pi_{p,2}$ is a natural one to make from a global viewpoint as we will see later when we discuss an application of this to global Rankin-Selberg integrals for product of modular forms.
    \end{rem}
\noindent 
\subsection{Integral lattice in $\pi_{p,1}\otimes \pi_{p,2}\otimes \pi_{p,3}$}\label{sec integral lattice in pi1 x pi2 x pi3}
For notational convenience we now let $\pi_{p,3}$ be an unramified Whittaker type $H_p$-representation of the form $I(\chi_p,|\cdot|_p^{-1/2})$ such that $\omega_{\pi_{p,1}}\omega_{\pi_{p,2}}\omega_{\pi_{p,3}}=1$. From \Cref{thm trilinear forms prasad} we know that a non-zero invariant period in $\Hom_{H_p}(\pi_{p,1}\otimes \pi_{p,2}\otimes \pi_{p,3},\mathbf{1})$ is unique up to scalars and is non vanishing on the spherical vector $W_{\pi_{p,1}}^\mathrm{sph}\otimes W_{\pi_{p,2}}^\mathrm{sph}\otimes W_{\pi_{p,3}}^\mathrm{sph}$. We can thus rescale it to map this spherical vector to $1$ and call such a period \textit{normalized}. In fact by \cite{harris2001note} such a period always exists. Once again we can make use of the $H_p$-equivariant map $$\mathcal{S}(\mathbf{Q}_p^2)\rightarrow \pi_{p,3},\ \phi\mapsto f_{\hat{\phi},\chi_p}$$ constructed in \cite{Loeffler_Skinner_Zerbes_2021}, which we describe in this section for completeness: For a Schwartz function $\phi\in\mathcal{S}(\mathbf{Q}_p^2)$, we write $\hat{\phi}$ for its Fourier transform 
$$\hat{\phi}(x,y):=\int_{\mathbf{Q}_p}\int_{\mathbf{Q}_p}e_p(xv-yu)\phi(u,v)\ dudv$$
where $e_p$ is the standard additive character of $\mathbf{Q}_p$ mapping $1/p^n$ to $\mathrm{exp}(2\pi i/p^n)$ (this is an explicit description of the character $\psi$ fixed at the beginning of \Cref{sec Hecke forms}) and $du, dv$ denote the usual additive Haar measure on $\mathbf{Q}_p$. Since $\pi_{p,3}$ is unramified, the character $\chi_p$ is of the form $|\cdot|_p^a$ for some $a\in\mathbf{C}$. Let $b=a+\tfrac{1}{2}$, then by \cite[Proposition $3.2.2$]{Loeffler_Skinner_Zerbes_2021},
$$f_{\phi,\chi_p}(h):=\lim_{s\rightarrow 0} \frac{|\det(h)|_p^{s+a+\tfrac{1}{2}}}{L\left(|\cdot|_p^{b},2s+1\right)}\int_{\mathbf{Q}_p^\times}\phi((0,x)h)|x|_p^{2s+b+1}\ d^\times x,\ h\in H_p$$
is well-defined and gives an element of $I(\chi_p,|\cdot|_p^{-1/2})$. Furthermore, the assignment $\phi\mapsto f_{\hat{\phi},\chi_p}$ is surjective, $H_p$-equivariant and by \cite[Lemma $3.2.5$]{Loeffler_Skinner_Zerbes_2021} it maps $\ch(\mathbf{Z}_p^2)$ to $W_{\pi_{p,3}}^\mathrm{sph}$.
\begin{cor}
    Let $\mathcal{Z}$ be the non-zero normalized invariant period in $\Hom_{H_p}(\pi_{p,1}\otimes \pi_{p,2}\otimes \pi_{p,3},\mathbf{1})$. Assume that the spherical Hecke eigensystem of $\pi_{p,1}\boxtimes \pi_{p,2}$ restricts to a morphism $\mathcal{H}_{G_p}^\circ(\mathcal{R})\rightarrow\mathcal{R}$ for some $\mathbf{Z}[1/p]$-algebra $\mathcal{R}\subseteq\mathbf{C}$. Then:
    \begin{enumerate}
        \item The period $\mathcal{Z}$ is valued in $\mathcal{R}$ on the lattice 
        \begin{align*}
            \scalemath{0.9}{\mathrm{span}_{\mathcal{R}}\left\{ g\left(W_{\pi_{p,1}}^\mathrm{sph}\otimes W_{\pi_{p,2}}^\mathrm{sph}\right)\otimes f_{\hat{\phi},\chi_p}\ |\ g\in G_p\ \mathrm{and}\ \phi\in\mathcal{S}(\mathbf{Q}_p^2)\ \mathrm{valued}\ \mathrm{in}\ \frac{1}{\vol_{H_p}\left(\mathrm{Stab}_{H_p}(\phi)\cap g G_p^\circ g^{-1}\right)}\mathcal{R}\right\}.}
        \end{align*}
        \item The period $\mathcal{Z}$ is valued in the ideal $\left(p-1,L(\pi_{p,1}\boxtimes \pi_{p,2},0)^{-1}\right)\subseteq\mathcal{R}$ on the lattice 
        \begin{align*}
           \scalemath{0.9}{ \mathrm{span}_{\mathcal{R}}\left\{ g\left(W_{\pi_{p,1}}^\mathrm{sph}\otimes W_{\pi_{p,2}}^\mathrm{sph}\right)\otimes f_{\hat{\phi},\chi_p} \ |\ g\in G_p\ \mathrm{and}\ \phi\in\mathcal{S}_0(\mathbf{Q}_p^2)\ \mathrm{valued}\ \mathrm{in}\ \frac{1}{\vol_{H_p}\left(\mathrm{Stab}_{H_p}(\phi)\cap g G_p^\circ[p] g^{-1}\right)}\mathcal{R}\right\}.}
        \end{align*}
        \item If $\omega_p:=\omega_{\pi_{p,1}}\omega_{\pi_{p,2}}\neq 1$, the period $\mathcal{Z}$ satisfies
        $$L(\omega_p,0)^{-1}\mathcal{Z}(-)\in \left(p-1,L(\pi_{p,1}\boxtimes \pi_{p,2},0)^{-1}\right)\subseteq\mathcal{R}$$
        on the lattice
        \begin{align*}
            \scalemath{0.9}{\mathrm{span}_{\mathcal{R}}\left\{ g\left(W_{\pi_{p,1}}^\mathrm{sph}\otimes W_{\pi_{p,2}}^\mathrm{sph}\right)\otimes f_{\hat{\phi},\chi_p} \ |\ g\in G_p\ \mathrm{and}\ \phi\in\mathcal{S}(\mathbf{Q}_p^2)\ \mathrm{valued}\ \mathrm{in}\ \frac{1}{\vol_{H_p}\left(\mathrm{Stab}_{H_p}(\phi)\cap g G_p^\circ[p] g^{-1}\right)}\mathcal{R}\right\}.}
        \end{align*}
        \end{enumerate}
        \begin{proof}
            This follows from \Cref{thm integral JPSS periods} and the map induced by $f_{(\hat{-}),\chi_p}$ at the level of $\mathrm{Hom}$-spaces.
        \end{proof}
\end{cor}

\subsection{Rankin-Selberg integral for product of modular forms}\label{rankin selber integrals}
\subsubsection{The local integrals}\label{sec local integrals}
We once again fix the additive character $\psi:\mathbf{Q}_p\rightarrow\mathbf{C}^\times$ of conductor one from previous sections. As before, we consider the family of $G_p$-representations $\pi_{p,1}\boxtimes \pi_{p,2}$ where $\pi_{p,1}$ and $\pi_{p,2}$ are unramified $H_p$-representations of Whittaker type, and we identify $\pi_{p,1}$ with $\mathcal{W}(\pi_{p,1},\psi)$ and $\pi_{p,2}$ with $\mathcal{W}(\pi_{p,2},\overline{\psi})$.

\begin{defn}[\cite{jacquet1983rankin}]\label{def zeta integral}
    For $W_1\in\mathcal{W}(\pi_{p,1},\psi), W_2\in\mathcal{W}(\pi_{p,2},\overline{\psi})$ and $\phi\in\mathcal{S}(\mathbf{Q}_p^2)$, the zeta integral $Z(W_1,W_2,\phi,s)$ is given by
    \begin{align}
        Z(\phi,W_1,W_2,s)=\int_{N_p\backslash H_p}W_1(h)W_2(h)\phi((0,1)h)|\mathrm{det}(h)|_p^s\ dh
    \end{align}
\end{defn}
\begin{prop}\label{prop convergence}
    \begin{enumerate}
        \item The zeta integral of \emph{\Cref{def zeta integral}} converges for $\Re(s)$ large enough and admits unique meromorphic continuation as a rational function of $p^s$.
        \item As $W_1,W_1$ and $\phi$ vary, the $Z(W_1.W_2,\phi,s)$ generate the fractional ideal $L(\pi_{p,1}\boxtimes \pi_{p,2},s)\mathbf{C}[p^s,p^{-s}]$.
        
        \end{enumerate}
        \begin{proof}
            This is one of the main results of \cite{jacquet1983rankin} where it is stated and proved in greater generality.
        \end{proof}
\end{prop}
\begin{prop}\label{prop ev at spherical vector}
    Let $\phi_0=\ch(\mathbf{Z}_p^2)$. Then, 
    $$Z(\phi_0,W_{\pi_{p,1}}^\mathrm{sph},W_{\pi_{p,1}}^\mathrm{sph},s)=L(\pi_{p,2}\boxtimes\pi_{p,2},s).$$
    \begin{proof}
        This is well know and the proof uses \cite{shintani1976explicit} to evaluate the spherical Whittaker functions in terms of Satake parameters.
    \end{proof}
\end{prop}
\begin{defn}\label{def linear form} Let $\pi_{p,1}, \pi_{p,2}$ be as above. 
    The linear form $\mathcal{Z}:\mathcal{S}(\mathbf{Q}_p^2)\otimes \pi_{p,1}\otimes \pi_{p,2} \rightarrow\mathbf{C}$ is given by
    $$\mathcal{Z}(\phi\otimes W_1\otimes W_1)=\lim_{s\rightarrow 0}\frac{Z(\phi,W_1,W_2,s)}{L(\pi_{p,1}\boxtimes \pi_{p,2},s)}.$$
\end{defn}
\noindent By \Cref{prop convergence}, the linear form $\mathcal{Z}$ is well defined. Additionally, \Cref{prop ev at spherical vector} shows that $\mathcal{Z}$ takes the value $1$ on $\phi_0\otimes W_{\pi_{p,1}}^\mathrm{sph}\otimes W_{\pi_{p,2}}^\mathrm{sph}$. A standard change of variables also shows that $\mathcal{Z}$ is in fact $H_p$-invariant. Putting everything together, the linear form $\mathcal{Z}$ gives a non-zero invariant period in the space $\Hom_{H_p}(\mathcal{S}(\mathbf{Q}_p^2)\otimes \pi_{p,1}\otimes \pi_{p,2},\mathbf{1})$, which by \Cref{prop JSPS}, is in fact a basis.
\subsubsection{The unramified period for a product of modular forms}
Let $f_1,f_2$ be normalized cuspidal Hecke eigenforms of even integral weights $k_1,k_2$, levels $N_1,N_2$ and Nebentypes $\epsilon_1,\epsilon_2$. Write $\pi_{f_i}$ for the corresponding cuspidal automorphic representation attached to $f_i$ (\cite{gelbart2006automorphic}). Let $S$ be a finite set of places containing $\infty$ and all the prime divisors of $N_1N_2$. Then, $\pi_{f_1}^S\boxtimes\pi_{f_2}^S$ is unramified and contains a unique normalized spherical vector $W_{\pi_{f_1}}^\mathrm{sph}\otimes W_{\pi_{f_2}}^\mathrm{sph}$. The unramified Rankin-Selberg period associated to $f_1\times f_2$ is the unique normalized linear form
 $$\mathcal{Z}_{f_1\times f_2}\in\mathrm{Hom}_{\GL_2(\mathbf{A}^S)}\left(\mathcal{S}((\mathbf{A}^S)^2)\otimes \pi_{f_1}^S\otimes\pi_{f_2}^S,\mathbf{1}\right)$$
 in this one-dimensional Hom-space.
 We fix a prime $\ell\nmid N_1N_2$ and an isomorphism $\iota:\mathbf{C}\simeq \overline{\mathbf{Q}}_\ell$.  

\begin{thm}\label{thm global rankin selberg}
     Let $\ell\nmid N_1N_2$ be a fixed prime and fix $\mathbf{C}\simeq_\iota\overline{\mathbf{Q}}_\ell$. Let $\mathbf{L}_{f_1,f_2}$ be the smallest $\ell$-adic number field containing the composite of the number fields of $f_1$ and $f_2$ under $\iota$. Then the following are true: 
    \begin{enumerate}
        \item For any $g_1,g_2\in\GL_2(\mathbf{A}^S)$ and any decomposable Schwartz function $\Phi=\otimes_{p\not\in S}\Phi_p\in \mathcal{S}((\mathbf{A}^S)^2)$ where each $\Phi_p$ is valued in $\vol_{H_p}(\mathrm{Stab}_{H_p}(\Phi_p)\cap g_pG_p^\circ g_p^{-1})^{-1}\mathcal{O}_{\mathbf{L}_{f_1,f_2}}$, the period $\mathcal{Z}_{f_1\times f_2}$ satisfies $$\mathcal{Z}_{f_1\times f_2}(\Phi,g_1W_{\pi_{f_1}^S}^\mathrm{sph},g_2W_{\pi_{f_2}^S}^\mathrm{sph})\in\mathcal{O}_{\mathbf{L}_{f_1,f_2}}.$$ 
        \item Suppose, moreover, that $S_0$ denotes a finite set of primes disjoint from $S$, and for each $p\in S_0$, $\Phi_p$ is valued in $\vol_{H_p}(\mathrm{Stab}_{H_p}(\Phi_p)\cap g_pG_p^\circ[p]g_p^{-1})^{-1}\mathcal{O}_{\mathbf{L}_{f_1,f_2}}$ \emph{(note the addition of the determinant level $G_p^\circ[p]$ instead of $G_p^\circ$)}. Then, the period $\mathcal{Z}_{f_1\times f_2}$ also satisfies
        $$ \mathcal{Z}_{f_1\times f_2}(\Phi,g_1W_{\pi_{f_1}^S}^\mathrm{sph},g_2W_{\pi_{f_2}^S}^\mathrm{sph})\cdot \left(\prod_{\substack{p\in S_0\ \mathrm{s.t:}\\ \Phi_p(0,0)\neq 0}} L_p(\epsilon_1\epsilon_2,0)^{-1}\right)\in \prod_{\substack{p\in S_{0}}}\left(p-1,L_p(\pi_{f_1}\boxtimes \pi_{f_2},0)^{-1}\right)\subseteq\mathcal{O}_{\mathbf{L}_{f_1,f_2}}.$$
        In particular, if the level of $f_1,f_2$ satisfies $(\#(\mathbf{Z}/N_1 N_2\mathbf{Z})^\times,\ell)=1$, then the containment holds without the bracketed product of abelian $L$-factors.
    \end{enumerate}

        \begin{proof}
        The linear form $\mathcal{Z}_{f_1\times f_2}$ can be realized as 
        $$\mathcal{Z}_{f_1\times f_2}(\Phi,g_1W_{\pi_{f_1}^S}^\mathrm{sph},g_2W_{\pi_{f_2}^S}^\mathrm{sph})=\prod_{p\not\in S}\mathcal{Z}(\Phi_p\otimes g_{1,p}W_{\pi_{f_1,p}}^\mathrm{sph}\otimes g_{2,p}W_{\pi_{f_2,p}}^\mathrm{sph})$$
        where $\mathcal{Z}$ is as in \Cref{def linear form}. We are now in a position to apply \Cref{thm integral JPSS periods}, at each prime $p\not\in S$ for which the expression inside the product is not $1$, with the $\mathbf{Z}[1/p]$-algebra $\mathcal{R}$ taken to be $\mathcal{O}_\mathbf{L}$. Thus, it suffices to show that the local condition on the spherical Hecke eigensystem of each local component $\pi_{f_1,p}\boxtimes \pi_{f_2,p}$ is satisfied. Indeed, this follows similarly to \cite[Theorem $8.1.1$]{groutides2024integral} using as mentioned before, the normalizations of \cite{loeffler2012computation}. If the last assumption on the level of $f_1,f_2$ holds then the bracketed integral on the left is a unit in $\mathcal{O}_\mathbf{L}$ which concludes the proof.
    \end{proof}
\end{thm}

\appendix
\section{Structure of $C_c^\infty(P_p\backslash G_p/G_p^\circ)[\tfrac{1}{1-\mathcal{S}_p}]$}\label{appendix}
This appendix is dedicated to the proof of \Cref{thm freeness of P-invariant vectors}. As such we refer the reader back to \Cref{sec Hecke forms} for notation and definitions. 

\begin{thm}
     The module $C_c^\infty(P_p\backslash G_p/G_p^\circ)[\tfrac{1}{1-\mathcal{S}_p}]$ is free of rank one over $\mathcal{H}_{G_p}^\circ[\tfrac{1}{1-\mathcal{S}_p}]$, generated by the characteristic function $f_0:=\ch(P_pG_p^\circ)$. In particular, for $f\in C_c^\infty(P_p\backslash G_p/G_p^\circ)$, we have $f=\tfrac{-\mathcal{S}_p\mathcal{Q}_f^{'}}{1-\mathcal{S}_p}\cdot f_0$, where $\mathcal{Q}_f$ is the Hecke operator described in \emph{\Cref{sec inverting}}.
    \begin{proof} We firstly show that the module in question is cyclic with $f_0$ as a generator.
        By \eqref{eq: inverting}, it suffices to show that the intersection of the kernels of the linear forms $\Lambda_{\Pi_p}$ as we vary the representation $\Pi_p$, is trivial. For an element $f$ in this intersection, we have $\Lambda_{\Pi_p}(f)=\Theta_{\Pi_p}(\mathcal{Q}_f)=0$ for every $\Pi_p$. By the usual density argument, this implies that $\mathcal{Q}_f$ is identically zero in $\mathcal{H}_{G_p}^\circ$. We now claim that $\mathcal{Q}_f=0$ implies that $f=0$. Firstly, for $m,n,r,t\in\mathbf{Z}$,  we put $f_{m,n}^{r,t}:=\ch\left(P_p p^r\left(p^t\left[\begin{smallmatrix}
            p^m & \\
            & 1
        \end{smallmatrix}\right]\left[\begin{smallmatrix}
            1 & p^{n}\\
            & 1
        \end{smallmatrix}\right],1\right) G_p^\circ\right)$ in $C_c^\infty(P_p\backslash G_p/G_p^\circ)$. We can always write $f$ uniquely as 
        $$f=\left(\sum_{i=1}^N a_i f_{m_i,0}^{r_i,t_i}\right)+\left(\sum_{j=1}^M b_j f_{\mu_j,\nu_j}^{\rho_j,\tau_j}\right)$$
        where:
        \begin{itemize}
            \item $|m_0|\leq |m_1|\leq  \dots \leq |m_N| $; if $|m_i|=|m_{i+1}|$ then $r_i\leq r_{i+1}$; if $|m_i|=|m_{i+1}|$ and $r_i=r_{i+1}$ then $t_i\leq t_{i+1}$.
            \item For each $j$, $\nu_j<-\mathrm{max}\{\mu_j,0\}$. 
            \item All the $a_i$ and $b_j$ are non-zero.
        \end{itemize}
        By \cref{prop explicit formula for Lambda} we have 
        \begin{align}\label{eq:18}
            \mathcal{Q}_f=\left(\sum_{i=1}^N a_i\mathscr{V}_{r_i,t_i}\mathscr{P}_{m_i}\right)+\left(\sum_{j=1}^M b_j \mathscr{V}_{\rho_j,\tau_j}\left(\mathscr{P}_{\mu_j}+\mathscr{Q}_{\mu_j,\nu_j}\mathcal{P}_p(1)\right)\right)=0
        \end{align}
        in $\mathcal{H}_{G_p}^\circ=\mathbf{C}[\mathcal{S}_{p,1}^{\pm 1},\mathcal{T}_{p,1},\mathcal{S}_{p,2}^{\pm 1},\mathcal{T}_{p,2}]$. The condition $\nu_j<-\mathrm{max}\{\mu_j,0\}$ guarantees that each $\mathscr{D}_{\mu_j,\nu_j}$ is identically non-zero by \Cref{prop explicit formula for Lambda}. 
        We write  \eqref{eq:18} as 
        \begin{align}\label{eq:19}
            \left(\sum_{i=1}^N a_i\mathscr{V}_{r_i,t_i}\mathscr{P}_{m_i}\right)+\left(\sum_{j=1}^M b_j \mathcal{V}_{\rho_j,\tau_j}\mathscr{P}_{\mu_j}\right)+\left(\sum_{j=1}^Mb_j \mathscr{V}_{\rho_j,\tau_j}\mathscr{Q}_{\mu_j,\nu_j}\mathcal{P}_p(1)\right)=0
        \end{align}
        If the rightmost sum in \eqref{eq:19} is non-zero, then $\mathcal{P}_p(1)$ divides $\left(\sum_{i=1}^N a_i\mathscr{V}_{r_i,t_i}\mathscr{P}_{m_i}\right)+\left(\sum_{j=1}^M b_j \mathscr{V}_{\rho_j,\tau_j}\mathscr{P}_{\mu_j}\right)$.
        \textbf{Claim $0$:} We now claim that $\mathcal{P}_p(1)$ does not divide $\left(\sum_{i=1}^N a_i\mathscr{V}_{r_i,t_i}\mathscr{P}_{m_i}\right)+\left(\sum_{j=1}^M b_j \mathscr{V}_{\rho_j,\tau_j}\mathscr{P}_{\mu_j}\right)$.
        
        In particular, we show that $\mathcal{P}_p(1)$ never divides a non-zero finite sum of the form $\sum c\ (\mathcal{S}_{p,1}\mathcal{S}_{p,2})^s\mathcal{S}_{p,1}^k \mathscr{P}_\lambda$. Consider an arbitrary sum of this form, which we write as 
        \begin{align}\label{eq:20}\Delta:=\left(p^{-n/2}\mathcal{P}_{p,1}^{(n)}(1)\sum_{i,j}c_{i,j}(\mathcal{S}_{p,1}\mathcal{S}_{p,2})^{s_i}\mathcal{S}_{p,1}^{k_j}\right)&+\left(p^{-n/2}\mathcal{P}_{p,2}^{(n)}(1)\sum_{i,j}d_{i,j}(\mathcal{S}_{p,1}\mathcal{S}_{p,2})^{\tilde{s}_i}\mathcal{S}_{p,1}^{\tilde{k}_j}\right)\\
        &+ \text{terms of this form involving }  \scalemath{0.9}{\mathcal{P}_{p,1}^{(e_1)}(1),\mathcal{P}_{p,2}^{(e_2)}(1)}\text{ with }  e_1,e_2<n.
        \end{align}
It is straightforward to compute, using the explicit description of the spherical Hecke eigensystem of the representations $\Pi_p$, that
        \begin{itemize}
            \item $\mathcal{P}_p(1)=1-\frac{1}{p}\mathcal{T}_{p,1}\mathcal{T}_{p,2}-\frac{1}{p}\mathcal{S}_{p,1}\mathcal{S}_{p,2}\mathcal{T}_{p,1}\mathcal{T}_{p,2}+\frac{1}{p}\mathcal{S}_{p,1}\mathcal{T}_{p,2}^2+\frac{1}{p}\mathcal{S}_{p,2}\mathcal{T}_{p,1}^2-2\mathcal{S}_{p,1}\mathcal{S}_{p,2}+\mathcal{S}_{p,1}^2\mathcal{S}_{p,2}^2$. Note that $\mathcal{P}_p(1)=\mathcal{P}_{p,1}^{(1)}(1)=\mathcal{P}_{p,2}^{(1)}(1)$ by \Cref{rem :4.3.4}.
            \item $\frac{1}{p^{n/2}}\mathcal{P}_{p,1}^{(n)}(1)=\frac{1}{p^{\lfloor n/2 \rfloor}}\mathcal{T}_{p,1}^{n-1}-\frac{1}{p^{\lfloor n/2 \rfloor}}\mathcal{S}_{p,1}\mathcal{T}_{p,2}\mathcal{T}_{p,1}^{n-2}+O(\mathcal{T}_{p,1}^{n-3})$, for $n\geq 2$ 
            \item $\frac{1}{p^{n/2}}\mathcal{P}_{p,2}^{(n)}(1)=\frac{1}{p^{\lfloor n/2 \rfloor}}\mathcal{T}_{p,2}^{n-1}-\frac{1}{p^{\lfloor n/2 \rfloor}}\mathcal{S}_{p,2}\mathcal{T}_{p,1}\mathcal{T}_{p,2}^{n-2}+O(\mathcal{T}_{p,2}^{n-3})$, for $n\geq 2$.
        \end{itemize}
        where $O(\mathcal{T}_{p,i}^{e})$ denotes a linear combination of terms of the form $\mathcal{S}_{p,1}^a\mathcal{S}_{p,2}^b\mathcal{T}_{p,i}^c$ with $0\leq c\leq e$.
        From this, by looking at the degrees of the $\mathcal{T}_{p,i}$ terms, we can assume that $n\geq 3$ in \eqref{eq:20}, since otherwise the non-divisibility claim holds trivially. It is also important to note that $\Delta$ has no monomials containing $\mathcal{T}_{p,1}^a\mathcal{T}_{p,2}^b$ with both $a,b\geq 2$. If $\mathcal{P}_p(1)$ divides $\Delta$ in $\mathcal{H}_{G_p}^\circ$, then it is not hard to see that we can write 
        $$\Delta=\left(A_1\mathcal{T}_{p,1}^{n-3}+A_2\mathcal{T}_{p,1}^{n-4}\mathcal{T}_{p,2}+\dots +A_{n-3}\mathcal{T}_{p,1}\mathcal{T}_{p,2}^{n-4}+A_{n-2}\mathcal{T}_{p,2}^{n-3}+C\right)\mathcal{P}_p(1)$$
        where $A_i\in\mathbf{C}[\mathcal{S}_{p,1}^{\pm 1},\mathcal{S}_{p,2}^{\pm 1}]$, all the exponents of the $\mathcal{T}_{p,i}$ ($i=1,2$) are non-negative, and $C$ is a  $\mathbf{C}[\mathcal{S}_{p,1}^{\pm 1},\mathcal{S}_{p,2}^{\pm 1}]$-linear combination of monomials containing only $\mathcal{T}_{p,1}^a\mathcal{T}_{p,2}^b$ with $a+b<n-3$, and $a,b\geq 0.$\\
        Comparing coefficients of $\mathcal{T}_{p,1}^{n-1}$ and $\mathcal{T}_{p,2}^{n-1}$ in $\Delta$, we get the relations
        \begin{align}\label{eq:21}
            A_1=\frac{1}{p^{\lfloor n/2 \rfloor -1}}\sum_{i,j}c_{i,j}\mathcal{S}_{p,1}^{s_i+k_j}\mathcal{S}_{p,2}^{s_i-1}\ \text{ and }\ A_{n-2}=\frac{1}{p^{\lfloor n/2 \rfloor -1}}\sum_{i,j}d_{i,j}\mathcal{S}_{p,1}^{\tilde{s}_i+\tilde{k}_j-1}\mathcal{S}_{p,2}^{\tilde{s}_i}.
        \end{align}
        The monomials that contribute to $\frac{1}{p^{\lfloor n/2\rfloor}}\mathcal{S}_{p,1}\mathcal{T}_{p,2}\mathcal{T}_{p,1}^{n-2}$ in $\Delta$ are $A_1\mathcal{T}_{p,1}^{n-3}$ and $A_2\mathcal{T}_{p,1}^{n-4}\mathcal{T}_{p,2}$. Similarly, the monomials contributing to $\frac{1}{p^{\lfloor n/2\rfloor}}\mathcal{S}_{p,2}\mathcal{T}_{p,1}\mathcal{T}_{p,2}^{n-2}$ in $\Delta$ are $A_{n-2}\mathcal{T}_{p,2}^{n-3}$ and $A_{n-3}\mathcal{T}_{p,1}\mathcal{T}_{p,2}^{n-4}$. Comparing coefficients of $\mathcal{T}_{p,2}\mathcal{T}_{p,1}^{n-2}$ and $\mathcal{T}_{p,1}\mathcal{T}_{p,2}^{n-2}$, we get the relations 
        \begin{align}\label{eq:22}
            -\frac{1}{p}(1+\mathcal{S}_{p,1}\mathcal{S}_{p,2})A_1+\frac{1}{p}\mathcal{S}_{p,2}A_2&=\frac{-1}{p^{\lfloor n/2 \rfloor}}\sum_{i,j}c_{i,j}\mathcal{S}_{p,1}^{s_i+k_j+1}\mathcal{S}_{p,2}^{s_i}\\
            \nonumber \frac{1}{p}(1+\mathcal{S}_{p,1}\mathcal{S}_{p,2})A_{n-2}+\frac{1}{p}\mathcal{S}_{p,1}A_{n-3}&= \frac{-1}{p^{\lfloor n/2 \rfloor}}\sum_{i,j}d_{i,j}\mathcal{S}_{p,1}^{\tilde{s}_i+\tilde{k}_j}\mathcal{S}_{p,2}^{\tilde{s}_i+1}.
        \end{align}
        Combining \eqref{eq:21} and \eqref{eq:22} and simplifying, we obtain the relations
        \begin{align}\label{eq:24}
            A_1=\mathcal{S}_{p,2}A_2\ \text{ and }\ A_{n-2}=\mathcal{S}_{p,2}A_{n-3}.
        \end{align}
        We observe that for $n=3$ or $4$, \eqref{eq:24} already implies our claim that $\mathcal{P}_p(1)|\Delta$ if and only if $\Delta=0$. Thus we may assume that $n\geq 5$. Comparing coefficients of $\mathcal{T}_{p,1}^a\mathcal{T}_{p,2}^b$ with $a+b=n-1$ and $a,b\geq 2$ (which as we remarked earlier, they vanish), we obtain the following system of equations:
        \begin{align}\label{eq:25}
            \mathcal{S}_{p,1}A_1-(1+\mathcal{S}_{p,1}\mathcal{S}_{p,2})A_2+\mathcal{S}_{p,2}A_3&=0\\
           \nonumber \mathcal{S}_{p,1}A_2-(1+\mathcal{S}_{p,1}\mathcal{S}_{p,2})A_3+\mathcal{S}_{p,2}A_4&=0\\
          \nonumber  \vdots\\
           \nonumber \mathcal{S}_{p,1}A_{n-4}-(1+\mathcal{S}_{p,1}\mathcal{S}_{p,2})A_{n-3}+\mathcal{S}_{p,2}A_{n-2}&=0.
        \end{align}
        A lengthy but elementary substitution exercise, together with \eqref{eq:24}, yields that 
        $$A_2=\mathcal{S}_{p,1}\mathcal{S}_{p,2}A_2\ \text{for}\ n=5$$
        $$A_2=\mathcal{S}_{p,2}^{n-5}A_{n-3}\ \text{and}\ A_{n-3}=\mathcal{S}_{p,1}^{n-5} A_2\ \text{for}\ n\geq 6.$$
        Combining this with \eqref{eq:24} and \eqref{eq:25}, we deduce that $A_1=\dots =A_{n-2}=0$, which finishes the proof of \textbf{Claim $0$}.
        
        Thus, \eqref{eq:19} now implies that $S:=\sum_{j=1}^Mb_j \mathscr{V}_{\rho_j,\tau_j}\mathscr{Q}_{\mu_j,\nu_j}$ is identically zero. Up to re-ordering we may assume that the tuple $(-\nu_M,-\mu_M, \rho_M, \tau_M)$ is maximal under lexicographic ordering, among all such tuples for each $j=0,\dots,M$.\\
        \textbf{Claim $1$:} We claim that the coefficients $b_j$ are zero.
        
        We prove that $b_M=0$, and thus inductively, all the $b_i=0$. To show this, we use the explicit expression 
of $\mathscr{V}$ and $\mathscr{Q}$ found in 
\cref{prop explicit formula for Lambda}. The leading monomial of $\mathscr{V}_{\rho_M,\tau_M}\mathscr{Q}_{\mu_M,\nu_M}$ is given up to a non-zero scalar multiple, by 
$$
b_M\ (\mathcal{S}_{p,1}\mathcal{S}_{p,2})^{\rho_M}\mathcal{S}_{p,1}^{\tau_M}\ \mathcal{T}_{p,1}^{-\nu_M-1}\ \mathcal{T}_{p,2}^{-\nu_M-\mu_M-1}
$$
regardless of the parity of $\mu_M$. This can be seen directly from \Cref{prop explicit formula for Lambda}. But the claim that $b_M=0$ now follows easily by considering the tuples mentioned above: If $-\nu_{M-1}<-\nu_{M}$ we are done by considering the exponent of $\mathcal{T}_{p,1}$. Assume we have equality, then if $-\mu_{M-1}<-\mu_M$ then we are done by considering the exponent of $\mathcal{T}_{p,2}$. We argue like this two more times with $\rho_M$ and $\tau_M$. Since the values of $\nu_j,\mu_j,\rho_j,\tau_j$ completely determine $f_{\mu_j,\nu_j}^{\rho_j,\tau_j}$, it cannot be the case that two different $j$ give rise to the same tuples $(\nu_j,\mu_j,\rho_j,\tau_j)$. Thus, \textbf{Claim $1$} holds by the exponent comparison.

Now, \eqref{eq:19} and \textbf{Claim $1$}, imply that 
 $$T:=\sum_{i=1}^N a_i\mathscr{V}_{r_i,t_i}\mathscr{P}_{m_i}=0.$$
 We now use this to deduce the following.\\
 \textbf{Claim $2$:} Our second and last claim is that the coefficients $a_i$ are also zero.
 
 We do this by showing that $a_N=0$ and then proceeding inductively as before. Recall that
 $$\mathscr{P}_{m_i}=\frac{1}{p^{m_i}}\begin{cases}
        \mathcal{P}_{p,1}^{(m_i+1)}(1),\ &\mathrm{if}\ m_i\geq 0\\
        \mathcal{P}_{p,2}^{(|m_i|+1)}(1),\ &\mathrm{if}\ m_i<0
    \end{cases}$$
 Suppose first that $|m_N|\geq 2$. Looking at the explicit expressions for $\mathcal{P}_{p,1}^{(n)}(1)$ and $\mathcal{P}_{p,2}^{(n)}(1)$ it follows that $a_N\mathscr{V}_{r_N,t_N}\mathscr{P}_{m_N}$ is the only term of $T$ that contains a single monomial which is a non-zero scalar multiple of
 $$a_N\cdot\begin{dcases}(\mathcal{S}_{p,1}\mathcal{S}_{p,2})^{r_N}\mathcal{S}_{p,1}^{t_N}\mathcal{T}_{p,1}^{m_N},\ &m_N\geq 0\\
(\mathcal{S}_{p,1}\mathcal{S}_{p,2})^{r_N}\mathcal{S}_{p,1}^{t_N}\mathcal{T}_{p,2}^{|m_N|},\ &m_N< 0
 \end{dcases}$$
 Thus considering inequalities between exponents as before, it follows that $a_N=0$. Now suppose that $|m_N|=1$. If $0=|m_{N-1}|<1$ we are clearly done by considering once again exponents of $\mathcal{T}_{p,1}$ or $\mathcal{T}_{p,2}$. If $|m_{N-1}|=1$, we let $K\in\{1,\dots, N-1\}$ be the smallest integer for which $$|m_K|=|m_{K+1}|=\dots=|m_{N-1}|=|m_N|=1.$$
 Then, we have
 \begin{align*}T&=\left(\sum_{i=1}^{K-1}a_i(\mathcal{S}_{p,1}\mathcal{S}_{p,2})^{r_i}\mathcal{S}_{p,1}^{t_i}\right)\mathscr{P}_0+\left(\sum_{i=K}^N a_i(\mathcal{S}_{p,1}\mathcal{S}_{p,2})^{r_i}\mathcal{S}_{p,1}^{t_i}\mathscr{P}_{m_i}\right),\ \text{where}\ m_i=\pm 1\\
 &=\left(\sum_{i=1}^{K-1}a_i(\mathcal{S}_{p,1}\mathcal{S}_{p,2})^{r_i}\mathcal{S}_{p,1}^{t_i}\right)\mathscr{P}_0+\left(\sum_{i\in I_1}a_i(\mathcal{S}_{p,1}\mathcal{S}_{p,2})^{r_i}\mathcal{S}_{p,1}^{t_i}\right)\mathscr{P}_1+ \left(\sum_{i\in I_{-1}}a_i(\mathcal{S}_{p,1}\mathcal{S}_{p,2})^{r_i}\mathcal{S}_{p,1}^{t_i}\right)\mathscr{P}_{-1}.\end{align*}
 where $I_1\cup I_{-1}=\{K,\dots,N\}$, $i\in I_1$ if and only if $m_1=1$ and $i\in I_{-1}$ if and only if $m_i=-1$.
 We write 
 $$C_1:=\sum_{i\in I_1}a_i(\mathcal{S}_{p,1}\mathcal{S}_{p,2})^{r_i}\mathcal{S}_{p,1}^{t_i}\ \text{ and }\ C_{-1}:=\sum_{i\in I_{-1}}a_i(\mathcal{S}_{p,1}\mathcal{S}_{p,2})^{r_i}\mathcal{S}_{p,1}^{t_i}.$$
 Recall that 
 \begin{align*}
     \mathscr{P}_0=1-\mathcal{S}_{p,1}\mathcal{S}_{p,2},\ 
     \mathscr{P}_1=-\frac{1}{p}\mathcal{S}_{p,1}\mathcal{T}_{p,2}+\frac{1}{p}\mathcal{T}_{p,1},\
     \mathscr{P}_{-1}=\frac{1}{p}\mathcal{S}_{p,2}\mathcal{T}_{p,1}+\frac{1}{p}\mathcal{T}_{p,2}.
 \end{align*}
 Since $T=0$, we must have
 $$-\frac{C_1}{p}\mathcal{S}_{p,1}\mathcal{T}_{p,2}+\frac{C_1}{p}\mathcal{T}_{p,1}-\frac{C_{-1}}{p}\mathcal{S}_{p,2}\mathcal{T}_{p,1}+\frac{C_{-1}}{p}\mathcal{T}_{p,2}=0.$$
 From this, we deduce that
 $$C_1=C_{-1}\mathcal{S}_{p,2}\ \text{ and }\ C_{-1}=C_1\mathcal{S}_{p,1}.$$
 It follows at once that $C_1=C_{-1}=0.$ Since for $i\in I_1$, $f_{m_i,0}^{r_i,t_i}=f_{1,0}^{r_i,t_i}$ is completely determined by $(r_i,t_i)$ we conclude that $a_i=0$ for every $i\in I_1$ and in an identical fashion $a_i=0$ for every $i\in I_{-1}$. This deals with the case where $|m_N|=1$. Finally, if $m_N=0$, then 
 $$T=\left(\sum_{i=1}^{N}a_i(\mathcal{S}_{p,1}\mathcal{S}_{p,2})^{r_i}\mathcal{S}_{p,1}^{t_i}\right)\mathscr{P}_0=0$$
 and $f_{m_i,0}^{r_i,t_i}=f_{0,0}^{r_i,t_i}$ is once again determined by $(r_i,t_i)$. Thus, \textbf{Claim $2$} follows.
 
 Putting everything together, we have shown that $\mathscr{Q}_f=0$ implies that $f=0$ and as a result, we deduce that
 $$\bigcap_{\Pi_p} \mathrm{ker}(\Lambda_{\Pi_p})=0.$$
 This is exactly what we wanted to prove in order to establish the cyclicity statement. For freeness, note that $\Lambda_{\Pi_p}:C_c^\infty(P_p\backslash G_p/G_p^\circ)\rightarrow(\Pi_p^\vee)^{G_p^\circ}$, for $\omega_{\pi_{p,1}}\neq \omega_{\pi_{p,2}}^{-1}$, can be extended naturally to a Hecke equivariant map  $$C_c^\infty(P_p\backslash G_p/G_p^\circ)[\tfrac{1}{1-\mathcal{S}_p}]\longrightarrow (\Pi_p^\vee)^{G_p^\circ},\ \tfrac{f}{(1-\mathcal{S}_p)^r}\mapsto\tfrac{\Lambda_{\Pi_p}(f)}{\Theta_{\Pi_p}(1-\mathcal{S}_p^{-1})^r}.$$
 We just proved that every such element $\tfrac{f}{(1-\mathcal{S}_p)^r}$ is given by $\tfrac{-\mathcal{S}_p\mathcal{Q}_f^{'}}{(1-\mathcal{S}_p)^{r+1}}\cdot f_0.$ Thus freeness follows from Hecke equivariance and the fact that this slightly shrinked family of representations $\Pi_p$ is still dense in $\mathrm{Spec}(\mathcal{H}_{G_p}^\circ)$.
    \end{proof}
\end{thm}
    \bibliography{citation} 
\bibliographystyle{amsalpha}

\noindent\textit{Mathematics Institute, Zeeman Building, University of Warwick, Coventry CV4 7AL,
England}.\\
\textit{Email address}: Alexandros.Groutides@warwick.ac.uk
\end{document}